\documentclass[reqno, 11pt]{amsart}
\usepackage{amssymb}
\usepackage{graphicx}
\usepackage{amsmath}
\usepackage[margin=1.23in]{geometry}
\usepackage{mathrsfs} 

\usepackage{titlesec}
\usepackage{esint}
\usepackage{cite}
\usepackage{todonotes}
\usepackage[nobysame]{amsrefs}
\usepackage[title]{appendix}

\titleformat{\section}{\normalfont\bfseries\color{black}\centering}{\thesection}{0.5em}{}

\titleformat{\subsection}{\normalfont\bfseries\color{black}}{\thesubsection}{0.5em}{}

\titleformat{\subsubsection}{\normalfont\bfseries\color{black}}{\thesubsubsection}{0.5em}{}

\usepackage{verbatim}
\usepackage{mathrsfs}
\usepackage{bm}
\usepackage{cite}
\usepackage{todonotes}
\usepackage{yhmath}
\allowdisplaybreaks[4]

\numberwithin{equation}{section}

\newtheorem{theorem}{Theorem}[section]
\newtheorem{corollary}[theorem]{Corollary}
\newtheorem{lemma}[theorem]{Lemma}
\newtheorem{prop}[theorem]{Proposition}

\theoremstyle{definition}
\newtheorem{remark}[theorem]{Remark}

\theoremstyle{definition}

\theoremstyle{definition}

\makeatletter
\def\dashint{\operatorname%
{\, \, \text{\bf-}\kern-.80em\DOTSI\intop\ilimits@\!\!}}
\makeatother

\newcommand\sA{\mathscr{A}}

\def\\det{\text{\det}}

\def\.5{\frac{1}{2}}

\def\bR{\mathbb{R}}

\def\bS{\mathbb{S}}

\def\cD{\mathcal{D}}

\newcommand{\RN}[1]{%
\textup{\uppercase\expandafter{\romannumeral#1}}%
}

\renewcommand{\epsilon}{\varepsilon}
\newcommand{\osc}{\operatorname{osc}}
\newcommand{\dv}{\operatorname{div}}
\newcounter{marnote}

\newcommand\dashnorm[2]{\nparallel\kern-.2em #1 \Vert_{#2}}

\begin{document}

\title[Conductivity problem with imperfect L-C interfaces]{Optimal gradient estimates for conductivity problems with imperfect low-conductivity interfaces}

\author[H. Dong]{Hongjie Dong}
\address[H. Dong]{Division of Applied Mathematics, Brown University, 182 George Street, Providence, RI 02912, USA}
\email{Hongjie\_Dong@brown.edu}
\thanks{H. Dong was partially supported by the NSF under agreement DMS-2350129.}

\author[H. Li]{Haigang Li}
\address[H. Li]{School of Mathematical Sciences, Beijing Normal University, Beijing 100875, China.}
\email{hgli@bnu.edu.cn}
\thanks{H. Li was partially Supported by Beijing Natural Science Foundation (No. 1242006), the Fundamental
Research Funds for the Central Universities (No. 2233200015), and National Natural Science Foundation of China (No. 12471191).}

\author[Y. Zhao]{Yan Zhao}
\address[Y. Zhao]{School of Mathematical Sciences, Beijing Normal University, Beijing 100875, China.}
\email{zhaoyan\_9926@mail.bnu.edu.cn}



\begin{abstract}
This paper studies field concentration between two nearly touching conductors separated by imperfect low-conductivity interfaces, modeled by Robin boundary conditions. It is known that for any sufficiently small interfacial bonding parameter $\gamma > 0$, the gradient remains uniformly bounded with respect to the separation distance $\varepsilon$. In contrast, for the perfect bonding case ($\gamma = 0$, corresponding to the perfect conductivity problem), the gradient may blow up as $\varepsilon \to 0$ at a rate depending on the dimension. In this work, we establish optimal pointwise gradient estimates that explicitly depend on both $\gamma$ and $\varepsilon$ in the regime where these parameters are small. These estimates provide a unified framework that encompasses both the previously known bounded case ($\gamma > 0$) and the singular blow-up scenario ($\gamma = 0$), thus furnishing a complete and continuous characterization of the gradient behavior throughout the transition in $\gamma$. The key technical achievement is the derivation of new regularity results for elliptic equations as $\gamma\to0$, along with a case dichotomy based on the relative sizes of $\gamma$ and a distance function $\delta(x')$. Our results hold for strictly relatively convex conductors in all dimensions $n \geq 2$.
\end{abstract}

\maketitle

\section{Introduction and main results}

\subsection{ Introduction}

When two inclusions in a composite material approach each other while exhibiting material properties in high contrast to those of the surrounding matrix, the narrow inter-inclusion region may develop significant electric field or mechanical stress concentrations, as documented in the seminal works of Budiansky and Carrier \cite{BC}, Keller \cites{Keller,Keller2}, Markenscoff \cite{M}, and Babu\u{s}ka et.al. \cite{BASL}. Such extreme field concentrations can precipitate material failure and substantially alter the effective properties of the composite medium \cite{FlaKel}. The analysis of this singular behavior is not only important in composite material science but also represents a fundamental challenge in the theory of partial differential equations, owing to the mathematical difficulties introduced by vanishing inter-inclusion distances; see Li and Vogelius \cite{LiVe} and Li and Nirenberg \cite{LiNe}. Over the past three decades, substantial advances have been made in the quantitative analysis of field concentration for perfectly bonded inclusions, yielding optimal estimates for local fields in narrow regions and rigorous asymptotic characterizations of field enhancement.

An alternative model featuring non-ideal bonding conditions is of considerable practical significance, as it accounts for interfacial contact resistance resulting from surface roughness and possible debonding, while also providing approximations for membrane structures in biological systems. Motivated by influential physics studies of Torquato and Rintoul \cite{TorRin} and recent pioneering mathematical work of Fukushima, Ji, Kang, and Li \cite{FJKL}, this paper presents an investigation of such imperfect interface conditions. The profound influence of interfacial effects on effective material properties is well-established in various systems; see Lipton and Vernescu \cite{LipVer}, Hashin \cite{Hash}, and Benveniste \cite{Benv}. In particular, \cites{TorRin,LipVer} established rigorous bounds for the effective conductivity of random composites containing equisized spherical inclusions with imperfect interfaces through sophisticated constructions of trial fields and applications of minimum energy principles. Nevertheless, these studies did not provide precise estimates for local field behavior or characterize field concentration in the nearly-touching regime, as explicitly noted in \cite{TorRin}: ``the complexity of the microstructure prohibits one from obtaining the local fields exactly." The present work addresses this fundamental problem by establishing optimal estimates for field concentration within narrow regions between inclusions in the regime of small bonding parameter ($\gamma \ll 1$), thereby bridging a crucial gap in the existing literature.

We now describe the mathematical framework. Since our primary interest lies in the narrow region between two nearly touching inclusions, we formulate the problem with exactly two inclusions. Let $n \geq 2$, and let $\Omega \subset \mathbb{R}^n$ be a bounded open set with $C^{2}$ boundary, containing two strictly relatively convex inclusions $D_1$ and $D_2$. The boundaries $\partial{D}_{1}$ and $\partial{D}_{2}$ are of class $C^{2,\alpha}$, and their principal curvatures are bounded below by a positive constant $\kappa_0$. Denote the matrix domain exterior to the inclusions by $\tilde{\Omega}= \Omega \setminus\overline{D_{1}\cup {D}_{2}}$. The distance between the inclusions 
$$\varepsilon := \operatorname{dist}(D_1, D_2)\ll 1$$ is small, while their distance to the outer boundary $\partial\Omega$ remains of order one. If $\varepsilon$ is bounded away from zero, estimates for the gradient and derivatives follow from standard elliptic theory.

When the bonding between the inclusions and the matrix is perfect, with conductivities $k$ inside $D_{1}\cup D_{2}$ and $1$ in the matrix $\tilde{\Omega}$, the electric potential $u$ satisfies the classical transmission problem:
\begin{equation}\label{general_problem}
\left\{
\begin{aligned}
\mathrm{div}\left(a_{k}(x)\nabla{u}\right)&=0 &\mbox{in}&~\Omega,\\
u&=\varphi(x) &\mbox{on}&~\partial\Omega,
\end{aligned}
\right.
\end{equation}
where the piecewise constant coefficient $a_k(x) = k \chi_{D_1 \cup D_2} + \chi_{\widetilde{\Omega}}$ represents the conductivity distribution. A fundamental feature of perfect bonding is the continuity of both the potential and the flux across the interfaces:
\begin{equation}\label{transmission}
u |_+ = u|_- \quad \mbox{and}\quad \left.\frac{\partial{u}}{\partial\nu} \right|_+ = k \left.\frac{\partial{u}}{\partial\nu} \right|_- \quad \mbox{on}~\partial D_1\cup \partial D_{2}.
\end{equation}
Here and throughout the paper, the subscripts $\pm$ denote limits taken from the outside and inside of the inclusion, respectively, and $\nu$ is the unit inward normal vector to $D_j$ on $\partial D_j$, $j=1,2$. 

When the contrast $k$ is bounded away from $0$ and $\infty$, it was proved that the gradient of the solution remains bounded uniformly in $\varepsilon$, see, for instance, Bonnetier and Vogelius \cites{BoVe}, Li and Vogelius \cite{LiVe} and Li and Nirenberg \cite{LiNe}. Furthermore, higher-order derivative estimates for solutions to \eqref{general_problem} with piecewise coefficients and closely spaced inclusions have been derived by Dong and Zhang in \cite{dz}, Dong and Li \cite{dl}, Ji and Kang \cite{JK}, Dong and Yang \cite{DY}, Kim \cite{Kim}, and Dong and Xu \cite{dx2022}.  

However, if $k$ degenerates to either $\infty$ or $0$, the gradient $\nabla u$ may blow up as $\varepsilon \to 0$. Specifically, in the case of perfect conductors ($k = \infty$), problem \eqref{general_problem} reduces to
\begin{equation}\label{perfect}
\left\{ \begin{aligned}
\Delta u &=0&\mbox{in}&~\widetilde\Omega,\\
u&= K_j &\mbox{on}&~\partial D_j, j = 1,2,\\
u&=\varphi(x)&\mbox{on}&~\partial\Omega,
\end{aligned} \right.
\end{equation}
where the constants $K_j$ are uniquely determined by the flux conditions 
$\int_{\partial D_j} \partial_\nu u \, dS = 0$, $j=1,2.$  The optimal blow-up rates of $\nabla u$ for this limiting problem are now well-established: \begin{itemize} 
\item[---] $\varepsilon^{-1/2}$ in two dimensions, by Ammari, Kang, and collaborators \cites{AKL,AKLLL,akllz} and Yun \cite{y1};
\item[---] $|\varepsilon \log \varepsilon|^{-1}$ in three dimensions, by Bao, Li, and Yin \cites{BLY, BLY2} and Lim and Yun \cite{limyun};
\item[---] $\varepsilon^{-1}$ in dimensions four and higher, by Bao, Li, and Yin \cite{BLY}.
\end{itemize} 
Subsequent works by Kang, Lim, and Yun \cites{kly, kly2} and Li et al. \cites{lwx, lly, Li} have further refined these results by establishing precise asymptotics that capture the singular behavior of $\nabla u$ near the origin.

In the insulated case ($k = 0$), problem \eqref{general_problem} becomes
\begin{equation*}
\left\{ \begin{aligned}
\Delta u &=0&\mbox{in}&~\widetilde\Omega,\\
\partial_\nu u&=0&\mbox{on}&~\partial D_j, j = 1,2,\\
u&=\varphi(x)&\mbox{on}&~\partial\Omega.
\end{aligned} \right.
\end{equation*}
While the optimal blow-up rate for $\nabla u$ in two dimensions was established about two decades ago through a dual argument \cites{AKL,AKLLL},  the higher-dimensional case has been resolved only recently in Dong, Li, and Yang \cite{DLY} under an asymptotically radially symmetry condition. See also earlier work by Li and Yang \cite{LY1} and Weinkove \cite{Wen1}. Subsequent work \cites{DLY2, LiZhao} extended this result to strictly relatively convex insulators. The optimal blow-up rate in higher dimensions is given by $\varepsilon^{-\frac{1}{2}+\alpha(\lambda_{1})}$, where $\alpha(\lambda_{1})$ is related to the first nonzero eigenvalue $\lambda_{1}$ of an elliptic operator on $\bS^{n-2}$, which in turn is determined by the principal curvatures of the inclusions. Thus, the study of field concentration phenomena---particularly the optimal blow-up rates---is relatively complete for the case of perfect bonding. For a comprehensive review, we refer to Kang's ICM survey \cite{Kang}.

When the bonding is non-ideal, at least one of transmission conditions in \eqref{transmission} fails to hold. We consider imperfect bonding interfaces of low-conductivity type (LC-type), which arises as the limiting case of each inclusion surrounded by a low-conductivity shell as its thickness tends to zero. In this limit, the transmission conditions become
$$
\left.\frac{\partial{u}}{\partial\nu} \right|_+ = k \left.\frac{\partial{u}}{\partial\nu} \right|_- =- \gamma^{-1} (u |_+ - u|_- ),
$$
where $\gamma > 0$ is the bonding parameter; we then say the inclusion has an LC-type imperfect bonding interface. Under such conditions, the continuity of the potential is no longer guaranteed. Another type of imperfect bonding, of high-conductivity type (HC-type), violates the continuity of flux. In two dimensions, the HC-type problem is dual to the LC-type problem. For detailed derivations, we refer to \cite{FJKL} and Benveniste and Miloh \cite{BenMil}.

In this paper, we assume that $D_1$ and $D_2$ possess LC-type imperfect bonding interfaces and take $k \to \infty$. The conductivity problem then reduces to the following Robin-type boundary value problem:
\begin{equation}\label{lcctype}
\left\{ \begin{aligned}
&\Delta u =0\quad  &\mbox{in}&~\widetilde\Omega,\\
&u + \gamma \partial_\nu u = K_j \quad &\mbox{on}&~\partial D_j, \,j = 1,2,\\
&\int_{\partial D_j} \partial_\nu u \, dS= 0\quad & j&=1,2,\\
&u=\varphi(x)\quad &\mbox{on}&~\partial\Omega,
\end{aligned} \right.
\end{equation}
where $\varphi\in C^{2}(\partial \Omega)$, $\gamma>0$, $\nu$ denotes the inward unit normal to $D_j$ on $\partial D_j$, and the constant $ K_j$ is uniquely determined by the the third line of \eqref{lcctype}. This configuration models suspensions of metallic particles in epoxy matrices with Kapitza-type interfacial resistance, as studied in \cites{TorRin, ArRo}. The solution $u \in H^{1}(\widetilde\Omega)$ to problem \eqref{lcctype} is indeed the unique minimizer of the functional
\begin{equation*}
I_{\gamma}[u] = \min_{v \in \sA}I[v],\qquad~\,\,\,\,~~I_{\gamma}[v]:= \int_{\widetilde\Omega} |\nabla v|^2+\frac{1}{\gamma}\sum_{i=1}^{2}\int_{\partial D_i} |v-(v)_{\partial D_i}|^2,
\end{equation*}
where $ (v)_{\partial D_i}:= \fint_{\partial D_i} v\, dS,\ i=1,2$, over the admissible set
$$
\sA := \{ v \in H^{1}(\widetilde\Omega):\, v = \varphi~~\mbox{on}~~\partial\Omega \}.$$
A proof of this variational characterization can be found in Lemma 2.2 of Dong, Yang, and Zhu \cite{DYZ}.

Consequently, the study of conductivity problems with imperfect bonding interfaces---especially those of low-conductivity type---is of fundamental importance in both applied mathematics and materials science. Such interfaces commonly arise in composite materials, where non-ideal bonding between phases creates thin interfacial regions with markedly reduced conductivity. These imperfect contacts significantly influence macroscopic material properties, including effective electrical and thermal conductivity, mechanical stress distribution, and fracture behavior.

From a mathematical standpoint, these problems are characterized by singularly perturbed Robin-type boundary conditions involving a small parameter $\gamma$ that governs the interface conductivity. The central challenge is to characterize the blow-up behavior of the electric field (the gradient of the solution) as the inter-inclusion distance $\varepsilon\to 0$ and the interface parameter $\gamma\to 0$. The present work provides a complete and sharp quantitative description of this regime, thereby bridging the gap to the well-understood perfect conductor case. Moreover, the mathematical framework developed here applies directly to analogous physical phenomena, such as anti-plane elasticity and ideal flow past obstacles, rendering the results broadly impactful across disciplines.

As mentioned above, when $\gamma=0$, problem \eqref{lcctype} reduces to \eqref{perfect}, which has been extensively studied. While, for $\gamma > 0$, a striking observation from \cite{FJKL} is that in two dimensions for circular inclusions, any fixed $\gamma > 0$ can completely suppress gradient blow-up for certain boundary data. This contrasts sharply with the well-known result \cite{AKL} that for  $\gamma = 0$ (the perfect conductivity), $|\nabla u|$ blows up as $\varepsilon \to 0$. This demonstrates that even a thin, low-conductivity coating can effectively inhibit gradient singularity formation. Motivated by the biological principle that living systems cannot tolerate excessive stress, the authors of \cite{FJKL} conjectured that this ``stress shielding" phenomenon should persist for arbitrary inclusion shapes and in all dimensions, provided $\gamma$ is bounded away from zero. In \cite{DYZ}, this conjecture was confirmed under certain symmetry conditions or when $\gamma$ is sufficiently small. However, for large $\gamma$, it was demonstrated in \cite{DYZ} that gradient blow-up may occur. 
The present paper further resolves this conjecture by establishing sharp quantitative estimates that simultaneously account for both small $\varepsilon$ and small $\gamma$.

\subsection{Main results}

To state our results precisely, we begin by parametrizing the boundaries of the inclusions and introducing necessary notation. Working in coordinates $x=(x',x_n)$, we position the inclusions such that $D_1=D_1^*+(0',\varepsilon)$ and $D_2=D_2^*$, where $D_1^*$ and $D_2^*$ are two open sets touching at the origin with the inner normal of $\partial D_1^*$ aligned with the positive $x_n-$axis. Near the origin, the boundaries $\partial D_1$ and $\partial D_2$ are described by
\begin{equation}\label{fg0}
x_n=\varepsilon+f_{1}(x'),\quad x_n=f_{2}(x'),\quad |x'|\leq 2R
\end{equation}
where $f_{1},f_{2}\in C^{2,\alpha}$ satisfy
\begin{equation}\label{fg}
f_{1}(0')=f_{2}(0')=0,\quad D_{x'}f_{1}(0')=D_{x'}f_{2}(0')=\vec{0}, \quad D^2(f_{1}-f_{2})(0')\geq\kappa>0.
\end{equation}

Define the separation function and a reference function by
$$\delta(x'):=\varepsilon+f_{1}(x')-f_{2}(x') \quad\mbox{and}\quad\eta(x'):=\varepsilon+|x'|^2.$$
By \eqref{fg0}--\eqref{fg} and Taylor's theorem, there exists $\kappa>0$ such that
\begin{equation*}
\kappa |x'|^2\leq f_{1}(x')-f_{2}(x')\leq \frac{1}{\kappa} |x'|^2,\quad \mbox{for} ~0\leq |x'|\leq 2R\,\, \mbox{and some}~ \kappa>0,
\end{equation*}
which implies $\eta(x')\sim \delta(x')$. Here, $A\sim B$ denotes the relation $\frac{1}{C}A<B<CA$ for a constant $C$ depending only on $n$, $R$, $\alpha$, $\kappa$, and the upper bound of $\|f_{1}\|_{C^{2,\alpha}}$ and $\|f_{2}\|_{C^{2,\alpha}}$. For $|x'_{0}|\leq\,R$, we introduce the narrow region centered at $x_0'$, 
\begin{equation}\label{definition_omegat}
\Omega_t(x_0):=\{|x'-x_0'|\leq t,~ f_{2}(x')\leq x_n\leq \varepsilon+f_{1}(x')\},
\end{equation}
with top and bottom boundaries
$$\Gamma_{1,t}(x_0):= \{|x'-x_0'|\leq t, x_n=\varepsilon+f_{1}(x')\}\quad\mbox{and}~ \Gamma_{2,t}(x_0):=\{|x'-x_0'|\leq t, x_n=f_{2}(x')\}.$$ We write $\Omega_t:=\Omega_t(0)$ and $\Gamma_{i,t} :=\Gamma_{i,t} (0')$, $i=1,2,$ when centered at the origin. 

Recall the Frobenius norm for the Hessian of $f_1 - f_2$ is given by
$$
|D^2(f_1 - f_2)(x')| = \left( \sum_{i,j=1}^{n-1} |\partial_{ij}(f_1 - f_2)(x')|^2 \right)^{1/2}.
$$
We define the uniform bound of this norm over the ball $\mathtt{B}_R := \{ |x'| < R \}$ as
$\|D^2(f_1 - f_2)\|_{L^{\infty}(\mathtt{B}_R)} = \sup_{x' \in \mathtt{B}_R} |D^2(f_1 - f_2)(x')|.$
Using this notation, we introduce the parameter $\gamma_0$ by
\begin{equation}\label{gamma00}
\gamma_0 := \frac{2} {(n+1) \|D^2(f_1 - f_2)\|_{L^{\infty}(\mathtt{B}_R)}} .
\end{equation}

The main results of this paper are the following.

\begin{theorem}\label{upperbound}
Let $u$ be the solution of \eqref{lcctype} in dimension $n \geq 2$ with boundaries $\partial D_1$ and $\partial D_2$ satisfying \eqref{fg0}--\eqref{fg}. Let $\gamma_0$ be the constant specified by \eqref{gamma00}. Then for $\varepsilon \in (0, \frac{1}{4})$ and $\gamma \in (0, \gamma_0)$, we have for $x \in \Omega_{R/2}$,
\begin{equation*}
|D u(x)|\leq
\left\{ \begin{aligned}
&\frac{C}{\sqrt{\gamma+\varepsilon+|x'|^2}}\|\varphi\|_{C^{2}(\partial\Omega)}& \mbox{for}~n=2,\\
&\frac{C}{ {(\gamma+\varepsilon+|x'|^2)}|\ln(\varepsilon+\gamma)|}\|\varphi\|_{C^{2}(\partial\Omega)}& \mbox{for}~n=3,\\
&\frac{C}{{\gamma+\varepsilon+|x'|^2}}\|\varphi\|_{C^{2}(\partial\Omega)}& \mbox{for}~n\geq 4,
\end{aligned} \right.
\end{equation*}
and
$$\|D u\|_{L^{\infty}(\tilde{\Omega}\backslash\Omega_{R/2})}\leq C\|\varphi\|_{C^{2}(\partial\Omega)}\quad\mbox{for}~ n\geq2,$$
where the constant $C>0$ depends only on $n$, $R$, $\alpha$, $\kappa$, a lower bound of $\gamma_0-\gamma$ and upper bounds of $\|f_{1}\|_{C^{2,\alpha}}$ and $\|f_{2}\|_{C^{2,\alpha}}$.
\end{theorem}

This result demonstrates that the gradient in the narrow region exhibits a dimension-dependent blow-up rate, characterized explicitly by the bonding parameter $\gamma$ and the gap $\varepsilon$, with validity over the precisely quantified range $\gamma \in (0, \gamma_0)$. Moreover, the blow-up rate is shown to recover that of the perfect conductivity problem \eqref{perfect} in the limit as $\gamma \to 0$.

\begin{corollary}
Under the hypotheses of Theorem \ref{upperbound}, for $\varepsilon\in(0,\frac{1}{4})$ and $\gamma\in(0,\gamma_0)$, the following global estimates hold
\begin{equation*}
\|D u\|_{L^{\infty}(\tilde{\Omega})}\leq
\left\{ \begin{aligned}
&\frac{C}{\sqrt{\varepsilon+\gamma}}\|\varphi\|_{C^{2}(\partial \Omega)}&\mbox{for}~n=2,\\
&\frac{C}{(\varepsilon+\gamma)|\ln(\varepsilon+\gamma)|}\|\varphi\|_{C^{2}(\partial \Omega)}&\mbox{for}~n=3,\\
&\frac{C}{\varepsilon+\gamma}\|\varphi\|_{C^{2}(\partial \Omega)}&\mbox{for}~n\geq 4,
\end{aligned} \right.
\end{equation*}
with the same dependence of the constants as above.
\end{corollary}


These upper bounds are shown to be optimal via the following lower-bound example.

\begin{theorem}\label{lowerbound}
Let $n \geq 2$. Let $\Omega$ be the ball of radius $6$ centered at $(0', \frac{\varepsilon}{2})$, and let $D_1$, $D_2$ be the unit balls centered at $(0', 1 + \varepsilon)$ and $(0', -1)$, respectively. Let $u \in H^1(\widetilde{\Omega})$ be the solution to \eqref{lcctype} with boundary data $\varphi = x_n - \frac{\varepsilon}{2}$. Then there exist sufficiently small positive constants $\bar{\varepsilon}$ and $\bar{\gamma}$, depending only on $n$, such that for all $\varepsilon \in (0, \bar{\varepsilon})$ and $\gamma \in (0, \bar{\gamma})$,
\begin{equation*}
   \|D u\|_{L^{\infty}(\tilde{\Omega})}\geq
\left\{ \begin{aligned}
&\frac{C}{\sqrt{\varepsilon+\gamma}} &\mbox{for}~n=2,\\
&\frac{C}{(\varepsilon+\gamma)|\ln(\varepsilon+\gamma)|} &\mbox{for}~n=3,\\
&\frac{C}{\varepsilon+\gamma} &\mbox{for}~n\geq 4,
\end{aligned} \right.
\end{equation*}
where $C$ is a constant depending only on $n$.
\end{theorem}
This lower bound confirms that the estimates in Theorem \ref{upperbound} are sharp---the gradient indeed blows up at the predicted rate as $\varepsilon+\gamma\to 0$.

The proof of Theorem \ref{upperbound} proceeds by flattening the boundary of the narrow region (see the transformation \eqref{variableschange} and the resulting equation \eqref{equflatten} below). For a point $x_0 \in \mathbb{R}^n$, denote by $B_r(x_0)$ the ball of radius $r$ centered at $x_0$. Let 
$$B_{r}^{+}:=B_{r}(0)\cap\{x_{n}>0\},\quad \Gamma_{r}^{0}:=B_{r}(0)\cap\{x_{n}=0\},\quad\Gamma^{+}_{r}:=\partial{B}_{r}(0)\cap\{x_{n}>0\}.$$
We need to consider the following boundary-value problem on the half unit ball
\begin{equation}\label{equgradientsmallgamma}
\left\{ \begin{aligned}
\partial_i(A^{ij}\partial_jU)&=\partial_iF^i+G& \mbox{in}&~ B_{1}^{+},\\
h\,U+\hat{\gamma} A^{ij}\partial_jU\nu_{i}&=\hat{\gamma} F^i\nu_{i}+\phi& \mbox{on}&~\Gamma_{1}^{0},
\end{aligned} \right.
\end{equation}
where $\nu=(0',-1)$ is the unit outward normal to $B_{1}^{+}$ on $\Gamma_{1}^{0}$, $\hat{\gamma}$ is a constant, $h\in C^{\alpha}(B_{1}^{+})$, and the coefficient matrix $A^{ij} \in C^{\alpha}(B_{1}^{+})$ is uniformly elliptic and bounded:
\begin{equation}\label{uniformelliptic}
\lambda|\xi|^2\leq A^{ij}(x)\xi_i\xi_j, \quad~ \forall\xi\in\bR^n,\quad |A^{ij}(x)|\le \Lambda,
\end{equation}
with positive constants $\lambda$ and $\Lambda$. Here and henceforth, we adopt the Einstein summation convention for repeated indices. The following result is of independent interest.

\begin{theorem}\label{gradientsmallgamma3}
Let $\lambda,\hat\gamma>0$ and $U\in H^{1}(B_{1}^{+})$ be a solution of \eqref{equgradientsmallgamma} with $A^{ij}\in C^{\alpha}(B_{1}^{+})$ satisfying \eqref{uniformelliptic} in $B_{1}^{+}$.
Suppose that $F^i\in C^{\alpha}(B_{1}^{+})$, $G\in L^{\infty}(B_{1}^{+})$, and $h$ satisfies 
$$h\in C^{\alpha}(B_{1}^{+}),\quad |h|>\lambda~ \mbox{on} ~B_{1}^{+},\quad~\mbox{and}\,\,~~~ h(x)\hat{\gamma}>0~\mbox{in}~ \Gamma_{1}^{0},$$
with ${\phi}/{h}\in C^{1,\alpha}(B_{1}^{+})$. Then 
\begin{equation}\label{thm1.4_equgradient}
 \|U\|_{C^{1,\alpha}(B_{1/4}^{+})}\leq C\left(\|U\|_{L^2(B_{3/4}^{+})}+\|F\|_{C^{\alpha}(B_{3/4}^{+})}+\|G\|_{L^{\infty}(B_{3/4}^{+})}+\|\frac{\phi}{h}\|_{C^{1,\alpha}(B_{3/4}^{+})}\right),
\end{equation}
where $C$ depends on $\lambda,\Lambda,n,\|A^{ij}\|_{C^{\alpha}(B_{1}^{+})}$, and the upper bound of $\|h\|_{C^{\alpha}(B_{1}^{+})}$, but is independent of $\hat{\gamma}$.
\end{theorem}

In the limiting case when $\hat{\gamma}=0$, the boundary condition in \eqref{equgradientsmallgamma} reduces to a Dirichlet condition, and the estimate \eqref{thm1.4_equgradient} is classical. A standard technique for analyzing elliptic equations with Robin boundary conditions is to transform them into equivalent problems with homogeneous Neumann conditions \cite{Lieberm}. For sufficiently large $\hat{\gamma}$, specifically when $\hat{\gamma} \geq \hat{\gamma}_0$ for some fixed $\hat{\gamma}_0 \gg 1$, the boundary condition approximates a conormal condition as $\hat \gamma \to \infty$, and classical theory yields a bound of order $O(1/\hat{\gamma}_0)$. However, the case of small $\hat{\gamma}$ ($0 < \hat{\gamma} < \hat{\gamma}_0$) has remained open. Theorem \ref{gradientsmallgamma3} establishes the uniform bound \eqref{thm1.4_equgradient}, which is independent of $\hat{\gamma}$, demonstrating that the gradient remains bounded even as $\hat{\gamma} \to 0^+$. This uniform estimate is essential for the analysis in narrow regions, as it provides $\hat{\gamma}$-independent control over the regularity of the solution near the Robin boundary. We further remark that  analogous estimates hold for higher order norms of the solution, given appropriate smoothness of the data.

The main technical contributions of this paper, which culminate in Theorem \ref{boundedgradientw}, are developed in Sections \ref{sec_2}--\ref{sec_4}. In Section \ref{sec_2}, we establish uniform H\"older estimates for solutions to a class of elliptic equations with Robin boundary conditions as the parameter $\hat{\gamma}$, which scales the normal derivative, tends to zero. Section \ref{sec_3} is devoted to investigating the regularity of solutions for another class of elliptic equations whose lower-order coefficients are of order $O(\frac{1}{{\gamma}})$, particularly for small ${\gamma}$. In Section \ref{sec_4}, we prove Theorem \ref{boundedgradientw} via a dichotomy based on the relative sizes of the bonding parameter $\gamma$ and a distance function $\delta(x')$. Specifically, for a sufficiently large constant $\mu>0$, we treat the two complementary cases $\gamma\le \mu\delta$ and  $\gamma>\mu\delta$ by applying the regularity theories from Sections \ref{sec_2} and \ref{sec_3}, respectively. The results from both cases are then synthesized to complete the proof. Having established these ingredients, in Section \ref{sec_5} we adapt the decomposition technique similar as in the perfect conductivity problem and construct a new and suitable auxiliary function to prove Theorem \ref{upperbound}.
The optimality of these bounds (Theorem \ref{lowerbound}) is established in Section \ref{sec_6} by constructing a subsolution under a specific symmetry condition. Finally, Appendix A demonstrates the convergence of solutions to the perfect conductivity problem as $\gamma \to 0$.

\section{Proof of Theorem \ref{gradientsmallgamma3}}\label{sec_2}

This section is devoted to the proof of Theorem \ref{gradientsmallgamma3}. We begin by normalizing the Robin boundary condition in \eqref{equgradientsmallgamma} so that the coefficient of the conormal derivative becomes unity, yielding
$$\frac{h}{\hat{\gamma}}u+A^{ij}\partial_{j}u\nu_{i}=F^{i}\nu_{i}+\frac{\phi}{\hat{\gamma}}.$$
Observe that as $\hat{\gamma} \to 0$, the term $|\phi|_{L^\infty(\Gamma_1^0)} / \hat{\gamma}$ may become unbounded. The key insight is to consider the function $u - \phi/h$ instead of $u$ itself, which effectively reduces the problem to the case $\phi = 0$ without loss of generality. The proof follows Campanato's method, proceeding in two stages: first, we establish higher-order derivative estimates for constant-coefficient equations in Subsection \ref{subsec2.1}; then, using a perturbation argument and an iterative technique, we complete the proof of Theorem \ref{gradientsmallgamma3} in Subsection \ref{subsec2.2}.

We now introduce notation used throughout this section. For $k \geq 1$, let $D_{x'}^k$ denote any $k$th-order tangential derivative with respect to $x' = (x_1, \dots, x_{n-1})$. For $x_0 \in \mathbb{R}^n$ and $r > 0$, define
\begin{equation*}
\begin{aligned}
B_{r}^{+}(x_0)&:=\{y=(y',y_n)\in \bR^n~\big|~y\in B_r(x_0),~y_n>0\}, \\
\Gamma_{r}^{0}(x_0)&:=\{y=(y',y_n)\in \bR^n ~\big|~y\in \partial B_{r}^{+}(x_0), |y-x_0|<r,~y_{n}=0\}, \\
\Gamma_{r}^{+}(x_0)&:=\{y=(y',y_n)\in \bR^n ~\big|~y\in \partial B_{r}^{+}(x_0), ~|y-x_0|=r,~y_{n}>0\}.
\end{aligned}
\end{equation*}
When $x_0 = 0$, we write $B_r$, $B_r^+$, $\Gamma_r^0$, and $\Gamma_r^+$. The average of a function $U$ over a domain $\mathcal{D}$ is denoted by $(U)_{\mathcal{D}}$.

\subsection{Constant coefficient equation}\label{subsec2.1}

\begin{lemma}\label{propconstantequ}
Let $U_1 \in H^1(B_{1}^{+})$ be a solution of the constant-coefficient elliptic problem
\begin{equation}\label{constantequ}
\left\{ \begin{aligned}
\partial_i(\bar{A}^{ij}\partial_j U_1)=\partial_{i}\bar{F}^i&\quad\mbox{in}~ B_{1}^{+}, \\
\bar{h}\,U_1-\hat{\gamma}\, \bar{A}^{nj}\partial_jU_1=-\hat{\gamma}\bar{F}^n&\quad\mbox{on}~\Gamma_{1}^{0},
\end{aligned} \right.
\end{equation}
where $\bar h$ and $\bar{F}^i$ are constants, and $\bar{A}^{ij}$ is a constant matrix satisfying the uniform ellipticity condition
\begin{equation}\label{uniformellipticity}
\lambda |\xi|^2\leq \bar{A}^{ij}\xi_i\xi_j\leq \Lambda |\xi|^2,\quad\forall~\xi\in \bR^n.
\end{equation}
If $\hat{\gamma}\bar{h}>0$, then for any $0 < \rho < \frac{1}{4}$, the following estimates hold:
\begin{align}
&\|U_1\|_{C^k(B_{1/2}^{+})}\leq C\|U_1\|_{L^2(B_{1}^{+})}+C\|\bar{F}\|_{L^2(B_{1}^{+})}, \quad\forall k\geq0,\label{constant1}\\ 
&\|DU_1-(DU_1)_{B_{\rho}^{+}}\|^2_{L^{2}(B_{\rho}^{+})}\leq C\rho^{n+2}\| DU_1-(DU_1)_{B_{1}^{+}}\|^2_{L^{2}(B_{1}^{+})}, \label{constant3}
\end{align}
where $C$ depends only on $\lambda$, $\Lambda$, $n$, and $k$.
\end{lemma}

\begin{proof}
Without loss of generality, we assume $\bar{h} = 1$; the general case follows by dividing the boundary condition by $\bar{h}$ and replacing $\hat \gamma$ with $\hat{\gamma} / \bar{h}$.

\textbf{Proof of \eqref{constant1}.} Multiply \eqref{constantequ} by $U_1\eta^2$ where $\eta \in C_c^\infty(B_1)$ satisfies $\eta \equiv 1$ in $B_{1/2}^+$ and $|\nabla \eta| \leq C$ in $B_1^+$. Integration by parts yields
\begin{equation*}
\int_{B_{1}^{+}}(\bar{A}^{ij}\partial_jU_1-\bar{F}^i)\partial_i(U_1\eta^2)\,dx+\frac{1}{\hat{\gamma}}\int_{\Gamma_{1}^{0}}U_1^2\eta^2 \,dx'=0.
\end{equation*}
Using ellipticity and Cauchy’s inequality, we obtain the energy estimate
\begin{equation}\label{constantl2}
\int_{B_{1/2}^{+}}|DU_1|^2 \,dx \leq C \int_{B_{1}^{+}}U_1^2\,dx+\int_{B_{1}^{+}}|\bar{F}|^2\,dx.
\end{equation}
Since the coefficients $\bar{A}^{ij}$ are constant, we may differentiate the equation tangentially. The tangential derivatives $D_{x'}^k U_1$ ($k \geq 1$) still satisfy the same equation and boundary conditions as in \eqref{constantequ}. By iterating estimate \eqref{constantl2} for these derivatives, we obtain
\begin{equation*}
\int_{B_{1/2}^{+}}|D D_{x'}^kU_1|^2 \,dx\leq C\int_{B_{1}^{+}}| U_1|^2\,dx+\int_{B_{1}^{+}}|\bar{F}|^2\,dx.
\end{equation*}
To estimate the normal derivatives, we use the equation to express $\partial_n^2 U_1$:
\begin{equation}\label{partialnn}
\bar{A}^{nn}\partial_n^2U_1=-\sum_{i+j<2n}\bar{A}^{ij}\partial_{ij}U_1,\quad \mbox{in}~B_{1}^{+}.
\end{equation}
Since $\bar{A}^{nn} \geq \lambda > 0$, it follows that
\begin{equation*}
\|\partial_n^2U_1\|_{L^2(B_{1/2}^{+})}\leq C\sum_{i+j<2n} \|\partial_{x_ix_j}U_1\|_{L^2(B_{1/2}^{+})}\leq C\| U_1\|_{L^2(B_{1}^{+})}+C\|\bar{F}\|_{L^2(B_{1}^{+})}.
\end{equation*} 
By induction, for all $k \geq 0$, 
$$\|D_{x'}^kU_1\|_{L^2(B_{1/2}^{+})}+\sum_{l=1}^{2}\|\partial_n^lD_{x'}^kU_1\|_{L^2(B_{1/2}^{+})}\leq C\|U_1\|_{L^2(B_{1}^{+})}+C\|\bar{F}\|_{L^2(B_{1}^{+})}.$$
By repeatedly differentiating the equation in $x_n$, we can estimate higher-order normal derivatives (e.g., $\partial_n^3 U_1$, $\partial_n^4 U_1$, etc.). Combining these estimates and using the Sobolev embedding $H^\ell(B_{1/2}^+) \hookrightarrow C^k(\overline{B_{1/2}^+})$ for $\ell > k + n/2$, we conclude the proof of \eqref{constant1}.

\textbf{Proof of \eqref{constant3}.}
Since $D_{x_i} U_1 \,(1\leq i\leq n-1)$ satisfies \eqref{constantequ} with $\bar{F}^i=0$, we apply \eqref{constant1} to $D_{x_i}U_1$ to obtain
\begin{equation*}
\|D_{x'}U_1\|_{C^1(B_{1/4}^{+})}\leq C\|D_{x'}U_1\|_{L^{2}(B_{1/2}^{+})}.
\end{equation*}
From \eqref{partialnn}, it follows that
\begin{equation}\label{pfconstant3}
\|D^2U_1\|_{L^{\infty}(B_{1/4}^{+})}\leq C\|DD_{x'}U_1\|_{L^{\infty}(B_{1/4}^{+})}\leq C\|D_{x'}U_1\|_{L^2(B_{1/2}^{+})}.
\end{equation}
Estimate \eqref{pfconstant3} can be generalized as follows:
\begin{equation}\label{pfconstantss4}
\|D^2U_1\|_{L^{\infty}(B_{\frac{1}{2}-\frac{1}{2^{k+2}}}^{+})}\leq C2^{(k+3)(1+\frac{n}{2})}\|DU_1\|_{L^{2}(B_{\frac{1}{2}-\frac{1}{2^{k+3}}}^{+})},\quad \mbox{for}~k\geq 0.
\end{equation}

Indeed, for $y\in \Gamma_1^0$ with $|y'|\leq \frac{1}{2}-\frac{1}{2^{k+2}}$, a scaling argument applied to $U_1$ in $B_{\frac{1}{2^{k+3}}}^{+}(y)$, along with estimate \eqref{pfconstant3}, yields the following $L^{\infty}$-bound on $D^2U_1$ over $B_{\frac{1}{2^{k+4}}}^{+}(y)$:
\begin{equation}\label{pfconstantss41}
\|D^2U_1\|_{L^{\infty}(B_{\frac{1}{2^{k+4}}}^{+}(y))}\leq C2^{(k+3)(1+\frac{n}{2})}\|DU_1\|_{L^{2}(B_{\frac{1}{2^{k+3}}}^{+}(y))}\leq\,C2^{(k+3)(1+\frac{n}{2})}\|DU_1\|_{L^{2}(B_{\frac{1}{2}-\frac{1}{2^{k+3}}}^{+})}.
\end{equation}
For $y\in B^{+}_{\frac{1}{2}-\frac{1}{2^{k+2}}}$ with $y_n>\frac{1}{2^{k+4}}$, standard interior estimates for elliptic equations give an $L^{\infty}$-estimate for $D^2U_1$ in $B_{\frac{1}{2^{k+5}}}(y)$:
\begin{equation}\label{pfconstantss42}
\|D^2U_1\|_{L^{\infty}(B_{\frac{1}{2^{k+5}}}(y))}\leq C2^{(k+4)(1+\frac{n}{2})}\|DU_1\|_{L^{2}(B_{\frac{1}{2^{k+4}}}(y))}\leq\,C2^{(k+4)(1+\frac{n}{2})}\|DU_1\|_{L^{2}(B_{\frac{1}{2}-\frac{1}{2^{k+3}}}^{+})}.
\end{equation}
Hence, \eqref{pfconstantss41} and \eqref{pfconstantss42} together imply \eqref{pfconstantss4}. 

From \eqref{pfconstantss4} and the interpolation inequality, we obtain, for $k\geq 0$,
\begin{align*}
&\|D^2U_1\|_{L^{\infty}(B_{\frac{1}{2}-\frac{1}{2^{k+2}}}^{+})} \\
&\leq\frac{1}{3^{2(1+\frac{n}{2})}}\|D^2U_1\|_{L^{2}(B_{\frac{1}{2}-\frac{1}{2^{k+3}}}^{+})}+C2^{2(k+3)(1+\frac{n}{2})}\|U_1-(U_1)_{B^{+}_{1/2}}\|_{L^2(B^{+}_{\frac{1}{2}-\frac{1}{2^{k+3}}})}.
\end{align*}
Multiplying both sides by $\frac{1}{3^{2(k+3)(1+\frac{n}{2})}}$ and summing over $k = 0, 1, 2, \dots$, we derive 
\begin{equation*}
\begin{aligned}
&\sum_{k=0}^{+\infty}\frac{1}{3^{2(k+3)(1+\frac{n}{2})}}\|D^2U_1\|_{L^{\infty}(B^{+}_{2^{-1}-2^{-(k+2)}})}\\
&\leq\sum_{k=0}^{+\infty}\frac{1}{3^{2(k+4)( 1+\frac{n}{2})}}\|D^2U_1\|_{L^{\infty}(B^{+}_{2^{-1}-2^{-(k+3)}})}+C\sum_{k=0}^{+\infty}(\frac{2}{3})^{(k+3)(n+2)}\|U_1-(U_1)_{B^+_{1/2}}\|_{L^2(B^+_{1/2})}.
\end{aligned}
\end{equation*}
Since $\|D^2U_1\|_{L^{\infty}(B^+_{1/2})}<+\infty$, the summation on the right-hand side can be absorbed into the left-hand side, which implies 
\begin{equation}\label{D2ULinfty1}
\|D^2U_1\|_{L^{\infty}(B^{+}_{1/4})}\leq C\|U_1-(U_1)_{B^+_{1/2}}\|_{L^2(B^+_{1/2})}.
\end{equation}
Using the interpolation inequality and \eqref{D2ULinfty1}, we have
\begin{equation}\label{ULinfty}
\|DU_1\|_{L^{\infty}(B^{+}_{1/4})}\leq C\|U_1-(U_1)_{B^+_{1/2}}\|_{L^2(B^+_{1/2})}.
\end{equation}
Applying \eqref{ULinfty} to $D_{x_i}U_1(1\leq i\leq n-1)$ and using \eqref{partialnn}, we conclude
\begin{equation}\label{D2ULinfty}
\|D^2U_1\|_{L^{\infty}(B^{+}_{1/4})}\leq C\|DD_{x'}U_1\|_{L^{\infty}(B^{+}_{1/4})}\leq C\|D_{x'}U_1-(D_{x'}U_1)_{B^+_{1/2}}\|_{L^2(B^+_{1/2})}.
\end{equation}
Estimate \eqref{constant3} follows directly from \eqref{D2ULinfty}.
\end{proof}

Applying a scaling argument to Lemma \ref{propconstantequ} and combining it with standard elliptic interior estimates yields the following corollary, which will be useful later.

\begin{corollary}\label{constantlemmaR}
 Let $U_1\in H^1(B_{1}^{+})$ satisfy the constant-coefficient problem 
\begin{equation*}
\left\{ \begin{aligned}
\partial_i(\bar{A}^{ij}\partial_j U_1)&=\partial_{i}\bar{F}^i &\mbox{in}&~ B_{1}^{+}, \\
\bar{h}\,U_1-\hat{\gamma}\, \bar{A}^{nj}\partial_jU_1&=-\hat{\gamma}\bar{F}^n &\mbox{on}&~\Gamma_{1}^{0},
\end{aligned} \right.
\end{equation*}
with $\bar{A}^{ij}$ satisfying the uniform ellipticity condition \eqref{uniformellipticity}. If $\hat{\gamma}\bar{h}>0$, then for $x_0\in B_{1/2}^{+}$ and $0<\rho<r\leq \frac{1}{2}$, it holds that
\begin{align*}
&\|D^kU_1\|_{L^{\infty}(B_{r/2}^{+})}\leq Cr^{-n/2-k}\Big(\|U_1\|_{L^2(B_{r}^{+})}+\|\bar{F}\|_{L^2(B_{r}^{+})}\Big),\quad\forall~k\geq0, \\
&\|DU_1-(DU_1)_{B_{\rho}^{+}(x_0)}\|^2_{L^{2}(B_{\rho}^{+}(x_0))}\leq C(\frac{\rho}{r})^{n+2}\| DU_1-(DU_1)_{B_{r}^{+}(x_0)}\|^2_{L^{2}(B_{r}^{+}(x_0))}, 
\end{align*} 
where $C > 0$ depends only on $\lambda$, $\Lambda$, and $n$.
\end{corollary}

\subsection{Proof of Theorem \ref{gradientsmallgamma3}}\label{subsec2.2}

\begin{proof}[Proof of Theorem \ref{gradientsmallgamma3}]

By considering the function $U-\frac{\phi}{h}$, we may assume without loss of generality that $\phi = 0$. Throughout the proof, we define the quantity
\begin{equation*}
\mathcal{Q}:=\|U\|_{L^2(B_{1/2}^{+})}+\|F\|_{C^{\alpha}(B_{1/2}^{+})}+\|G\|_{L^{\infty}(B_{1/2}^{+})}.
\end{equation*}

{\bf Step 1. Decomposition of $U$.}
Let $x_0 \in B^{+}_{1/4}$ and $r\in(0,r_0)$, where $r_0\in(0,\frac{1}{8})$ will be chosen sufficiently small later. We decompose the solution via the method of frozen coefficients: 
$$U:=U_1+U_2\quad\mbox{in}~ B_{r}^{+}(x_0).$$ 
This decomposition separates the problem into a constant-coefficient part, which is tractable via classical methods, and a variable-coefficient correction that will be shown to be small. 

If $\Gamma_r^0(x_0)$ is nonempty, then $U_1$ solves the following constant–coefficient problem
\begin{equation}\label{equu1}
\left\{ \begin{aligned}
\partial_{i}(\bar{A}^{ij}\partial_jU_1)&=\partial_i\bar{F}^i &\mbox{in}&~ B_{r}^{+}(x_0), \\
\bar{h}\,U_1-\hat{\gamma}\,\bar{A}^{nj}\partial_{j}U_1&=-\hat{\gamma}\bar{F^n} &\mbox{on}&~\Gamma_{r}^{0}(x_0), \\
U_1&=U &\mbox{on}&~\Gamma^{+}_{r}(x_0), 
\end{aligned} \right.
\end{equation}
and $U_2$ satisfies the variable–coefficient problem
\begin{equation}\label{equu2}
\left\{ \begin{aligned}
\partial_{i}(\bar{A}^{ij}\partial_jU_2)=-\partial_{i}\tilde{F}^i+\partial_{i}(F^i-\bar{F}^i)&+G &\mbox{in}&~ B_{r}^{+}(x_0), \\
\bar{h}U_2-\hat{\gamma}\left(A^{nj}\partial_{j}U_2+\tilde{F}^n-F^n+\bar{F}^n\right)&=(\bar{h}-h)U &\mbox{on}&~\Gamma_{r}^{0}(x_0), \\
U_2&= 0 &\mbox{on}&~\Gamma^{+}_{r}(x_0),
\end{aligned} \right.
\end{equation}
where
$$\bar{A}^{ij} = A^{ij}(x_0),~ \bar{F}^i = F^i(x_0',0), ~\bar{h} = h(x_0',0),~\mbox{ and}~~\tilde{F}^i:=(A^{ij}-\bar{A}^{ij})\partial_{j}U,~ i=1,2,\dots,n.$$
If $\Gamma_r^0(x_0)$ is empty, then $U_1\in H^1(B_{r}(x_0))$ with $U_1=U$ on $\partial B_{r}(x_0)$ and $U_2\in H_0^1(B_{r}(x_0))$ satisfy the first line of \eqref{equu1} and \eqref{equu2}, respectively.

{\bf Step 2. Energy estimate for $U_2$.} We now derive an energy estimate to control the remainder term $U_2$. Multiply the equation in \eqref{equu2} by $U_2$ and integrate by parts over $B_{r}^{+}(x_0)$: 
\begin{equation}\label{integralU2}
\int_{B_r^{+}(x_0)}A^{ij}\partial_{j}U_2\partial_iU_2+\left(\tilde{F}^i-F^i+\bar{F}^i\right) \partial_iU_2+GU_2\,dx=\int_{\Gamma_{r}^{0}(x_0)}\frac{1}{\hat{\gamma}}U_2\Big((\bar{h}-h)U-\bar{h}U_2\Big)\,dx'.
\end{equation}
We estimate each term. By uniform ellipticity,
$$\int_{ B_{r}^{+}(x_0)}A^{ij}\partial_{j}U_2\partial_{i}U_2 \,dx\geq \lambda\int_{ B_{r}^{+}(x_0)}|DU_2|^2 \,dx.$$
For the next term, we apply Cauchy's inequality
\begin{equation*}
\begin{aligned}
&\big|\int_{ B_{r}^{+}(x_0)}\Big(\tilde{F}^i+\bar{F}^i-F^i\Big)\partial_{i}U_2\,dx\big| \\
&\leq\,\frac{\lambda}{8}\int_{ B_{r}^{+}(x_0)}|DU_2|^2\,dx +C \Big(\|DU\|_{L^{\infty}(B_{r}^{+}(x_0))}^2+\mathcal{Q}^2\Big)r^{n+2\alpha}.
\end{aligned}
\end{equation*}
For the term involving $G$, since $U_2=0$ on $\Gamma^{+}_{r}(x_0)$, we use the boundary Poincar\'{e} inequality 
and Cauchy’s inequality:
\begin{equation*}
\begin{aligned}
\big|\int_{ B_{r}^{+}(x_0)}GU_2\,dx\big|&\leq \frac{\lambda}{C}\int_{ B_{r}^{+}(x_0)}\frac{U_2^2}{r^2}\,dx+C\int_{ B_{r}^{+}(x_0)}G^2r^2\,dx\\
&\leq \frac{\lambda}{8}\int_{ B_{r}^{+}(x_0)}|DU_2|^2\,dx+C\mathcal{Q}^2r^{n+2}.
\end{aligned}
\end{equation*}
Finally, we tackle the boundary integral in \eqref{integralU2}. Using the boundary condition $\frac{U}{\hat{\gamma}}=\frac{A^{nj}\partial_jU-F^n}{h}$ on $\Gamma_r^0$ along with the trace theorem and Poincar\'{e} inequality to $U_2$ in $B_r^{+}$ again, we have
\begin{equation*}
\begin{aligned}
&\int_{\Gamma_{r}^{0}(x_0)}\frac{1}{\hat{\gamma}}U_2((\bar{h}-h)U-\bar{h}U_2)\,dx'\leq\,\int_{\Gamma_{r}^{0}(x_0)}\frac{U}{\hat{\gamma}}(\bar{h}-h)U_2\,dx'\\
&\leq\,C\int_{\Gamma_{r}^{0}(x_0)}\Big(\|DU\|_{L^{\infty}(B_{r}^{+}(x_0))}+\mathcal{Q}\Big)|(\bar{h}-h)U_2|\,dx'\\
&\leq\,Cr^{n+2\alpha}\Big(\|DU\|_{L^{\infty}(B_{r}^{+}(x_0))}^2+\mathcal{Q}^2\Big)+\frac{\lambda}{Cr} \int_{\Gamma_{r}^{0}(x_0)}U_2^2\,dx'\\
&\leq\,Cr^{n+2\alpha}\Big(\|DU\|_{L^{\infty}(B_{r}^{+}(x_0))}^2+\mathcal{Q}^2\Big)+\frac{\lambda}{8}\int_{B_r^{+}}|DU_2|^2\,dx.
\end{aligned}
\end{equation*}
Substituting all these estimates into \eqref{integralU2}, we obtain 
\begin{equation}\label{equ_DU2}
\int_{ B_{r}^{+}(x_0)}|DU_2|^2\,dx\leq C\Big(\|DU\|_{L^{\infty}(B_{r}^{+}(x_0))}^2+\mathcal{Q}^2\Big)r^{n+2\alpha}.
\end{equation}

{\bf Step 3. Iterative argument for the gradient estimate.}
By Corollary \ref{constantlemmaR} and inequality \eqref{equ_DU2}, for any $0<\rho<r\leq r_0$, we have
\begin{equation}\label{holderDu}
\begin{aligned}
&\int_{B_{\rho}^{+}(x_0)}|DU-(DU)_{B_{\rho}^{+}(x_0)}|^2\,dx \\
&\leq\,C\int_{B_{\rho}^{+}(x_0)}|DU_1-(DU_1)_{B_{\rho}^{+}(x_0)}|^2\,dx+C\int_{B_{\rho}^{+}(x_0)}|DU_2-(DU_2)_{B_{\rho}^{+}(x_0)}|^2\,dx\\
&\leq\,C\left(\frac{\rho}{r}\right)^{n+2}\int_{ B_{r}^{+}(x_0)}|DU_1-(DU_1)_{ B_{r}^{+}(x_0)}|^2\,dx+C\int_{ B_{r}^{+}(x_0)}|DU_2|^2\,dx\\
&\leq\,C\left(\frac{\rho}{r}\right)^{n+2}\int_{ B_{r}^{+}(x_0)}|DU-(DU)_{ B_{r}^{+}(x_0)}|^2\,dx+C\Big(\|DU\|_{L^{\infty}(B_{r}^{+}(x_0))}^2+\mathcal{Q}^2\Big)r^{n+2\alpha}.
\end{aligned}
\end{equation}
Define the averaged quantities:
$$\varPhi_{\rho}(x_0):=\fint_{B_{\rho}^{+}(x_0)}|DU-(DU)_{B_{\rho}^{+}(x_0)}|^2\,dx,\quad \varPsi_{\rho}(x_0):=\fint_{B_{\rho}^{+}(x_0)}|DU|\,dx.$$
From \eqref{holderDu}, we derive 
$$\varPhi_{\sigma r}(x_0)\leq C_0\sigma^2\varPhi_{r}(x_0)+C\sigma^{-n}r^{2\alpha}\Big(\|DU\|_{L^{\infty}(B_{r}^{+}(x_0))}^2+\mathcal{Q}^2\Big).$$
Choose $\sigma$ such that $C_0\sigma^2\leq \frac{1}{4}$. Iterating this inequality yields
\begin{equation}\label{iterationPhi}
\varPhi_{\sigma^jr}(x_0)\leq \frac{1}{4^j}\varPhi_{r}(x_0)+C\sum_{i=1}^{j}\frac{1}{4^{j-i}}(\sigma^ir)^{2\alpha}\Big(\|DU\|_{L^{\infty}(B_{r}^{+}(x_0))}^2+\mathcal{Q}^2\Big).
\end{equation}
By the triangle inequality and H\"{o}lder's inequality, 
\begin{equation}\label{relationPhiPsi}
\begin{aligned}
|\varPsi_{\sigma r}(x_0)-\varPsi_{ r}(x_0)|\leq&\,\fint_{B_{\sigma r}^{+}(x_0)}|DU-(DU)_{B_{r}^{+}(x_0)}|\,dx\\
\leq&\,C \fint_{B_{ r}^{+}(x_0)}|DU-(DU)_{B_{r}^{+}(x_0)}|\,dx\leq C\varPhi_{ r}(x_0)^{1/2}.
\end{aligned}
\end{equation}
Using \eqref{iterationPhi} and \eqref{relationPhiPsi}, we obtain
\begin{equation*}
\begin{aligned}
\big|\varPsi_{\sigma^j r}(x_0)-\varPsi_{\sigma^{j-1} r}(x_0)\big|\leq&\,C\varPhi_{ \sigma^{j-1}r}(x_0)^{1/2}\\
\leq&\,C2^{-j+1}\varPhi_{r}(x_0)^{1/2}+C\sum_{i=1}^{j-1}\frac{1}{2^{j-i}}(\sigma^ir)^{\alpha}\Big(\|DU\|_{L^{\infty}(B_{r}^{+}(x_0))}+\mathcal{Q}\Big)
\end{aligned}
\end{equation*}
Summing from $j=1$ to $k$,
\begin{equation}\label{psisigmak}
\begin{aligned}
\varPsi_{\sigma^k r}(x_0)\leq&\,\varPsi_{ r}(x_0)+\sum_{j=1}^{k}\big|\varPsi_{\sigma^j r}(x_0)-\varPsi_{\sigma^{j-1} r}(x_0)\big|\\
\leq&\,\varPsi_{ r}(x_0)+C\varPhi_{r}(x_0)^{1/2}+C\sum_{j=1}^{k}\sum_{i=1}^{j-1}\frac{1}{2^{j-i}}(\sigma^ir)^{\alpha}\Big(\|DU\|_{L^{\infty}(B_{r}^{+}(x_0))}+\mathcal{Q}\Big)\\
\leq&\,\varPsi_{ r}(x_0)+C\varPhi_{r}(x_0)^{1/2}+C \sum_{i=1}^{k-1} (\sigma^ir)^{\alpha}\Big(\|DU\|_{L^{\infty}(B_{r}^{+}(x_0))}+\mathcal{Q}\Big).
\end{aligned}
\end{equation}
 Because $Du$ is continuous, letting $k\to\infty$ in \eqref{psisigmak} yields, for $r\in(0,r_0)$
\begin{equation}\label{Duxlinfty}
|DU(x_0)|\leq\varPsi_{ r}(x_0)+C\varPhi_{r}(x_0)^{1/2}+C_1r^{\alpha}\Big(\|DU\|_{L^{\infty}(B_{r}^{+}(x_0))}+\mathcal{Q}\Big).
\end{equation}
Choosing $r_0$ small such that $C_1r_0^{\alpha}\leq \frac{1}{3^n}$ in \eqref{Duxlinfty}. Applying H\"older's inequality, we then obtain
\begin{equation}\label{Duxlinfty1}
|DU(x_0)|\leq Cr^{-\frac{n}{2}} \Big(\|DU\|_{L^2(B^{+}_{3/8})}+\mathcal{Q}\Big)+\frac{1}{3^n}\|DU\|_{L^{\infty}(B_{r}^{+}(x_0))}.
\end{equation}
Now take $x_0\in B^{+}_{\frac{1}{4}-\frac{r_0}{4^{k+2}}}$ and $r=\frac{1}{4^{k+3}}r_0$ in \eqref{Duxlinfty1}. Since $x_0$ is arbitrary, it follows that
\begin{equation*}
\|DU\|_{L^{\infty}(B^{+}_{\frac{1}{4}-\frac{r_0}{4^{k+2}}})}\leq C2^{n(k+3)}\Big(\|DU\|_{L^2(B^{+}_{3/8})}+\mathcal{Q}\Big)+\frac{1}{3^n}\|DU\|_{L^{\infty}(B^{+}_{\frac{1}{4}-\frac{r_0}{4^{k+3}}})}.
\end{equation*}
Multiply both sides by $\frac{1}{3^{n(k+3)}}$ and sum over $k=0,1,2,\ldots$,
\begin{equation*} 
\begin{aligned}
&\sum_{k=0}^{\infty}\frac{1}{3^{n(k+3)}}\|DU\|_{L^{\infty}(B^{+}_{\frac{1}{4}-\frac{r_0}{4^{k+2}}})}\\
&\leq\,C\sum_{k=0}^{\infty}\big(\frac{2 }{3 }\big)^{n(k+3)}\Big(\|DU\|_{L^2(B^{+}_{3/8})}+\mathcal{Q}\Big)+\sum_{k=0}^{\infty}\frac{1}{3^{n(k+4)}}\|DU\|_{L^{\infty}(B^{+}_{\frac{1}{4}-\frac{r_0}{4^{k+3}}})}.
\end{aligned}
\end{equation*}
Since $\|DU\|_{L^{\infty}(B^{+}_{1/2})}<+\infty$, we absorb the right-hand summation to obtain 
\begin{equation}\label{Duxinfty3}
 \|DU\|_{L^{\infty}(B^{+}_{1/8})}\leq C\Big(\|DU\|_{L^2(B^{+}_{3/8})}+\mathcal{Q}\Big).
\end{equation}

Recall $\phi=0$. Multiply \eqref{equgradientsmallgamma} by $U\xi^2$ where $\xi\in C_c^{\infty}(B_{1/2})$ satisfies $\xi=1$ in $B^{+}_{3/8}$ and $|D\xi|\leq C$ in $B_{1/2}^{+}$. Integrating by parts and Cauchy's inequality yield 
$\|DU\|_{L^2(B^{+}_{3/8})}\leq C\mathcal{Q}.$
Substituting into \eqref{Duxinfty3} gives
$$\|DU\|_{L^{\infty}(B^{+}_{1/8})}\leq C\mathcal{Q}. $$
From this, \eqref{holderDu}, and Campanato's characterization of H\"older continuous functions, we conclude
$$[DU]_{C^{\alpha}(B^{+}_{1/16})}\leq C\mathcal{Q}.$$
Finally, by interpolation inequalities, we use $\|U\|_{L^{2}(B^{+}_{1/8})}$ and $\|DU\|_{L^{\infty}(B^{+}_{1/8})}$ to estimate $\|U\|_{L^{\infty}(B^{+}_{1/8})}$. This completes the proof of Theorem \ref{gradientsmallgamma3}.
\end{proof} 

\begin{lemma}\label{grdeientsmallgamma6}
Let $U\in H^{1}(B_{1}^{+})$ solve \eqref{equgradientsmallgamma} with $A^{ij}\in C^{\alpha}(B_{1}^{+})$ satisfying \eqref{uniformelliptic} in $B_{1}^{+}$.
Suppose $F^i\in C^{\alpha}(B_{1}^{+})$, $G\in L^{\infty}(B_{1}^{+})$, $h\in C^{1,\alpha}(B_{1}^{+})$, $\,\phi\in C^{\alpha}(B_{1}^{+})$, and $|\hat{\gamma}|>1$. Then it holds that
\begin{multline*}
 \|D U\|_{C^{\alpha}(B_{1/4}^{+})}\leq C\left(\|D U\|_{L^2(B_{3/4}^{+})}+\|\frac{U}{\hat{\gamma}}\|_{L^2(B_{3/4}^{+})}\right.\\
 \left.+\|\frac{\phi}{\hat{\gamma}}\|_{C^{\alpha}(B_{3/4}^{+})}+\|F\|_{C^{\alpha}(B_{3/4}^{+})}+\|G\|_{L^{\infty}(B_{3/4}^{+})}\right),
\end{multline*}
where $C$ depends on $\lambda,\Lambda,n,\|A^{ij}\|_{C^{\alpha}(B_{1}^{+})}$ and $\|h\|_{C^{\alpha}(B_{1}^{+})}$, but is independent of $\hat{\gamma}$.
\end{lemma}

\begin{proof}
Define $\widehat{U}=U-(U)_{B_{3/4}^{+}}$. Then $\widehat{U}$ solves 
\begin{equation*} 
\left\{ \begin{aligned}
\partial_i(A^{ij}\partial_j\widehat{U})&=\partial_iF^i+G& \mbox{in}&~ B_{1}^{+},\\
\frac{h}{\hat{\gamma}}\widehat{U} - A^{nj}\partial_j\widehat{U}&=-F^n+\frac{1}{\hat{\gamma}}(\phi-h\,(U)_{B_{3/4}^{+}})& \mbox{on}&~\Gamma_{1}^{0}.
\end{aligned} \right.
\end{equation*}
By using  \cite{Lieberm}*{Theorem 5.54}, 
\begin{multline*}
\|D\widehat{U}\|_{C^{\alpha}(B_{1/2}^{+})}\leq C\left(\|\widehat{U}\|_{L^2(B_{3/4}^{+})}+\|\frac{1}{\hat{\gamma}}(\phi-h\,(U)_{B_{3/4}^{+}})\|_{C^{\alpha}(B_{3/4}^{+})}\right.\\
\left.+\|F\|_{C^{\alpha}(B_{3/4}^{+})}+\|G\|_{L^{\infty}(B_{3/4}^{+})}\right).
\end{multline*}
The Poincar\'e inequality gives $\|\widehat{U}\|_{L^2(B_{3/4}^{+})}\leq C\|D\widehat{U}\|_{L^2(B_{3/4}^{+})}.$ Using H\"{o}lder inequality, we have
$$|(U)_{B_{3/4}^{+}}|\leq C\|{U}\|_{L^2(B_{3/4}^{+})}.$$ Since $DU=D\widehat{U}$, the proof is finished.
\end{proof}
 
\section{Gradient estimate in a reduced dimensional space for small $\gamma$}\label{sec_3}

In this section, we study the elliptic equation $L\mathtt{U}=\partial_{i}\mathtt{F}^{i}+\mathtt{G}$ in $\bR^{n-1}$, where the operator $L$ is defined by 
$$L\mathtt{U}=\partial_{i}(A^{ij}(x')\partial_{j}\mathtt{U})+c(x')\mathtt{U}$$
with $c(x')=-\frac{1}{\gamma}$ for a small positive constant $\gamma$. Our primary objective is to characterize the behavior of the solution and its dependence on $\gamma$ in the small-$\gamma$ regime. In Subsection \ref{subsection31}, we consider the case where the principal coefficients $A^{ij}$ are uniformly elliptic. By combining Agmon's method with the weak Harnack inequality, we derive an iterative formula for the $L^\infty$-norm of $\mathtt{U}$, which ultimately establishes the regularity of $\mathtt{U}$ as $\gamma \to 0$. Subsection \ref{subsection32} addresses the degenerate case, employing the comparison principle, iterative techniques, and the regularity results from Subsection \ref{subsection31}. We remark that, through an argument building on Subsection \ref{subsection31}, the assumption that $c(x')$ is constant can be relaxed to the case where $c(x')$ is measurable and satisfies $\frac{\lambda}{\gamma}<-c(x')\leq\frac{\Lambda}{\gamma}$ for positive constants $\lambda$ and $\Lambda$. Consequently, the constant-coefficient assumption in the theorem of Subsection \ref{subsection32} can be similarly extended within this measurable and bounded framework.

To distinguish from the full-space $\mathbb{R}^n$ notation, we denote by $\mathtt{B}_r(x_0')$ the ball in $\mathbb{R}^{n-1}$ of radius $r$ centered at $x_0'$, writing $\mathtt{B}_r$ when centered at the origin. Functions on $\mathbb{R}^{n-1}$ are denoted by $\mathtt{U}$, $\mathtt{F}$, $\mathtt{G}$, etc.

\subsection{Interior and global estimates for uniform elliptic equations}\label{subsection31}

\begin{prop}\label{thmgammato0}
Let $n\geq 2$, and let $\mathtt{U}\in H^1(\mathtt{B}_{1})$ be a solution of
\begin{equation}\label{inGamma}
\sum\limits_{i,j=1}^{n-1}\partial_{i}\Big(A^{ij}(x')\partial_{j}\mathtt{U}\Big)-\frac{1}{\gamma}\mathtt{U}=\sum\limits_{i=1}^{n-1}\partial_i\mathtt{F}^i+\mathtt{G}\quad \mbox{in}~\mathtt{B}_{1}\subset \bR^{n-1},
\end{equation}
where the coefficients $A^{ij}\in L^{\infty}(\mathtt{B}_{1})$ satisfy the uniformly elliptic condition \eqref{uniformelliptic} in $\mathtt{B}_{1}$ and $\gamma > 0$.
 
(i) Assume $\mathtt{F}\in L^{\infty}(\mathtt{B}_{1})$ and $\mathtt{G}\in L^{\infty}(\mathtt{B}_{1})$. Then there exists $\bar{\gamma}_0=\bar{\gamma}_0(\lambda,\Lambda,n, \|A^{ij}\|_{L^{\infty}(\mathtt{B}_{1})})$ such that for all $\gamma\in(0,\bar{\gamma}_0)$,
\begin{equation}\label{equgammato021}
\|\mathtt{U}\|_{L^{\infty}(\mathtt{B}_{1/4})}\leq\, C\big(\frac{1}{2}\big)^{\frac{1}{8{C}_{0}\sqrt{\gamma}}}\|\mathtt{U}\|_{L^2(\mathtt{B}_{3/4})}+C\sqrt{\gamma}\|\mathtt{F}\|_{L^{\infty}(\mathtt{B}_{3/4})}+C\gamma\|\mathtt{G}\|_{L^{\infty}(\mathtt{B}_{3/4})}, 
\end{equation}
where $C_0$ and $C$ depend only on $\lambda,\Lambda,n$ and $\|A^{ij}\|_{L^{\infty}(\mathtt{B}_{1})}$, and are independent of  $\frac{1}{\gamma}$.
 
(ii) Assume $A^{ij}\in C^{\alpha}(\mathtt{B}_{1})$, $\mathtt{F}\in C^{\alpha}(\mathtt{B}_{1})$, and $\mathtt{G}\in L^{\infty}(\mathtt{B}_{1})$. Then there exists $\bar{\gamma}_0=\bar{\gamma}_0(\lambda,\Lambda,n, \|A^{ij}\|_{C^{\alpha}(\mathtt{B}_{1})})$ such that for all $\gamma\in(0,\bar{\gamma}_0)$,
\begin{equation}\label{equgammato0}
\|\mathtt{U}\|_{L^{\infty}(\mathtt{B}_{1/4})}\leq\, C\big(\frac{1}{2}\big)^{\frac{1}{8{C}_{0}\sqrt{\gamma}}}\|\mathtt{U}\|_{L^2(\mathtt{B}_{1/2})}+ C\gamma^{\frac{1+\alpha}{2}} [\mathtt{F}]_{C^{\alpha}(\mathtt{B}_{1/2})}+C\gamma\|\mathtt{G}\|_{L^{\infty}(\mathtt{B}_{1/2})},
\end{equation}
and
\begin{equation}\label{equgammato02}
\begin{aligned}
&\|D \mathtt{U}\|_{L^{\infty}(\mathtt{B}_{1/4})}+{\gamma}^{\frac{\alpha}{2}}[D \mathtt{U}]_{C^{\alpha}(\mathtt{B}_{1/4})}\\
&\leq C\big(\frac{1}{2}\big)^{\frac{1}{8{C}_{0}\sqrt{\gamma}}}\gamma^{-1/2}\|\mathtt{U}\|_{L^2(\mathtt{B}_{3/4})}+C\gamma^{\frac{\alpha}{2}} [\mathtt{F}]_{C^{\alpha}(\mathtt{B}_{3/4})}+C\gamma^{1/2}\|\mathtt{G}\|_{L^{\infty}(\mathtt{B}_{3/4})},
\end{aligned}
\end{equation}
where $C_0$ and $C$ depend only on $\lambda,\Lambda,n$ and $\|A^{ij}\|_{C^{\alpha}(\mathtt{B}_{1})}$, and are independent of $\frac{1}{\gamma}$.
\end{prop} 
 
\begin{proof}
Here we adapt Agmon's idea (see \cite{Kr07}*{Lemma 5.5}) to extend the problem into a higher-dimensional space where the singular term $-\frac{1}{\gamma}\mathtt{U}$ is transformed into a regular derivative by considering an auxiliary function on a higher-dimensional space, thereby enabling the use of standard elliptic estimates.

We define the cylindrical domain $Q_R\subset\bR^n=\{(x',x_n)\,|\, x'\in\bR^{n-1},x_n\in\bR\}$ by $Q_R:=\mathtt{B}_R\times (-R,R)$, and introduce the elliptic operator 
$$L_0W:=\sum\limits_{i,j=1}^{n-1}\partial_{i}\Big(A^{ij}(x')\partial_{j}W\Big)+\partial_{nn}W.$$
Suppose $\mathtt{U}(x')$ is a solution to equation \eqref{inGamma}. Define the function $W(x',x_n)=\mathtt{U}(x')\cos\frac{x_n}{\sqrt{\gamma}}$. Then $W$ satisfies $$L_0W=\sum\limits_{i=1}^{n}\partial_i\tilde{F}^i,$$ 
where the modified right-hand side terms are given by 
$$\tilde{F}^i(x',x_n)=\mathtt{F}^{i}(x')\cos\frac{x_n}{\sqrt{\gamma}},~ 1\leq i\leq n-1,\quad \tilde{F}^n(x',x_n)=\sqrt{\gamma}\mathtt{G}(x')\sin\frac{x_n}{\sqrt{\gamma}}.$$
These terms are explicitly defined and remain controlled under the assumptions on $\mathtt{F}$ and $\mathtt{G}$.
By classical elliptic theory, we obtain the interior estimate
\begin{equation}\label{inGammaW}
  \|W\|_{L^{\infty}(Q_{1/8})}\leq C\Big(\|W\|_{L^{2}(Q_{1/4})}+\|\tilde{F}\|_{L^{\infty}(Q_{1/4})}\Big).  
\end{equation} 
Returning to the original function $\mathtt{U}$, and noting that $|W(x',t)|\leq|\mathtt{U}(x')|$, we deduce
\begin{equation}\label{inGammaU}
\|\mathtt{U}\|_{L^{\infty}(\mathtt{B}_{1/8})}\leq C\Big(\|\mathtt{U}\|_{L^{2}(\mathtt{B}_{1/4})}+\|{\mathtt{F}}\|_{L^{\infty}(\mathtt{B}_{1/4})}+\sqrt{\gamma}\|{\mathtt{G}}\|_{L^{\infty}(\mathtt{B}_{1/4})}\Big).
\end{equation}

(i) The proof of estimate \eqref{equgammato021} relies on an application of the weak Harnack inequality to $W$. For any $r\in (0,\frac{1}{8})$, applying the weak Harnack inequality (see Theorem 8.18 in \cite{GT}) to the functions $W-\underset{Q_{4r}}{\inf}\,W$ and $\underset{Q_{4r}}{\sup}\,W -W$ in $Q_{4r}$, we obtain
\begin{equation}\label{inGammaA}
\underset{Q_r}{\osc}\, W\leq \beta\, \underset{Q_{4r}}{\osc}\, W+Cr\|\tilde{F}\|_{L^{\infty}(Q_{4r})},
\end{equation}
where $\beta\in(0,1)$ is a universal constant and $C$ depends only on $\lambda,\Lambda, n$.

Define $r_0=\frac{\sqrt{\gamma}\pi}{2}$ and set $r_{i+1}=4r_i$, for $i\geq 0$. Then inequality \eqref{inGammaA} implies
\begin{equation}\label{inGammaB}
\underset{Q_{r_i}}{\osc}\, W\leq \beta\, \underset{Q_{r_{i+1}}}{\osc}\, W+Cr_{i+1}\|\tilde{F}\|_{L^{\infty}(Q_{r_{i+1}})}.
\end{equation} 
Let $k_0$ be the smallest integer such that $\beta^{k_0}\leq \frac{1}{4}$, and set ${C}_{0}=4^{k_0}\pi$. Iterating \eqref{inGammaB} from $i=0$ to $k_0-1$ yields 
\begin{equation}\label{inGammaC}
\underset{Q_{\frac{\sqrt{\gamma}\pi}{2}}}{\osc}\, W\leq \frac{1}{4}\, \underset{Q_{{C}_{0}\sqrt{\gamma}}}{\osc}\, W+C\sqrt{\gamma}\|\tilde{F}\|_{L^{\infty}(Q_{ {C}_{0}\sqrt{\gamma}})}.
\end{equation} 
Note that $\|\mathtt{U}\|_{L^{\infty}(\mathtt{B}_{\frac{\sqrt{\gamma}\pi}{2}})}\leq \underset{Q_{\frac{\sqrt{\gamma}\pi}{2}}}{\osc}\, W$ and $\underset{Q_{{C}_{0}\sqrt{\gamma}}}{\osc}\, W\leq 2\|\mathtt{U}\|_{L^{\infty}(\mathtt{B}_{ {C}_{0}\sqrt{\gamma}})}$. Therefore, \eqref{inGammaC} implies
$$\|\mathtt{U}\|_{L^{\infty}(\mathtt{B}_{\frac{\sqrt{\gamma}\pi}{2}})}\leq \frac{1}{2} \|\mathtt{U}\|_{L^{\infty}(\mathtt{B}_{\ {C}_{0}\sqrt{\gamma}})}+C \sqrt{\gamma}\|{\mathtt{F}}\|_{L^{\infty}(\mathtt{B}_{ {C}_{0}\sqrt{\gamma}})}+C\gamma\|{\mathtt{G}}\|_{L^{\infty}(\mathtt{B}_{ {C}_{0}\sqrt{\gamma}})}.$$
This provides an estimate for $\mathtt{U}$ at the scale $\sqrt{\gamma}$. To establish the global estimate over $\mathtt{B}_{1/4}$ we iterate this inequality along a chain of balls. Specifically, for any $x_0\in \mathtt{B}_{{C}_{0}k\sqrt{\gamma}}$, applying the above estimate to the ball $\mathtt{B}_{{C}_{0}\sqrt{\gamma}}$ yields
\begin{equation}\label{inGammaD}
\|\mathtt{U}\|_{L^{\infty}(\mathtt{B}_{{C}_{0}k\sqrt{\gamma}})}\leq \frac{1}{2} \|\mathtt{U}\|_{L^{\infty}(\mathtt{B}_{{C}_{0}(k+1)\sqrt{\gamma}})}+C \sqrt{\gamma}\|{\mathtt{F}}\|_{L^{\infty}( \mathtt{B}_{{C}_{0}(k+1)\sqrt{\gamma}})}+C\gamma\|{\mathtt{G}}\|_{L^{\infty}(\mathtt{B}_{{C}_{0}(k+1)\sqrt{\gamma}})}.
\end{equation}
Choose $\bar{\gamma}_0>0$ such that $C_0\sqrt{\bar{\gamma}_0}=\frac{1}{16}$. For $\gamma\in (0, \bar{\gamma}_0)$, let $k_{\gamma}$ be the largest integer less than $\frac{1}{8C_0\sqrt{\gamma}}-1$. Iterating \eqref{inGammaD} from $k=1$ to $k_{\gamma}$ gives 
\begin{equation}\label{inGammaE}
\|\mathtt{U}\|_{L^{\infty}(\mathtt{B}_{\frac{1}{16}})}\leq C\big(\frac{1}{2}\big)^{\frac{1}{8{C}_{0}\sqrt{\gamma}}} \|\mathtt{U}\|_{L^{\infty}(\mathtt{B}_{1/8})}+C\Big(\sqrt{\gamma}\|{\mathtt{F}}\|_{L^{\infty}( \mathtt{B}_{1/8})}+\gamma\|{\mathtt{G}}\|_{L^{\infty}(\mathtt{B}_{1/8})}\Big).
\end{equation}
Substituting the interior estimate \eqref{inGammaU} into \eqref{inGammaE}, we obtain
$$\|\mathtt{U}\|_{L^{\infty}(\mathtt{B}_{\frac{1}{16}})}\leq C\big(\frac{1}{2}\big)^{\frac{1}{8{C}_{0}\sqrt{\gamma}}} \|\mathtt{U}\|_{L^{2}(\mathtt{B}_{1/4})}+C\Big(\sqrt{\gamma}\|{\mathtt{F}}\|_{L^{\infty}( \mathtt{B}_{1/4})}+\gamma\|{\mathtt{G}}\|_{L^{\infty}(\mathtt{B}_{1/4})}\Big). $$
 
(ii) Under the higher regularity assumptions in (ii), the same general strategy applies, but we now employ Schauder estimates. Specifically, in \eqref{inGammaW}, the term $\|\tilde{F}\|_{L^{\infty}({Q}_{1/4})}$ may be replaced with $C (\sum_{i=1}^{n-1}[\tilde{F}^{i}]_{C_{x'}^{\alpha}(Q_{\frac{1}{4}})}+[\tilde{F}^{n}]_{C_{x_n}^{\alpha}(Q_{\frac{1}{4}})})$, and in \eqref{inGammaA}, the term $Cr\|\tilde{F}\|_{L^{\infty}(Q_{4r})}$ is replaced with $Cr^{1+\alpha}(\sum_{i=1}^{n-1}[\tilde{F}^{i}]_{C_{x'}^{\alpha}(Q_{4r})}+[\tilde{F}^{n}]_{C_{x_n}^{\alpha}(Q_{4r})})$. Here, we employ the notation for a partial H\"older semi-norm of the function $f$ with respect to $x'$ over domain $\mathcal{D}$ as :$$[f]_{C_{x'}^{\alpha} (\mathcal{D})}=\underset{\substack{(z',x_n),(y',x_n)\in \mathcal{D}\\ z'\neq y'}}{\sup}\frac{|f(z',x_n)-f(y',x_n)|}{|z'-y'|^{\alpha}},$$ and a partial H\"older semi-norm of the function $f$ with respect to $x_n$ over domain $\mathcal{D}$ as:
$$[f]_{C_{x_n}^{\alpha}(\mathcal{D})}=\underset{\substack{(x',z_n),(x',y_n)\in  \mathcal{D}\\ z_n\neq y_n}}{\sup}\frac{|f(x',z_n)-f(x',y_n)|}{|z_n-y_n|^{\alpha}}.$$
 
To prove the gradient estimate \eqref{equgammato02}, we perform a change of variables: set $x'=\sqrt{\gamma}y'$ with $y'\in \mathtt{B}_{1}$ and define 
$\widetilde{\mathtt{U}}(y')=\mathtt{U}(\sqrt{\gamma}y')= \mathtt{U}(x')$. 
This rescaling absorbs the parameter $\gamma$ into the spatial variable, resulting in a uniformly elliptic equation with order-one coefficients. Specifically, $\widetilde{\mathtt{U}}\in H^1(\mathtt{B}_{1})$ satisfies 
\begin{equation*}
\partial_{i}(\widetilde{A}^{ij}\partial_{j}\widetilde{\mathtt{U}})-\tilde{\mathtt{U}}=\dv(\sqrt{\gamma}\widetilde{\mathtt{F}})+\gamma \widetilde{\mathtt{G}}, \quad\mbox{in}~\mathtt{B}_{1},
\end{equation*}
where $\widetilde{A}^{ij}(y')=A^{ij}(\sqrt{\gamma}y')$, $\widetilde{\mathtt{F}}(y')=\mathtt{F}(\sqrt{\gamma}y')$ and $\widetilde{\mathtt{G}}(y')=\mathtt{G}(\sqrt{\gamma}y')$. Applying standard elliptic estimates yields
\begin{equation*}
\|D \widetilde{\mathtt{U}}\|_{C^{\alpha}(\mathtt{B}_{1/2})}\leq\, C\|\widetilde{\mathtt{U}}\|_{L^2(\mathtt{B}_{1})}+C[\sqrt{\gamma}\widetilde{\mathtt{F}}]_{C^{\alpha}(\mathtt{B}_{1})}+C\|\gamma \widetilde{\mathtt{G}}\|_{L^{\infty}(\mathtt{B}_{1})}.
\end{equation*}
Reverting to the original variables, we obtain
\begin{equation}\label{inGammaF}
\begin{aligned}
{\gamma}^{1/2}\|D \mathtt{U} \|_{L^{\infty}(\mathtt{B}_{\frac{\sqrt{\gamma}}{2}})}+\gamma^{\frac{1+\alpha}{2}}[D \mathtt{U} ]_{C^{\alpha}(\mathtt{B}_{\frac{\sqrt{\gamma}}{2}})}
\leq\, C \|\mathtt{U}\|_{L^{\infty}(\mathtt{B}_{\sqrt{\gamma}})} +C\gamma^{\frac{1+\alpha}{2}}[\mathtt{F}]_{C^{\alpha}(\mathtt{B}_{\sqrt{\gamma}})}+C\gamma \|\mathtt{G}\|_{L^{\infty}(\mathtt{B}_{\sqrt{\gamma}})}.
\end{aligned}
\end{equation}
Finally, using estimate \eqref{equgammato0} to bound $ \|\mathtt{U}\|_{L^{\infty}(\mathtt{B}_{\sqrt{\gamma}})}$ in \eqref{inGammaF} yields the desired gradient estimate \eqref{equgammato02}. 
\end{proof}

For the solution to equation \eqref{inGamma}, standard elliptic theory yields the corresponding estimates for $\gamma\ge \bar{\gamma}_0$. Combined with the results of Proposition \ref{thmgammato0} for $\gamma\in(0,\bar{\gamma}_0)$, we obtain the following corollary.

\begin{corollary}\label{corthmgammato0}
Let $\mathtt{U}\in H^{1}(\mathtt{B_1})$ be a solution of  \eqref{inGamma}. Suppose the coefficients $A^{ij}\in L^{\infty}(\mathtt{B}_{1})$ satisfy the uniformly elliptic condition \eqref{uniformelliptic} in $\mathtt{B}_{1}$, $\mathtt{F}\in L^{\infty}(\mathtt{B}_1)$ and $\mathtt{G}\in L^{\infty}(\mathtt{B}_1)$. 

(i) For any $\gamma\in(0, +\infty)$, 
\begin{equation*}
\|\mathtt{U}\|_{L^{\infty}(\mathtt{B}_{1/4})}\leq\, C \|\mathtt{U}\|_{L^2(\mathtt{B}_{3/4})}+C\sqrt{\gamma}\|\mathtt{F}\|_{L^{\infty}(\mathtt{B}_{3/4})}+C\gamma\|\mathtt{G}\|_{L^{\infty}(\mathtt{B}_{3/4})}, 
\end{equation*}
where the constant $C$ depends only on $\lambda,\Lambda,n$ and $\|A^{ij}\|_{L^{\infty}(\mathtt{B}_{1})}$.

(ii) If, in addition, $A^{ij}\in C^{\alpha}(\mathtt{B}_1)$ $\mathtt{F}\in C^{\alpha}(\mathtt{B}_1)$, then for any $\gamma\in (0, +\infty)$, 
\begin{equation*}
\|\mathtt{U}\|_{L^{\infty}(\mathtt{B}_{1/4})}\leq\, C\|\mathtt{U}\|_{L^2(\mathtt{B}_{1/2})}+ C\gamma^{\frac{1+\alpha}{2}} [\mathtt{F}]_{C^{\alpha}(\mathtt{B}_{1/2})}+C\gamma\|\mathtt{G}\|_{L^{\infty}(\mathtt{B}_{1/2})},
\end{equation*}
and
\begin{equation*}
\|D \mathtt{U}\|_{L^{\infty}(\mathtt{B}_{1/4})} 
\leq C\|\mathtt{U}\|_{L^2(\mathtt{B}_{3/4})}+C\gamma^{\frac{\alpha}{2}} [\mathtt{F}]_{C^{\alpha}(\mathtt{B}_{3/4})}+C\gamma^{1/2}\|\mathtt{G}\|_{L^{\infty}(\mathtt{B}_{3/4})},
\end{equation*}
where $C$ depends only on $\lambda,\Lambda,n$, and $\|A^{ij}\|_{C^{\alpha}(\mathtt{B}_{1})}$.
 \end{corollary}

Combining Corollary~\ref{corthmgammato0}~(i) with a standard extension argument, we obtain the following estimate near the boundary.

\begin{corollary}\label{corbdgammato0}
For $n\geq2$, let $\mathtt{U}\in H^1(\mathtt{B}_{1}^{+})$ satisfy
\begin{equation*}
\left\{ \begin{aligned}
\sum\limits_{i,j=1}^{n-1}\partial_{i}(A^{ij}\partial_{j}\mathtt{U})-\frac{1}{\gamma}\mathtt{U}&=\sum\limits_{i=1}^{n-1}\partial_i\mathtt{F}^i+\mathtt{G} \quad&\mbox{in}&&~ \mathtt{B}_{1}^{+}&\subset \bR^{n-1},\\
 \mathtt{U}&=0\quad&\mbox{on}&&~\mathtt{\Gamma}_{1}^{0}&:=\mathtt{B}_{1}\cap\{x_{n-1}=0\},
 \end{aligned} \right.
\end{equation*}
with $A^{ij}\in L^{\infty}(\mathtt{B}_{1}^{+})$ satisfying \eqref{uniformelliptic} in $\mathtt{B}_{1}^{+}.$
Assume $\mathtt{F},\mathtt{G}\in L^{\infty}(\mathtt{B}_{1}^{+})$. Then for $\gamma\in(0,+\infty)$,
\begin{equation*}
\|\mathtt{U}\|_{L^{\infty}(\mathtt{B}_{1/4}^+)}\leq C\|\mathtt{U}\|_{L^2(\mathtt{B}^{+}_{3/4})}+C\sqrt{\gamma}\|\mathtt{F}\|_{L^{\infty}(\mathtt{B}_{3/4}^{+})}+C\gamma\|\mathtt{G}\|_{L^{\infty}(\mathtt{B}_{3/4}^{+})},
\end{equation*}
where $C$ depends only on $\lambda,\Lambda,n$ and $\|A^{ij}\|_{L^{\infty}(\mathtt{B}_{1})}$.
\end{corollary}

\subsection{Pointwise estimates for degenerate equations}\label{subsection32}

For $l,s>0$, $\tilde{\varepsilon}\in(0,\frac{1}{4})$, we define the weighted norm
$$\|\mathtt{F}\|_{l,\mathtt{B}_s}:=\underset{x'\in \mathtt{B}_s}{\sup}\frac{|\mathtt{F}(x')|}{(\tilde{\varepsilon}+|x'|^2)^l}.$$ 

\begin{prop}\label{epsilongammato0}
Let $n\geq 2$, $\tilde \lambda,\tilde \Lambda>0$, and suppose that $v\in H^1(\mathtt{B}_{1})$ satisfies 
\begin{equation}\label{equ_v}
\sum\limits_{i=1}^{n-1}\partial_i \Big(\big(\tilde{\varepsilon}+t(x')\big)\partial_i v(x')\Big)-\frac{1}{\gamma}v(x')=\sum\limits_{i=1}^{n-1}\partial_i\mathtt{F}^i+\mathtt{G} \quad \mbox{for}~x'\in \mathtt{B}_1\subset \bR^{n-1},
\end{equation}
where $t\in C^1(\mathtt{B}_{1})$ satisfies 
\begin{equation}\label{assump_t}
\tilde{\lambda} |x'|^2\leq   t(x')\leq \tilde{\Lambda} |x'|^2 \quad\mbox{and}~ |D t(x')|\leq 2\tilde{\Lambda} |x'|.
\end{equation}
Assume $\mathtt{F},G\in L^{\infty}(\mathtt{B}_{1})$. Then, for any $\sigma > 0$, if we define
\begin{equation}\label{bargamma0}
\tilde{\gamma}_0 = \left( \tilde{\Lambda}\sigma\big(n+1+2(\frac{\sigma}{2}-1)_{+}\big) \right)^{-1},
\end{equation}
the estimate
\begin{equation*}
|v(x')| \leq C  (\tilde{\varepsilon} + |x'|^2)^{\sigma/2} \left( \|v\|_{L^{\infty}(\partial \mathtt{B}_{1})} + \sqrt{\gamma} \|\mathtt{F}\|_{\frac{\sigma+1}{2},\mathtt{B}_{1}} + \gamma \|\mathtt{G}\|_{\frac{\sigma}{2},\mathtt{B}_{1}} \right)
\end{equation*}
holds whenever $0 < \gamma < \tilde{\gamma}_0$.  The constant $C > 0$ depends only on $n$, $\tilde{\lambda}$, $\tilde{\Lambda}$, $\sigma$ and a lower bound of $\tilde{\gamma}_0-\gamma$.
\end{prop}

For notational simplicity, we define the operator
$$\mathcal{L}_{\gamma}v:= \sum\limits_{i=1}^{n-1}\partial_{i} \Big(\big(\tilde{\varepsilon}+t(x')\big)\partial_{i} v(x')\Big)-\frac{1}{\gamma}v(x').$$
The proof relies on a comparison principle with a barrier function that captures the decay near the origin. We decompose the solution into a homogeneous part, controlled via the maximum principle, and an inhomogeneous part, estimated via energy methods and a blow-up argument.

Specifically, for $0<r<1$, we write
$$v:=v_1+v_2,\quad\mbox{in}~ \mathtt{B}_r,$$
where $v_{1}\in H^1(\mathtt{B}_{r})$ satisfies the homogeneous equation
\begin{equation*}
\left\{ \begin{aligned}
 \mathcal{L}_{\gamma}v_1&= 0&\mbox{in}&~\mathtt{B}_{r}, \\
 v_1&=v&\mbox{on}&~\partial \mathtt{B}_{r}, 
\end{aligned} \right.
\end{equation*}
and $v_2\in H_0^1(\mathtt{B}_r)$ solves the inhomogeneous problem
\begin{equation*}
\left\{ \begin{aligned}
 \mathcal{L}_{\gamma}v_2&=\sum\limits_{i=1}^{n-1}\partial_i\mathtt{F}^i+\mathtt{G}&\mbox{in}&~\mathtt{B}_{r}, \\
v_2&=0&\mbox{on}&~\partial \mathtt{B}_{r}.
\end{aligned} \right.
\end{equation*}
We need the following two lemmas.  

\begin{lemma}\label{epsilongammato01}
For any $\tilde{\sigma}>0$, let ${\gamma}({\tilde{\sigma}})$ be defined as 
\begin{equation}\label{gammatildesigma}
	 {\gamma}({\tilde{\sigma}}):=\frac{1}{2\tilde{\Lambda}\tilde{\sigma}(n+1)+4\tilde{\Lambda}\tilde{\sigma}(\tilde{\sigma}-1)_{+} }
\end{equation} 
Then for any $\gamma\in(0, {\gamma}({\tilde{\sigma}})]$ and $r\in (0,1)$,
\begin{equation*}
	  	|v_1(x')|\leq  \Big(\frac{\tilde{\varepsilon}+\tilde{\Lambda}|x'|^2}{\tilde{\varepsilon}+\tilde{\Lambda}r^2}\Big)^{\tilde{\sigma}}\|v_1\|_{L^{\infty}(\partial \mathtt{B}_r)}.
\end{equation*}
\end{lemma}
	  
\begin{proof}
A direct computation shows that for $\tilde{\eta}(x')= \tilde{\varepsilon}/\tilde{\Lambda}+|x'|^2$ and any $\tilde{\sigma}>0$,
\begin{equation*}
	  \begin{aligned}
	  			\partial_i\tilde{\eta}^{\tilde{\sigma}}= 2\tilde{\sigma}\tilde{\eta}^{\tilde{\sigma}-1}x_i, \quad 
	  			\partial_{ij}\tilde{\eta}^{\tilde{\sigma}}= 4\tilde{\sigma}(\tilde{\sigma}-1)x_ix_j\tilde{\eta}^{\tilde{\sigma}-2}+2\tilde{\sigma}\tilde{\eta}^{\tilde{\sigma}-1}\delta_i^j,
	  \end{aligned}
\end{equation*}
where $\delta_i^j$ is the Kronecker symbol. Using the assumptions on $t$ \eqref{assump_t}, we obtain
\begin{equation*}
	\begin{aligned}
	  	\mathcal{L}_{\gamma}\tilde{\eta}^{\tilde{\sigma}}=&\, \sum\limits_{i=1}^{n-1}\partial_{i}\big(\tilde{\varepsilon}+t(x')\big)\partial_{i}\tilde{\eta}(x')^{\tilde{\sigma}}+\big(\tilde{\varepsilon}+t(x')\big)\Delta\tilde{\eta}(x')^{\tilde{\sigma}}-\frac{1}{\gamma}\tilde{\eta}(x')^{\tilde{\sigma}}\\
	  	\leq&\,4\tilde{\Lambda}\tilde{\sigma} \tilde{\eta}(x')^{\tilde{\sigma}}+4\tilde{\Lambda}\tilde{\sigma}(\tilde{\sigma}-1)_{+}\tilde{\eta}(x')^{\tilde{\sigma}}+2\tilde{\Lambda}(n-1)\tilde{\sigma}\tilde{\eta}(x')^{\tilde{\sigma}}-\frac{1}{\gamma}\tilde{\eta}(x')^{\tilde{\sigma}}\\
	  	=&\,\Big(2\tilde{\Lambda}\tilde{\sigma}(n+1)+4\tilde{\Lambda}\tilde{\sigma}(\tilde{\sigma}-1)_{+}-\frac{1}{\gamma}\Big)\tilde{\eta}(x')^{\tilde{\sigma}}.
	  \end{aligned} 
\end{equation*}
Therefore, for all $ \gamma\in(0, \gamma({\tilde{\sigma}})]$,  we have $\mathcal{L}_{\gamma}\tilde{\eta}^{\tilde{\sigma}}\leq0$. 
	  	
Thus, for any $0<r<1$,
\begin{equation*}
	  \left\{ \begin{aligned}
	  			\mathcal{L}_{\gamma}\Big(v_1+\frac{\tilde{\eta}(x')^{\tilde{\sigma}}}{( \tilde{\varepsilon}/\tilde{\Lambda}+r^2)^{\tilde{\sigma}}}\|v_1\|_{L^{\infty}(\partial \mathtt{B}_r)}\Big)&\leq 0&\mbox{in}&~\mathtt{B}_r, \\
	  			\quad \quad\,\,\,v_1+\frac{\tilde{\eta}(x')^{\tilde{\sigma}}}{( \tilde{\varepsilon}/\tilde{\Lambda}+r^2)^{\tilde{\sigma}}}\|v_1\|_{L^{\infty}(\partial \mathtt{B}_r)}&\geq 0&\mbox{on}&~\partial \mathtt{B}_r,
	  \end{aligned} \right.
\end{equation*}
and
\begin{equation*}
	  \left\{ \begin{aligned}
	  	\mathcal{L}_{\gamma}\Big(v_1-\frac{\tilde{\eta}(x')^{\tilde{\sigma}}}{( \tilde{\varepsilon}/\tilde{\Lambda}+r^2)^{\tilde{\sigma}}}\|v_1\|_{L^{\infty}(\partial \mathtt{B}_r)}\Big)&\geq 0&\mbox{in}&~\mathtt{B}_r, \\
	  	\quad \quad\,\,\, v_1-\frac{\tilde{\eta}(x')^{\tilde{\sigma}}}{( \tilde{\varepsilon}/\tilde{\Lambda}+r^2)^{\tilde{\sigma}}}\|v_1\|_{L^{\infty}(\partial \mathtt{B}_r)}&\leq 0&\mbox{on}&~\partial \mathtt{B}_r.
	  \end{aligned} \right.
\end{equation*}
The result follows by applying the maximum principle to $v_1\pm\frac{\tilde{\eta}(x')^{\tilde{\sigma}}}{(\tilde{\varepsilon}/\tilde{\Lambda}+r^2)^{\tilde{\sigma}}}\|v_1\|_{L^{\infty}(\partial \mathtt{B}_r)}$.
\end{proof}

\begin{lemma}\label{epsilongammato02}
Suppose that $\|\mathtt{F}\|_{\frac{\sigma+1}{2},\mathtt{B}_r}+\|\mathtt{G}\|_{\frac{\sigma}{2},\mathtt{B}_r}<\infty$.  
 Let $\tilde{\gamma}_0$ be defined as \eqref{bargamma0}. Then for any $\gamma\in(0,\tilde\gamma_0)$ and $r\in (0,1)$,
\begin{equation*}
\|v_2\|_{L^{\infty}(\mathtt{B}_r)}\leq C\Big(\sqrt{\gamma}\|\mathtt{F}\|_{\frac{\sigma+1}{2},\mathtt{B}_r}+\gamma\|\mathtt{G}\|_{\frac{\sigma}{2},\mathtt{B}_r}\Big)(\tilde{\varepsilon}+r^2)^{\sigma/2}.
\end{equation*}
where $C$ depends only on $\tilde{\lambda}$, $\tilde{\Lambda}$, $n$ and $\sigma$, and is independent of $\tilde{\varepsilon}$ and $\gamma$.
\end{lemma}

\begin{proof}
 \textbf{Step 1.} We show that for $\rho\in(\sqrt{\tilde{\varepsilon}}, r]$,
\begin{equation}\label{step1u21}
\begin{aligned}
\|v_2\|_{L^{\infty}(\mathtt{B}_{\rho}\backslash \mathtt{B}_{\rho/8})}
\leq \|v_2\|_{L^{\infty}(\partial \mathtt{B}_{\rho})}+C\Big(\sqrt{\gamma}\|\mathtt{F}\|_{\frac{\sigma+1}{2},\mathtt{B}_r}+\gamma\|\mathtt{G}\|_{\frac{\sigma}{2},\mathtt{B}_r}\Big)(\tilde{\varepsilon}+\rho^2)^{\sigma/2}.
\end{aligned}
\end{equation}

To prove this, we decompose 
$$v_2:=v_{21}+v_{22}\quad\mbox{in}~ \mathtt{B}_{\rho},$$ where 
\begin{equation*}
\left\{ \begin{aligned}
 \mathcal{L}_{\gamma}v_{21}&=0&\mbox{in}&~\mathtt{B}_{\rho}, \\
v_{21}&=v_2&\mbox{on}&~\partial \mathtt{B}_{\rho}, 
\end{aligned} \right.
\quad\mbox{and}\quad 
\left\{ \begin{aligned}
 \mathcal{L}_{\gamma}v_{22}&=\sum\limits_{i=1}^{n-1}\partial_i\mathtt{F}^i+\mathtt{G}&\mbox{in}&~\mathtt{B}_{\rho}, \\
 v_{22}&=0&\mbox{on}&~\partial \mathtt{B}_{\rho}.
 \end{aligned} \right.
\end{equation*}
The claim then follows from the maximum principle and a rescaling argument.

First, the maximum principle gives
\begin{equation}\label{step1v1}
\|v_{21}\|_{L^{\infty}(\mathtt{B}_{\rho})}\leq \|v_2\|_{L^{\infty}(\partial \mathtt{B}_{\rho})}.
\end{equation}
For $v_{22}$, we rescale by setting $x'=\rho y'$ for $y'\in \mathtt{B}_{1}\subset \bR^{n-1}$ and define $$\tilde{v}_{22}(y')=v_{22}(\rho y'),\quad\hat{\varepsilon}=\frac{\tilde{\varepsilon}}{\rho^2},\quad \tilde{t}(y')=\frac{1}{\rho^2}t(\rho y'),\quad \tilde{\mathtt{F}}(y')=\frac{1}{\rho}\mathtt{F}(\rho y'),\quad\mbox{and}~ \tilde{\mathtt{G}}(y')=\mathtt{G}(\rho y').$$
Then $\tilde{v}_{22}\in H_0^1(\mathtt{B}_{1})$ satisfies
\begin{equation*}
 \sum\limits_{i=1}^{n-1}\partial_{i} \Big((\hat{\varepsilon}+\tilde{t}(y'))\partial_{i}\tilde{v}_{22}\Big)-\frac{1}{\gamma}\tilde{v}_{22}=\dv \tilde{\mathtt{F}}+\tilde{\mathtt{G}}\quad\mbox{in}~\mathtt{B}_{1}\subset \bR^{n-1}.
\end{equation*}
Multiplying both sides by $\tilde{v}_{22}$ and integrating by parts gives
\begin{equation*}
\int_{\mathtt{B}_{1}} (\hat{\varepsilon}+\tilde{t}(y'))|D \tilde{v}_{22}|^2+\frac{1}{\gamma}\tilde{v}_{22}^2dy'=\int_{\mathtt{B}_{1}} \tilde{\mathtt{F}}^i\partial_{i}{\tilde v}_{22}-\tilde{\mathtt{G}}\tilde v_{22}dy'.
\end{equation*}
By H\"{o}lder's inequality, 
\begin{equation*}
\int_{\mathtt{B}_{1}}\Big( (\hat{\varepsilon}+\tilde{t}(y'))|D {\tilde v}_{22}|^2+\frac{1}{\gamma} {\tilde v}_{22}^2\Big)dy'\leq C\int_{\mathtt{B}_{1}}\frac{|\tilde{\mathtt{F}}|^2}{\hat{\varepsilon}+\tilde{t}(y')}+\gamma\tilde{\mathtt{G}}^2dy'.
\end{equation*}
Multiplying both sides by $\gamma$ and using the bounds on $t$, we have
\begin{equation}\label{step1tildev2}
 \| {\tilde v}_{22}\|_{L^2(\mathtt{B}_{1})}\leq C\Big(\sqrt{\gamma}\,\Big\|\frac{\tilde{\mathtt{F}}}{\sqrt{ \hat{\varepsilon}+|y'|^2}}\Big\|_{L^{\infty}(\mathtt{B}_{1})}+\gamma\,\|\tilde{\mathtt{G}}\|_{L^{\infty}(\mathtt{B}_{1})}\Big).
\end{equation}
By using a flattening transform, and Corollary \ref{corbdgammato0} and the corresponding interior estimates on a standard ball, together with \eqref{step1tildev2}, we obtain
\begin{equation*}
\begin{aligned}
\|{\tilde v}_{22}\|_{L^{\infty}(\mathtt{B}_{1}\backslash \mathtt{B}_{1/8})}\leq&\,C \|\tilde{v}_{22}\|_{L^2(\mathtt{B}_{1}\backslash \mathtt{B}_{1/16})}+C\big(\sqrt{\gamma}\|\tilde{\mathtt{F}}\|_{L^{\infty}(\mathtt{B}_{1}\backslash \mathtt{B}_{1/16})}+\gamma\|\tilde{\mathtt{G}}\|_{L^{\infty}(\mathtt{B}_{1}\backslash \mathtt{B}_{1/16})}\big)\\
\leq&\,C\Big(\sqrt{\gamma}\,\Big\|\frac{\tilde{\mathtt{F}}}{\sqrt{ \hat{\varepsilon}+|y'|^2}}\Big\|_{L^{\infty}(\mathtt{B}_{1})}+\gamma\,\|\tilde{\mathtt{G}}\|_{L^{\infty}(\mathtt{B}_{1})}\Big).
\end{aligned}
\end{equation*}
Returning to $v_{22}$, we have 
\begin{equation*}
\|v_{22}\|_{L^{\infty}(\mathtt{B}_{\rho}\backslash \mathtt{B}_{\rho/8})}\leq C\Big(\sqrt{\gamma}\|\mathtt{F}\|_{\frac{\sigma+1}{2},\mathtt{B}_r}+\gamma\|\mathtt{G}\|_{\frac{\sigma}{2},\mathtt{B}_r}\Big)(\tilde{\varepsilon}+\rho^2)^{\sigma/2}.
\end{equation*}
Combining this with \eqref{step1v1} proves \eqref{step1u21}.

\textbf{Step 2.} Using a covering argument and iteration, we prove 
\begin{equation*}
\|v_2\|_{L^{\infty}(\mathtt{B}_{r}\backslash \mathtt{B}_{\sqrt{\tilde{\varepsilon}}}))}\leq C\Big(\sqrt{\gamma}\|\mathtt{F}\|_{\frac{\sigma+1}{2},\mathtt{B}_r}+\gamma\|\mathtt{G}\|_{\frac{\sigma}{2},\mathtt{B}_r}\Big)(\tilde{\varepsilon}+r^2)^{\sigma/2}.
\end{equation*}

Since $v_2\in H_0^1(\mathtt{B}_r)$, \eqref{step1u21} gives
\begin{equation}\label{step2u21}
\|v_2\|_{L^{\infty}(\mathtt{B}_{r}\backslash \mathtt{B}_{r/4})}\leq C\Big(\sqrt{\gamma}\|\mathtt{F}\|_{\frac{\sigma+1}{2},\mathtt{B}_r}+\gamma\|\mathtt{G}\|_{\frac{\sigma}{2},\mathtt{B}_r}\Big)(\tilde{\varepsilon}+r^2)^{\sigma/2}.
\end{equation}
From the classical theory for elliptic equations, $v_2\in C^{\beta}(\mathtt{B}_r)$ for some $\beta\in(0,1)$. For any $\rho\in(\sqrt{\tilde{\varepsilon}},\frac{r}{4})$, we choose $\rho_0=\rho$, $\rho_{i}=2^{i}\rho$, $i=1,2,\ldots,k$, and $k$ such that $\frac{r}{4}\leq2^{k-1}\rho_0<2^k\rho_0<r$. By \eqref{step1u21},
\begin{equation}\label{iteru2}
\|v_2\|_{L^{\infty}(\partial \mathtt{B}_{\rho_i})}\leq \|v_2\|_{L^{\infty}(\partial \mathtt{B}_{\rho_{i+1}})}+C\Big(\sqrt{\gamma}\|\mathtt{F}\|_{\frac{\sigma+1}{2},\mathtt{B}_r}+\gamma\|\mathtt{G}\|_{\frac{\sigma}{2},\mathtt{B}_r}\Big)\rho_{i+1}^{\sigma}.
\end{equation}
Iterating \eqref{iteru2} from $i=0$ to $k-1$, and using \eqref{step2u21}, we obtain
\begin{equation*}
\begin{aligned}
\|v_2\|_{L^{\infty}(\partial \mathtt{B}_{\rho_0})} \leq&\,\|v_2\|_{L^{\infty}(\partial \mathtt{B}_{2^k\rho_0})}+C\Big(\sqrt{\gamma}\|\mathtt{F}\|_{\frac{\sigma+1}{2},\mathtt{B}_r}+\gamma\|\mathtt{G}\|_{\frac{\sigma}{2},\mathtt{B}_r}\Big)\sum_{i=1}^{k}(2^i\rho_0)^{\sigma}\\
\leq&\,Cr^{\sigma}\Big(\sqrt{\gamma}\|\mathtt{F}\|_{\frac{\sigma+1}{2},\mathtt{B}_r}+\gamma\|\mathtt{G}\|_{\frac{\sigma}{2},\mathtt{B}_r}\Big).
\end{aligned}
\end{equation*}
Combining this with \eqref{step2u21} gives
\begin{equation}\label{step2u233}
\|v_2\|_{L^{\infty}(\mathtt{B}_r\backslash \mathtt{B}_{\sqrt{\tilde{\varepsilon}}})}\leq C \Big(\sqrt{\gamma}\|\mathtt{F}\|_{\frac{\sigma+1}{2},\mathtt{B}_r}+\gamma\|\mathtt{G}\|_{\frac{\sigma}{2},\mathtt{B}_r}\Big)(\tilde{\varepsilon}+r^2)^{\sigma/2}.
\end{equation}

\textbf{Step 3.} We estimate $v_2$ in $\mathtt{B}_{\sqrt{\tilde{\varepsilon}}}$ via a similar decomposition and scaling.

Decompose $$v_2:=v_{21}+v_{22}\quad\mbox{in}~ \mathtt{B}_{2\sqrt{\tilde{\varepsilon}}},$$ as in Step 1, with $\rho=2\sqrt{\tilde{\varepsilon}}$.  For $v_{21}$, we apply the maximum principle. For $v_{22}$, we utilize Corollary \ref{corbdgammato0}, the corresponding interior estimates on a standard ball, and a scaling argument. By combining the estimates for $v_{21}$ and  $v_{22}$, we derive
\begin{equation*}
\|v_2\|_{L^{\infty}(\mathtt{B}_{\sqrt{\tilde{\varepsilon}}})}\leq \|v_2\|_{L^{\infty}(\partial \mathtt{B}_{2\sqrt{\tilde{\varepsilon}}})}+C\Big(\sqrt{\gamma}\|\mathtt{F}\|_{\frac{\sigma+1}{2},\mathtt{B}_r}+\gamma\|\mathtt{G}\|_{\frac{\sigma}{2},\mathtt{B}_r}\Big)\tilde{\varepsilon}^{\sigma/2}.
\end{equation*}
By using \eqref{step2u233},
$$\|v_2\|_{L^{\infty}(\mathtt{B}_{\sqrt{\tilde{\varepsilon}}})}\leq C\Big(\sqrt{\gamma}\|\mathtt{F}\|_{\frac{\sigma+1}{2},\mathtt{B}_r}+\gamma\|\mathtt{G}\|_{\frac{\sigma}{2},\mathtt{B}_r}\Big)\tilde{\varepsilon}^{\sigma/2}.$$
This, together with \eqref{step2u233}, completes the proof.
\end{proof}

\begin{proof}[Proof of Proposition \ref{epsilongammato0}]
Without loss of generality, we assume $\|v\|_{L^{\infty}(\partial \mathtt{B}_{1})}+\sqrt{\gamma}\|\mathtt{F}\|_{\mathtt{B}_{1}, \frac{\sigma+1}{2}}+\gamma\|\mathtt{G}\|_{\mathtt{B}_{1}, \sigma/2}\leq 1.$ 

Let $\tilde{\gamma}_0$ and $\gamma({\tilde{\sigma}})$ be defined as in \eqref{bargamma0} and \eqref{gammatildesigma}, respectively. One can easily verify that $\tilde{\gamma}_0 = \gamma({\frac{\sigma}{2}})$ and that $\gamma({\tilde{\sigma}})$ is strictly decreasing with respect to $\tilde{\sigma}$. Therefore, for any $\gamma \in (0, \tilde{\gamma}_0)$, there exists $\bar{\sigma} > \frac{\sigma}{2}$ such that $\gamma \in (0, \gamma({\bar{\sigma}})]$.

Then by Lemmas \ref{epsilongammato01} and \ref{epsilongammato02}, for $0<\rho<r<1$, we obtain
\begin{equation}\label{estw1w2}
\begin{aligned}
\|v\|_{L^{\infty}(\partial \mathtt{B}_{\rho})}\leq&\,\|v_1\|_{L^{\infty}(\partial \mathtt{B}_{\rho})}+\|v_2\|_{L^{\infty}(\partial \mathtt{B}_{\rho})}
\leq\, \Big(\frac{\tilde{\varepsilon}+\tilde{\Lambda}\rho^2}{\tilde{\varepsilon}+\tilde{\Lambda}r^2}\Big)^{\bar{\sigma}} \|v\|_{L^{\infty}(\partial \mathtt{B}_r)}+C (\tilde{\varepsilon}+r^2)^{\sigma/2}.
\end{aligned}
\end{equation}
For $\rho\in(\sqrt{\tilde{\varepsilon}},r)$, this implies
\begin{equation}\label{iterationepsilongammato0}
\|v\|_{L^{\infty}(\partial \mathtt{B}_{\rho})}\leq \Big(\frac{(\tilde{\Lambda}+1)\rho^2}{\tilde{\Lambda}r^2}\Big)^{\bar{\sigma}}\|v\|_{L^{\infty}(\partial \mathtt{B}_r)}+Cr^{\sigma}.
\end{equation} 
Fix $\mu\in(0, 1)$ such that $\frac{\tilde{\Lambda}+1}{\tilde{\Lambda}}\mu^{2\bar{\sigma}}\leq \frac{1}{2}\mu^{\sigma}$. Iterating \eqref{iterationepsilongammato0} along the scales $\rho=\mu^{i+1}$ and $r=\mu^i$ in \eqref{iterationepsilongammato0} from $i=0$ to $k-1$, with $\rho_k:=\mu^k\in(\sqrt{\tilde{\varepsilon}},1)$, yields
\begin{equation}\label{iteraw11}
\begin{aligned}
\|v\|_{L^{\infty}(\partial \mathtt{B}_{\rho_{k}})} 
\leq&\,\frac{1}{2^k}\mu^{k\sigma}\|v\|_{L^{\infty}(\partial \mathtt{B}_{1})}+C\sum_{i=1}^{k} (\frac{1}{2}\mu^{\sigma})^{i-1}(\mu^{k-i})^{\sigma}
\leq\,C\mu^{k\sigma}.
\end{aligned}
\end{equation}

For $|x'|\in(\sqrt{\tilde{\varepsilon}},1)$, choose $k$ such that $\mu^{k+1}<|x'|\leq\mu^k$, then by \eqref{iterationepsilongammato0} and \eqref{iteraw11}, 
\begin{equation}\label{estimatew}
|v(x')|\leq \|v\|_{L^{\infty}(\partial \mathtt{B}_{|x'|})}\leq C\|v\|_{L^{\infty}(\partial \mathtt{B}_{\mu^k})}+C\mu^{k\sigma}\leq C|x'|^{\sigma}\leq C(\tilde{\varepsilon}+|x'|^2)^{\sigma/2}.
\end{equation}
For $x'\in \mathtt{B}_{\sqrt{\tilde{\varepsilon}}}$, we use \eqref{estw1w2} and \eqref{estimatew} to obtain 
\begin{equation*}
\|v\|_{L^{\infty}(\mathtt{B}_{\sqrt{\tilde{\varepsilon}}})}\leq C\underset{|x'|\leq \sqrt{\tilde{\varepsilon}}}{\sup}\|v\|_{L^{\infty}(\partial \mathtt{B}_{|x'|})}\leq \|v\|_{L^{\infty}(\partial \mathtt{B}_{2\sqrt{\tilde{\varepsilon}}})}+C{\tilde{\varepsilon}}^{\sigma/2}\leq C{\tilde{\varepsilon}}^{\sigma/2}.
\end{equation*}
Combining these estimates completes the proof of Proposition \ref{epsilongammato0}.
\end{proof}

\subsection{Pointwise gradient estimates for degenerate equations}

\begin{prop}\label{prop_gradient_v}
 Let  $v \in H^1(\mathtt{B}_{1})$  be a solution to \eqref{equ_v}, and let  $t(x')$  satisfy \eqref{assump_t}. Suppose  $\mathtt{F} \in C^{\alpha}(\mathtt{B}_{1})$ ,  $\mathtt{G} \in L^{\infty}(\mathtt{B}_{1})$ , and  $\|\mathtt{F}\|_{\frac{\sigma+1}{2},\mathtt{B}_{1}} + \|\mathtt{G}\|_{\frac{\sigma}{2},\mathtt{B}_{1}} < \infty$ . For any  $\sigma > 0$ , let  $\tilde{\gamma}_0$  be as defined in \eqref{bargamma0}. Then, for all  $x' \in \mathtt{B}_{\frac{1}{2}}$ , the estimate
\begin{equation*}
\begin{aligned}
|D v(x')| \leq\, &\, C\tilde{\eta}(x')^{\frac{\sigma-1}{2}}\Big(\|v\|_{L^{\infty}(\partial \mathtt{B}_{1})} + \sqrt{\gamma}\|\mathtt{F}\|_{\frac{\sigma+1}{2},\mathtt{B}_{1}} + \gamma\|\mathtt{G}\|_{\frac{\sigma}{2},\mathtt{B}_{1}}\Big)\\
&+ C\tilde{\eta}(x')^{-1+\frac{\alpha}{2}}\gamma^{\frac{\alpha}{2}}[\mathtt{F}]_{C^{\alpha}(\mathtt{B}_{\frac{1}{8}\sqrt{\tilde{\eta}(x')}}(x'))} + C\gamma^{1/2}\tilde{\eta}(x')^{-1/2}\|\mathtt{G}\|_{L^{\infty}(\mathtt{B}_{\frac{1}{8}\sqrt{\tilde{\eta}(x')}}(x'))}
\end{aligned}
\end{equation*}
holds for all  $0 < \gamma < \tilde\gamma_0$ . Here, $\tilde\eta(x')=\tilde{\varepsilon}/\tilde{\Lambda}+|x'|^2$ and the constant C depends only on  $n$, $\tilde{\lambda}$, $\tilde{\Lambda}$, $\sigma$ and a lower bound of $\tilde{\gamma}_0-\gamma$.
\end{prop}
\begin{proof}
 For any $x_0'\in \mathtt{B}_{1/2}$, we use a change of variables in $\mathtt{B}_{\frac{1}{8}\sqrt{\tilde{\eta}(x_0')}}(x_0)$ by setting 
\begin{equation*}
x'=x_0'+\frac{1}{8}\sqrt{\tilde{\eta}(x_0')}y',\quad y'\in \mathtt{B}_{1},
\end{equation*}
and set $\tilde{v}(y)=v(x_0'+\frac{1}{8}\sqrt{\tilde{\eta}(x_0')}y')$. Then $\tilde{v}(y')$ solves 
$$\sum\limits_{i=1}^{n-1}\partial_i(\tilde{a}^{ii}\partial_i \tilde{v})-\frac{1}{\gamma}\tilde{v}=\sum\limits_{i=1}^{n-1}\partial_i\tilde{\mathtt{F}}^i+\tilde{\mathtt{G}}\quad \mbox{in}~\mathtt{B}_{1},$$
where $\tilde{a}^{ii}(y')=\frac{64}{\tilde{\eta}(x_0')}\delta(x_0'+\frac{1}{8}\sqrt{\tilde{\eta}(x_0')}y')$, and
$$\tilde{\mathtt{F}}^i(y')=\frac{8}{\sqrt{\tilde{\eta}(x_0')}}\mathtt{F}^i(x_0'+\frac{1}{8}\sqrt{\tilde{\eta}(x_0')}y'), \quad \tilde{\mathtt{G}}(y')=\mathtt{G}(x_0'+\frac{1}{8}\sqrt{\tilde{\eta}(x_0')}y').$$
Applying Corollary~ \ref{corthmgammato0}~(ii) to $\tilde{v}$ and rescaling back, we have 
\begin{equation*}
\begin{aligned}
|D v(x_0')|\leq&\,C\tilde{\eta}(x_0')^{-1/2}\|v\|_{L^{\infty}(\mathtt{B}_{\frac{1}{8}\sqrt{\tilde{\eta}(x_0')}}(x_0'))}\\
&+C\tilde{\eta}(x_0')^{-1+\frac{\alpha}{2}}\gamma^{\frac{\alpha}{2}}[\mathtt{F}]_{C^{\alpha}(\mathtt{B}_{\frac{1}{8}\sqrt{\tilde{\eta}(x_0')}}(x_0'))}+C\gamma^{1/2}\tilde{\eta}(x_0')^{-1/2}\|\mathtt{G}\|_{L^{\infty}(\mathtt{B}_{\frac{1}{8}\sqrt{\tilde{\eta}(x_0')}}(x_0'))}.
\end{aligned}
\end{equation*}
Combining with Proposition \ref{epsilongammato0}, we complete the proof of Proposition \ref{prop_gradient_v}.
\end{proof}

\section{Gradient estimates in the narrow region}\label{sec_4}

In this section, we establish gradient estimates for solutions to the nonhomogeneous equation in the narrow region $\Omega_{2R}$, subject to nonvanishing boundary conditions:
\begin{equation}\label{equw}
\left\{ \begin{aligned}
\Delta {w}&=\mathfrak{R}(x)&\mbox{in}&~\Omega_{2R},\\
{w}+\gamma\partial_{\nu}{w}&=\Psi_1(x')&\mbox{on}&~\Gamma_{1,2R},\\
{w}+\gamma\partial_{\nu}{w}&=\Psi_2(x')&\mbox{on}&~\Gamma_{2,2R},
\end{aligned} \right.
\end{equation}
where $\nu$ is the unit normal vector, pointing upward on $\Gamma_{1,2R}$ and downward on $\Gamma_{2,2R}$. The boundary profiles $f_{1}$ and $f_{2}$ satisfy assumptions \eqref{fg0}--\eqref{fg}. We show that under suitable growth conditions on the right-hand side and boundary data, the gradient remains bounded by a constant independent of both $\varepsilon$ and $\gamma$.

The proof of gradient boundedness proceeds via a case analysis based on the conductivity regime. In the low-conductivity regime $\gamma \leq \mu \eta(x')$, we combine an iterative estimate for $w$ with a rescaled version of Theorem \ref{gradientsmallgamma3}, directly yielding a gradient bound. In the high-conductivity regime ($\gamma > \mu \eta(x')$), a more refined approach is adopted. Inspired by studies of Neumann boundary conditions \cite{DLY}, we employ a dimension reduction technique, examining the vertical average $\bar{w}$ of $w$. This average satisfies an elliptic equation with degenerate principal coefficients and lower-order terms of order $1/\gamma$ for small $\gamma$. The regularity theory from Section \ref{sec_3} then provides a gradient estimate for $\bar{w}$ in this context. The final boundedness of $Dw$ in $\Omega_R$ is achieved by synthesizing these case-specific estimates. This systematic division into regimes, along with the tailored methods applied in each case, yields the global gradient bound. Here, $\mu$ is a fixed large constant, to be specified in Subsection \ref{subsectionboundedgra}.

To streamline notation, we introduce a weighted H\"older norm. For a domain $\mathcal{D}\subset \bR^n$ with diameter $d = \text{diam}\, \mathcal{D}$, define
\begin{equation*}
\begin{aligned}
\|\Psi\|^*_{C^{\alpha}(\mathcal{D})}:=&\,\|\Psi\|_{L^{\infty}(\mathcal{D})}+d^{\alpha}[\Psi]_{C^{\alpha}(\mathcal{D})}, \\ 
\|\Psi\|^*_{C^{1,\alpha}(\mathcal{D})}:=&\,\|\Psi\|_{L^{\infty}(\mathcal{D})}+d\,\|\nabla\Psi\|_{L^{\infty}(\mathcal{D})}+d^{1+\alpha}\underset{1\leq i\leq n}{\sup}[\partial_i\Psi]_{C^{\alpha}(\mathcal{D})}.
\end{aligned}
\end{equation*}

We recall several facts and fix notation. Let $\Omega_r(x_0)$ be as in \eqref{definition_omegat}, and set $\eta(x') := \varepsilon + |x'|^2$ and $\delta(x') := \varepsilon + f_1(x') - f_2(x')$. Under assumptions \eqref{fg0}--\eqref{fg}, we have $\eta(x') \sim \delta(x')$ and
$$|Df_{1}(x')|+|Df_{2}(x')|\leq C\eta(x')^{1/2}$$ for all $x\in\Omega_{R}$. We may assume without loss of generality that $R \leq 1/32$. Under these conditions, for $\varepsilon \leq \frac{1}{4}R^2$, the following inclusions and comparisons hold:
 \begin{itemize}
 \item[---] $\Omega_{\eta(x')}(x)\subset \Omega_{R}$ for all $x\in \Omega_{\frac{3}{4}R}$;
 \item[---] $\Omega_{\eta(x')}(x)\subset \Omega_{\frac{1}{16}\sqrt{\eta(x')}}(x)$ for all $x\in\Omega_ {\frac{3}{4}R}$;
 \item[---] $\Omega_{\frac{1}{8}\sqrt{\eta(x')}}(x)\subset\Omega_{\frac{3}{4}R}$ for all $x\in\Omega_{\frac{R}{2}}$;
 \item[---] for any $x\in \Omega_{R}$ and $x_0\in \Omega_{\frac{1}{8}\sqrt{\eta(x')}}(x)$, we have $\eta(x')\sim\eta(x_0')$.
\end{itemize}
    
\begin{theorem}\label{boundedgradientw}
Let $w \in H^1(\Omega_{2R})$ be a solution to \eqref{equw} with boundary profiles $f_1$ and $f_2$ satisfying \eqref{fg}.  Assume the normalization conditions :
$$\|w\|_{L^{\infty}(\Omega_{2R})} \leq 1,  $$
and 
\begin{equation}\label{estH}
	|\mathfrak{R}(x_0)| \leq \frac{1}{\gamma + \delta(x_0')}, \quad
	\|\Psi_i\|^*_{C^{1,\alpha}(\Omega_{\frac{3}{4}\eta(x_0')}(x_0))} \leq \frac{\gamma\delta(x_0')}{\gamma + \delta(x_0')}, 
	\quad  i=1,2,\,\, \forall x_0 \in \Omega_R.
\end{equation}
Let $\gamma_0$ be given by \eqref{gamma00}. Then, for any $\varepsilon \in (0, \varepsilon_0)$ and $\gamma \in (0, \gamma_0)$, there holds
\begin{equation}\label{equboundedgradientw}
	  \|Dw\|_{L^{\infty}(\Omega_{R/2})} \leq C.
\end{equation}
Here, $\varepsilon_0 > 0$ and $C > 0$ depend only on $n$, $R$, $\alpha$, $\kappa$, a lower bound of $\gamma_0-\gamma$ and upper bounds of $\|f_{1}\|_{C^{2,\alpha}}$ and $\|f_{2}\|_{C^{2,\alpha}}$.
\end{theorem}

\begin{remark}
Under the hypotheses of Theorem \ref{boundedgradientw}, for a sufficiently large constant $\mu$, the quantity $\eta(x_0')^{-1}\sum_{i=1}^2\|\Psi_i\|^*_{C^{1,\alpha}(\Omega_{\frac{3}{4}\eta(x_0')}(x_0))} 
+\eta(x_0')\|\mathfrak{R}\|_{L^{\infty}(\Omega_{\frac{3}{4}\eta(x_0')}(x_0))}$ on the right-hand side of \eqref{C1allgamma} below remains bounded when $\gamma\leq\mu\eta(x_0')$. Conversely, when $\gamma>\mu\eta(x_0')$, the expression $\frac{1}{\gamma}\sum_{i=1}^2\|\Psi_i\|^*_{C^{\alpha}(\Omega_{\frac{3}{4}\eta(x_0')}(x_0))}+\eta(x_0')\|\mathfrak{R}\|_{L^{\infty}(\Omega_{\frac{3}{4}\eta(x_0')}(x_0))}\sim\frac{\eta(x_0')}{\gamma}$ on the right-hand side of \eqref{C1largegamma} is much smaller than $1$. This dichotomy is essential for establishing the uniform gradient bound in Theorem \ref{boundedgradientw}. 
\end{remark}

In the next two subsections, we flatten the boundary via the coordinate transformation
\begin{equation}\label{variableschange}
\left\{ \begin{aligned}
y'&=x',\\
y_n&=2\eta(x_0')\Big(\frac{x_n-f_2(x')}{\varepsilon+f_1(x')-f_2(x')}-\frac{1}{2}\Big),
\end{aligned} \right.
\end{equation}
which maps the domain $\Omega_{s}(x_0)$ to the cylinder $Q_{s,\eta(x_0')}(x_0)$ of height $\eta(x_0')$, where
\begin{equation*}
Q_{s,t}(x_0):=\{y=(y',y_n)\in\bR^n~\big|~|y'-x_0'|<s,~|y_n|<t~\}.
\end{equation*}
Denote the upper and lower boundaries of the cylinder by $\Gamma_{s,t}^{\pm}(x_0)$. If $w$ is a solution to \eqref{equw}, then the function $\tilde{w}(y)=w(x)$ satisfies
\begin{equation}\label{equflatten}
\left\{\begin{aligned}
\partial_{i}(a^{ij}(y)\partial_{j}\tilde{w}(y))&=\tilde{\mathfrak{R}}(y)& \mbox{in}&~Q_{s,\eta(x_0')}(x_0),\\
h_1(y')\tilde{w}(y)+\gamma a^{nj}(y)\partial_{j}\tilde{w}(y)&= \tilde{\Psi}_{1}(y')& \mbox{on}&~\Gamma_{s,\eta(x_0')}^{+}(x_0),\\
h_2(y')\tilde{w}(y)-\gamma a^{nj}(y)\partial_{j}\tilde{w}(y)&= \tilde{\Psi}_{2}(y')& \mbox{on}&~\Gamma_{s,\eta(x_0')}^{-}(x_0),
\end{aligned}
\right.
\end{equation}
where the coefficients and data are given by:
$$a^{ij}(y)=\sum_{k=1}^{n}\frac{\partial y_i}{\partial x_{k}}\frac{\partial y_j}{\partial x_{k}}(\det \partial_{x}y)^{-1},\quad\tilde{\mathfrak{R}}(y)=\mathfrak{R}(x)(\det \partial_{x}y)^{-1},$$ $$h_i(y')=\sqrt{1+|\nabla_{y'}f_i(y')|^2}, \quad\mbox{and}~ \tilde{\Psi}_{i}(y')=\sqrt{1+|\nabla_{y'}f_i(y')|^2}\Psi_{i}(y')\,\,\mbox{for}~ i=1,2.$$ 
The coefficients $a^{ij}(y)$ can be computed explicitly as:
\begin{equation*}
\begin{aligned}
&a^{ii}(y)=\frac{\delta(y')}{2\eta(x_0')}\quad \mbox{for}~1\leq i\leq n-1;\quad\quad\,a^{ij}=0\quad \mbox{for}~1\leq i,j\leq n-1,\,i\neq j;\\
&a^{in}(y)=a^{ni}(y)=\partial_{i}f_{2}(y')\frac{y_n-\eta(x_0')}{2\eta(x_0')}-\partial_{i}f_{1}(y')\frac{y_n+\eta(x_0')}{2\eta(x_0')}\quad \mbox{for}~1\leq i\leq n-1;\\
&a^{nn}(y)=\frac{2\eta(x_0')}{\delta(y')}+\sum_{i=1}^{n-1}\frac{((\partial_{i}f_2)(y_n-\eta(x_0'))-(\partial_{i}f_1)(y_n+\eta(x_0')))^2}{2\eta(x_0')\delta(y')}.
\end{aligned}
\end{equation*}
 A direct computation shows that for $y\in Q_{\frac{1}{4}\sqrt{\eta(x_0')},\eta(x_0')}(x_0)$:
\begin{equation}\label{aijtransformed}
\begin{aligned}
 \frac{1}{C}\leq a^{ii}(y)\leq C \,~\mbox{for}&~1\leq i\leq n;\quad\,|a^{in}(y)|=|a^{ni}(y)| \leq C\eta(x_0')^{1/2}\, ~\mbox{for}~1\leq i\leq n-1;\\
 & |D\tilde{w}(y)|\sim |Dw(x)|;\quad\quad |\tilde{\mathfrak{R}}(y)|\sim|\mathfrak{R}(x)|;
\end{aligned}
\end{equation}
and 
\begin{equation}\label{aijtransformed1}
\begin{aligned}
 &[a^{ij}]_{C^{\alpha}(Q_{\frac{1}{4}\sqrt{\eta(x_0')},\eta(x_0')}(x_0))}\leq C\eta(x_0')^{-\alpha/2};\\
 &\|h_i\|_{C^{\alpha}(Q_{\frac{1}{4}\sqrt{\eta(x_0')},\eta(x_0')}(x_0))}\leq C\quad \mbox{for}~i=1,2;\\
 &\|\Psi_i\|^{*}_{C^{1,\alpha}(\Omega_{s,\eta(x_0')}(x_0))}\sim \|\tilde{\Psi}_i\|^{*}_{C^{1,\alpha}(Q_{ s,\eta(x_0')}(x_0))} \quad \mbox{for}~0<s<\frac{1}{4}\sqrt{\eta(x_0')}.
\end{aligned}
\end{equation}
For further details, we refer to Page 20 of \cite{DYZ}.

\subsection{Proof of Theorem \ref{boundedgradientw} for the case $\gamma\leq\mu\eta(x')$.}

The proof in this regime relies on the following two lemmas. For any domain $\mathcal{D}$, we denote the averaged $L^p$ norm of $u$ over $\mathcal{D}$ by
$$
\dashnorm{u}{L^p(\cD)}:= \left( \fint_{\cD} |u|^p \right)^{1/p}.
$$

\begin{lemma}\label{lemmac1alphaest1}
Let $w\in H^1(\Omega_{\eta(x_0')}(x_0))$ be a solution to the boundary value problem
\begin{equation*}
\left\{ \begin{aligned}
\Delta w&= \mathfrak{R}(x)&\mbox{in}&~ \Omega_{\eta(x_0')}(x_0), \\
w+\gamma \partial_{\nu}w&=\Psi_1(x)&\mbox{on}&~\Gamma_{1,\eta(x_0')}(x_0), \\
w+\gamma\partial_{\nu}w&=\Psi_2(x)&\mbox{on}&~\Gamma_{2,\eta(x_0')}(x_0),
\end{aligned} \right.
\end{equation*}
with $\gamma>0$ and $f_1,f_2$ satisfying \eqref{fg0}--\eqref{fg}. Assume $\mathfrak{R}\in L^{\infty}(\Omega_{\eta(x_0')}(x_0))$ and $\Psi_i\in C^{1,\alpha}(\Omega_{\eta(x_0')}(x_0))$, $i=1,2$. Then
\begin{equation}\label{C1allgamma}
\begin{aligned}
\|Dw\|_{L^{\infty}(\Omega_{\frac{1}{2}\eta(x_0')}(x_0))}
\leq&\,C\Big(\eta(x_0')^{-1}\dashnorm{w}{L^2(\Omega_{\frac{3}{4}\eta(x_0')}(x_0))}+\eta(x_0')^{-1}\sum_{i=1}^2\|\Psi_i\|^*_{C^{1,\alpha}(\Omega_{\frac{3}{4}\eta(x_0')}(x_0))}\\
&\quad\quad
+\eta(x_0')\|\mathfrak{R}\|_{L^{\infty}(\Omega_{\frac{3}{4}\eta(x_0')}(x_0))}\Big),
\end{aligned}
\end{equation} 
where the constant $C$ depends only on $n, \alpha, \kappa, \|f_{1}\|_{C^{2,\alpha}(B_R')}$, and $\|f_{2}\|_{C^{2,\alpha}(B_R')}$.
\end{lemma}
 
\begin{proof}
We flatten the domain via the change of variables given in \eqref{variableschange}. Let $\tilde{w}(y)=w(x)$. Then $\tilde{w}$ satisfies \eqref{equflatten}. Introduce the rescaled variables 
\begin{equation*}
\left\{ \begin{aligned}
y'&=x_0'+\eta(x_0')z', \\
y_n&= \eta(x_0')z_n,
\end{aligned} \right.
\end{equation*}
which maps $Q_{\eta(x_0'),\eta(x_0')}(x_0)$ onto $Q_{1,1}(0)$.
The rescaled function $\hat{w}(z)=\tilde{w}(y)$ satisfies 
 \begin{equation*} 
\left\{ \begin{aligned}
 \partial_{i}(\hat{a}^{ij}( z)\partial_{j}\hat{w}( z))&=\hat{\mathfrak{R}}(z) &\mbox{in}&~Q_{1,1}(0),\\
\hat{h}_{1}(z)\hat{w}(z) +\hat{\gamma} \hat{a}^{nj}( z)\partial_{j} \hat{w}(z)&= \hat{\Psi}_1(z') &\mbox{on}&~\Gamma_{1,1}^{+}( 0),\\
\hat{h}_{2}(z)\hat{w}(z) -\hat{\gamma} \hat{a}^{nj}( z)\partial_{j} \hat{w}(z)&= \hat{\Psi}_2(z') &\mbox{on}&~\Gamma_{1,1}^{-}(0),
\end{aligned} \right.
\end{equation*}
where
\begin{equation*}
\begin{aligned}
 &\hat{a}^{ij}(z')=a^{ij}(y'),\quad\hat{\gamma}=\frac{\gamma}{\eta(x_0')}, \quad \hat{h}_{i}(z')=h_i(y')\,\,\,\mbox{for}~i=1,2,\\
&\hat{\mathfrak{R}}(z)=\eta(x_0')^2\tilde{\mathfrak{R}}(y),\quad\quad\mbox{and}~\hat{\Psi}_i(z')=\tilde{\Psi}_{i}(y')\,\,\,\mbox{for}~i=1,2. 
\end{aligned}
\end{equation*}
 
For any $z_0\in \Gamma^{\pm}_{\frac{1}{2},1}(0)$, applying the uniform estimate from Theorem \ref{gradientsmallgamma3} to $\hat{w}$ in $B_{\frac{1}{4}}(z_0)\cap Q_{1,1}(0)$ and combining it with standard interior estimates for elliptic equations, we obtain
$$\|\nabla \hat{w}\|_{L^{\infty}(Q_{\frac{1}{2},1}(0))}\leq C\|\hat{w}\|_{L^2(Q_{\frac{3}{4},1}(0))}+C\|\hat{\mathfrak{R}}\|_{L^{\infty}(Q_{\frac{3}{4},1}(0))}+\sum_{i=1}^{2}\|\hat{\Psi}_{i}\|_{C^{1,\alpha}(Q_{\frac{3}{4},1}(0))}.$$
Rescaling back to $\tilde{w}$ yields
\begin{equation*}
\begin{aligned}
 &\eta(x_0')\|\nabla \tilde{w}\|_{L^{\infty}(Q_{\eta(x_0')/{2},\eta(x_0')}(x_0))}\leq C\dashnorm{\tilde{w}}{L^2(Q_{3\eta(x_0')/4,\eta(x_0')}(x_0))}\\
&\quad\quad\quad\quad\quad\quad+C\eta(x_0')^2\|\tilde{\mathfrak{R}}\|_{L^{\infty} (Q_{3\eta(x_0')/4,\eta(x_0')}(x_0))}+\sum_{i=1}^{2}\|\tilde{\Psi}_{i}\|^{*}_{C^{1,\alpha}(Q_{(Q_{3\eta(x_0')/4,\eta(x_0')}(x_0))}(x_0))}.
\end{aligned}
\end{equation*}
Returning to $w$ and using \eqref{aijtransformed} and \eqref{aijtransformed1}, we obtain the desired estimate \eqref{C1allgamma}.
\end{proof}

\begin{lemma}\label{lemmagamma<eta}
Under the assumptions of Theorem \ref{boundedgradientw}, let $x_0\in\Omega_{R}$. If $\gamma\leq\mu \eta(x_0')$ for some constant $\mu>0$, then the following estimates hold
\begin{equation}\label{gamma<eta}
\begin{aligned}
&\dashnorm{w}{L^2(\Omega_{\eta(x_0')}(x_0))}\leq\,C \eta(x_0') \Big(\int_{\Omega_{\frac{1}{16}\sqrt{\eta(x_0')}}(x_0)}w^2\,dx\Big)^{1/2}+C\eta(x_0'),\\
&\hspace{-2.0cm}\mbox{and}\\
&\|\nabla w\|_{L^{\infty}(\Omega_{\eta(x_0')/2}(x_0))}\leq\,C,
\end{aligned}
\end{equation}
where $C$ depends on $n$, $\kappa$, $\alpha$, $\|f_{1}\|_{C^{2,\alpha}}$ and $\|f_{2}\|_{C^{2,\alpha}}$, and the upper bound of $\mu$, but is independent of $\gamma$ and $\varepsilon$.
\end{lemma}

\begin{proof}
Multiply the equation in \eqref{equw} by $w\xi^2$ and integrate by parts, where the cut-off function $\xi\in C^{\infty}(\Omega_{\frac{1}{16}\sqrt{\eta(x_0')}}(x_0))$ satisfies 
$$\xi=1~\mbox{in}~ \Omega_{s}(x_0),\quad \xi=0~ \mbox{in}~ \Omega_{t}(x_0)\backslash\Omega_{s}(x_0),~ 0<t<s<\frac{1}{16}\sqrt{\eta(x_0')},\quad\mbox{and}~ |D \xi|\leq \frac{C}{t-s}.$$ 
This yields 
\begin{equation*}
\begin{aligned}
\int_{\Omega_t(x_0)}|Dw|^2\xi^2\,dx\,+&\,\frac{1}{\gamma}\int_{\Gamma_{1,t}(x_0)\cup\Gamma_{2,t}(x_0)}w^2\xi^2dS\\
=&\, \int_{\Omega_t(x_0)}\Big(-\mathfrak{R}w\xi^2-2w\xi DwD\xi\Big) \,dx+\frac{1}{\gamma}\sum_{i=1}^{2}\int_{\Gamma_{i,t}(x_0)}\Psi_iw\xi^2dS.
\end{aligned}
\end{equation*}
Applying Young's inequality to the right–hand side gives
\begin{equation}\label{gamma<eta2}
\begin{aligned}
&\int_{\Omega_t(x_0)}|Dw|^2\xi^2\,dx+\frac{1}{\gamma}\int_{\Gamma_{1,t}(x_0)\cup\Gamma_{2,t}(x_0)}w^2\xi^2dS\\
&\leq C\left(\int_{\Omega_t(x_0)}w^2|D \xi|^2+|\mathfrak{R}w\xi^2|\,dx+\frac{1}{\gamma}\sum_{i=1}^{2}\int_{\Gamma_{i,t}(x_0)}\Psi_i^2\xi^2dS\right).
\end{aligned}
\end{equation}

Define the boundary trace 
\begin{equation}\label{boundarytrace}
w_{\Gamma_{1}}(x)=w(x',\varepsilon+f_{1}(x')),\quad w_{\Gamma_{2}}(x)=w(x', f_{2}(x')).
\end{equation}
 For $0<s<t<\frac{1}{16}\sqrt{\eta(x')}$, using \eqref{gamma<eta2} and the assumption $\gamma\leq \mu\eta(x_0')$, we obtain
\begin{equation*}
\begin{aligned}
\int_{\Omega_s(x_0)}|w|^2\,dx
\leq&\,C\int_{\Omega_s(x_0)}|w-w_{\Gamma_{2}}|^2\,dx+C\int_{\Omega_s(x_0)}|w_{\Gamma_{2}}|^2\,dx\\
\leq&\,C\eta(x_0')^2\int_{\Omega_s(x_0)}|Dw|^2\,dx+C\eta(x_0')\int_{\Gamma_{2,s}(x_0)}w^2dS\\
\leq&\,C\Big(\eta(x_0')^2+\eta(x_0')\gamma\Big)\left(\int_{\Omega_s(x_0)}|Dw|^2\,dx+\frac{1}{\gamma}\int_{\Gamma_{2,s}(x_0)}w^2dS\right)\\
\leq&\,C\eta(x_0')^2\left(\int_{\Omega_t(x_0)}w^2|D \xi|^2+|\mathfrak{R}w\xi^2|\,dx+\frac{1}{\gamma}\sum_{i=1}^{2}\int_{\Gamma_{i,t}(x_0)}\Psi_i^2\xi^2dS\right).
\end{aligned}
\end{equation*}
Let $C_0$ and $C_1$ be fixed universal constants. From \eqref{estH} and the above estimate, it follows that
\begin{equation}\label{gamma<eta4}
\int_{\Omega_{s}(x_0)}|w|^2\,dx\leq C_0\left(\frac{\eta(x_0')^2}{(t-s)^2}+\frac{1}{16C_0}\right)\int_{\Omega_{t}(x_0)}|w|^2\,dx+C_1\eta(x_0')^3t^{n-1}.
\end{equation} 
The remainder of the proof is split into two cases.

\emph{\bf Case 1.} If $\eta(x_0')+4\sqrt{C_0}\eta(x_0')>\frac{1}{16}\sqrt{\eta(x_0')}$, then $\eta(x_0')\geq \frac{1}{16^2 (4\sqrt{C_0}+1)^2}$, and the estimates in \eqref{gamma<eta} follow immediately from Theorem \ref{gradientsmallgamma3}.

\emph{\bf Case 2.}
If $\eta(x_0')+4\sqrt{C_0}\eta(x_0')\leq\frac{1}{16}\sqrt{\eta(x_0')}$, define a sequence $t_i$ by $$t_0=\eta(x_0'),\quad\mbox{and}~ t_{i+1}=t_i+4\sqrt{C_0}\eta(x_0').$$ Let 
$\omega(s)=\int_{\Omega_{s}(x_0)}|w|^2\,dx.$ Then we have the iteration formula
\begin{equation*}
\omega(t_i)\leq \frac{1}{8}\omega(t_{i+1})+C\eta(x_0')^3t_{i+1}^{n-1}.
\end{equation*}
Iterating from $i=0$ to $k(x_0)-1$, where $k(x_0)=\left[\frac{\frac{1}{16}\sqrt{\eta(x_0')}-\eta(x_0')}{4\sqrt{C_0}\eta(x_0')}\right]$, and using $(\frac{1}{8})^{k(x_0)}\leq C\eta(x_0')^{n+2}$, we obtain
\begin{equation}\label{gamma<eta41}
\begin{aligned} 
\omega(t_0)\leq&\,(\frac{1}{8})^{k(x_0)}\omega(t_{k(x_0)})+\sum_{i=1}^{k(x_0)}\frac{C\eta(x_0')^3}{8^{i-1}} \Big(\eta(x_0')+4i\sqrt{C_0}\eta(x_0')\Big)^{n-1}\\
\leq&\,C\eta(x_0')^{n+2}\int_{\Omega_{\frac{1}{16}\sqrt{\eta(x_0')}}(x_0)}|w|^2\,dx+C\eta(x_0')^{n+2}.
\end{aligned}
\end{equation}
This establishes the first estimate in \eqref{gamma<eta}. The gradient estimate follows by combining this result with Lemma \ref{lemmac1alphaest1}.
\end{proof}

\subsection{Theorem \ref{boundedgradientw} for the case $\gamma>\mu\eta(x')$.}

In this subsection, we estimate $|\nabla w(x)|$ in the regime $\gamma>\mu\eta(x')$. The key idea is to analyze the vertical average of the solution, defined by
\begin{equation}\label{barw}
\overline{w}(x'):=\fint_{f_{2}(x')}^{\varepsilon+f_{1}(x')}w(x',x_n)\,dx_n.
\end{equation}
Taking average to equation \eqref{equflatten} in the $y_n$-variable and applying the coordinate transformation from $y$ to $x$, we find that $\overline{w}$ satisfies the degenerate elliptic equation in the $n-1$ dimensional variable $x'$:
\begin{equation}\label{equbarw}
\sum\limits_{i=1}^{n-1}\partial_{i} (\delta(x')\partial_{i}\overline{w})-\frac{2}{\gamma}\overline{w}=\sum\limits_{i=1}^{n-1}\partial_iF^i+G,\quad \mbox{for}~x\in \mathtt{B}_{R},
\end{equation}
where
\begin{equation*}
\begin{aligned}
F^i(x'):=&\, \int_{f_{2}(x')}^{\varepsilon+f_{1}(x')}\Big(\frac{x_n-f_{2}(x')}{\delta(x')}\partial_{x_i}f_{1}(x')-\frac{x_n-\varepsilon-f_{1}(x')}{\delta(x')}\partial_{x_i}f_{2}(x')\Big)\partial_{x_n}w(x)\,dx_n,\\
G(x'):=&\,-\frac{2}{\gamma}\overline{w}(x')+\frac{1}{\gamma}\sum_{i=1}^{2}\Big(w_{\Gamma_{i}}(x') -\Psi_i(x')\Big)\sqrt{1+|D_{x'}f_{i}|^2}+\int_{f_{2}(x')}^{\varepsilon+f_{1}(x')}\mathfrak{R}(x)\,dx_n,
\end{aligned}
\end{equation*}
 $\delta(x')=\varepsilon+f_{1}(x')-f_{2}(x')$, and $w_{\Gamma_{i}}$ is the boundary trace given by \eqref{boundarytrace}.

Under assumptions \eqref{fg0}--\eqref{fg}, the following bounds hold for $i=1,2$:
\begin{equation}\label{fg1}
\begin{aligned}
&\|f_i\|_{L^{\infty}(\mathtt{B}_{\frac{1}{8}\sqrt{\eta(x')}}(x'))}\leq C\eta(x'), \\
&\|D f_i\|_{L^{\infty}(\mathtt{B}_{\frac{1}{8}\sqrt{\eta(x')}}(x'))}\leq C\eta(x')^{1/2},\quad\|D^2 f_i\|_{L^{\infty}(\Omega_{R})}\leq C, \\ 
&[f_i]_{C^{\alpha}(\mathtt{B}_{\frac{1}{8}\sqrt{\eta(x')}} (x'))}\leq C\|D f_i\|_{L^{\infty}(\mathtt{B}_{\frac{1}{8}\sqrt{\eta(x')}}(x'))}\eta(x')^{\frac{1-\alpha}{2}}\leq C\eta(x')^{1-\frac{\alpha}{2}},
\end{aligned}
\end{equation}

For convenience, define
$$\mathcal{F}^{i}(x):=\frac{x_n-f_{2}(x')}{\delta(x')}\partial_{x_i}f_{1}(x')-\frac{x_n-\varepsilon-f_{1}(x')}{\delta(x')}\partial_{x_i}f_{2}(x').$$
Then, using \eqref{fg1}, we obtain
\begin{equation}\label{mathcalFi}
\begin{aligned}
& \|\mathcal{F}^{i}\|_{L^{\infty}(\mathtt{B}_{\frac{1}{8}\sqrt{\eta(x')}} (x'))}\leq C\eta(x')^{1/2},\\
&[\mathcal{F}^{i}]_{C_{x'}^{\alpha}(\Omega_{\frac{1}{8}\sqrt{\eta(x')}} (x))}\leq C\|D_{x'}\mathcal{F}^{i}\|_{L^{\infty}(\Omega_{\frac{1}{8}\sqrt{\eta(x')}} (x))}\eta(x')^{\frac{1-\alpha}{2}}\leq C\eta(x')^{\frac{1-\alpha}{2}}.
\end{aligned} 
\end{equation}

From \eqref{fg1} and \eqref{mathcalFi}, we estimate the H\"older norm of $F^{i}$:
$$|F^i(x')|\leq\, C\eta(x')^{\frac{3}{2}}\|\partial_n w\|_{L^{\infty}(\Omega_{R})},$$
and
\begin{equation}\label{Ficalpha}
\begin{aligned}
\left[F^{i}\right]_{C^{\alpha}(\mathtt{B}_{\frac{1}{8}\sqrt{\eta(x')}} (x'))}\leq&\, C\sum_{j=1}^{2}[f_j]_{C^{\alpha}(\mathtt{B}_{\frac{1}{8}\sqrt{\eta(x')}} (x'))}\|\mathcal{F}^{i}\partial_nw\|_{L^{\infty}(\mathtt{B}_{\frac{1}{8}\sqrt{\eta(x')}} (x'))}\\
 &+C\eta(x')[\mathcal{F}^{i}]_{C_{x'}^{\alpha}(\mathtt{B}_{\frac{1}{8}\sqrt{\eta(x')}} (x'))}\|\partial_nw\|_{L^{\infty}(\Omega_{\frac{1}{8}\sqrt{\eta(x')}} (x))}\\
 &+C\eta(x')[\partial_nw]_{C^{\alpha}(\Omega_{\frac{1}{8}\sqrt{\eta(x')}} (x'))}\|\mathcal{F}^{i}\|_{L^{\infty}(\mathtt{B}_{\frac{1}{8}\sqrt{\eta(x')}} (x'))}\\
 \leq&\, C\eta(x')^{\frac{3}{2}-\alpha}\|\partial_nw\|_{L^{\infty}(\Omega_{\frac{1}{8}\sqrt{\eta(x')}} (x))}+C\eta(x')^{\frac{3}{2}}[\partial_nw]_{C^{\alpha}(\Omega_{\frac{1}{8}\sqrt{\eta(x')}} (x'))}.
\end{aligned}
\end{equation}

Under the assumption \eqref{estH}, the term $G$ satisfies
\begin{equation}\label{Ginfty}
\begin{aligned}
 |G(x')|\leq&\, C\sum_{i=1}^{2}\frac{1}{\gamma}|\bar{w}(x')-w_{\Gamma_{i}}(x')|_{L^{\infty}(\mathtt{B}_{\frac{1}{8}\sqrt{\eta(x')}} (x'))}\\
&+C\sum_{i=1}^{2}\frac{1}{\gamma}|\Psi_i(x')|+C\eta(x')\underset{f_2(x')\leq x_n\leq \varepsilon+f_{1}(x')}{\sup}|\mathfrak{R}(x',x_n)|\\
\leq&\, C\frac{\eta(x')}{\gamma}\|Dw\|_{L^{\infty}(\Omega_{R})}+C\frac{\eta(x')}{\gamma+\eta(x')}.
\end{aligned}
\end{equation}

\subsubsection{Pointwise estimate for $\overline{w}(x')$}

Using a scaling argument and Lemma \ref{grdeientsmallgamma6}, we obtain the following estimate.

\begin{lemma}\label{lemmac1alphaest}
Let $w\in H^1(\Omega_{\eta(x_0')}(x_0))$, $x_0\in\Omega_{\frac{3}{4}R}$, be a solution to
\begin{equation*}
\left\{ \begin{aligned}
\Delta w&= \mathfrak{R}(x)&\mbox{in}&~ \Omega_{\eta(x_0')}(x_0), \\
w+\gamma \partial_{\nu}w&=\Psi_1(x)&\mbox{on}&~\Gamma_{1,\eta(x_0')}(x_0), \\
w+\gamma\partial_{\nu}w&=\Psi_2(x)&\mbox{on}&~\Gamma_{2,\eta(x_0')}(x_0),
\end{aligned} \right.
\end{equation*}
with $\gamma>0$ and $f_1,f_2$ satisfying \eqref{fg0}--\eqref{fg}. Assume $\mathfrak{R}\in L^{\infty}(\Omega_{\eta(x_0')}(x_0))$ and $\Psi_i\in C^{1,\alpha}(\Omega_{\eta(x_0')}(x_0))$, $i=1,2$. If $\gamma>\eta(x_0')$, then 
\begin{equation}\label{C1largegamma}
\begin{aligned}
&\|Dw\|_{L^{\infty}(\Omega_{\frac{1}{2}\eta(x_0')}(x_0))}+\eta(x_0')^{\alpha}[Dw]_{C^{\alpha}(\Omega_{\frac{1}{2}\eta(x_0')}(x_0))}\\
&\leq C\Bigg(\frac{1}{\gamma}\dashnorm{w}{L^2(\Omega_{\frac{3}{4}\eta(x_0')}(x_0))}+\dashnorm{Dw}{L^2(\Omega_{\frac{3}{4}\eta(x_0')}(x_0))}+\frac{1}{\gamma}\sum_{i=1}^2\|\Psi_i\|^*_{C^{\alpha}(\Omega_{\frac{3}{4}\eta(x_0')}(x_0))}\\
& \quad +\eta(x_0')\|\mathfrak{R}\|_{L^{\infty}(\Omega_{\frac{3}{4}\eta(x_0')}(x_0))}\Bigg),
\end{aligned}
\end{equation}
where $C$ depends only on $n$, $\alpha$, $\kappa$, and the upper bound of $\|f_i\|_{C^{2,\alpha}(\mathtt{B}_{R})}$ for $i=1,2$.
\end{lemma} 
 
We now apply the pointwise estimates for degenerate equations from Section \ref{sec_3} (Propositions \ref{epsilongammato0} and \ref{prop_gradient_v}) to equation \eqref{equbarw}, which yields control over $\overline{w}$ and its gradient in terms of $\|Dw\|_{L^{\infty}(\Omega_R)}$.

\begin{prop}\label{epsilongammato2}
Under the assumptions of Theorem~\ref{boundedgradientw}, let $w \in H^1(\Omega_{R})$ be a solution to \eqref{equw}, with $\overline{w}$ defined as in \eqref{barw} and $\gamma_0$ given by \eqref{gamma00}. Then, for any $\varepsilon \in (0, \varepsilon_0]$ and $\gamma \in (0, \gamma_0)$, the following estimates hold:
\begin{enumerate}
\item[(i)] For $ |x'| \leq R $,
\begin{equation}\label{pointwisebarw}
|\overline{w}(x')| \leq C\, \eta(x') \left( \|Dw\|_{L^{\infty}(\Omega_R)} + 1 \right).
\end{equation}
\item[(ii)] For all $ x \in \Omega_{R/2} $ with $ \eta(x') < \gamma $,
\begin{equation}\label{pointwisegradientbarw}
|D \overline{w}(x')| \leq C \left( \gamma^{\frac{\alpha}{2}} \eta(x')^{\frac{1-\alpha}{2}} + \left( \frac{\eta(x')}{\gamma} \right)^{1/2} \right) \left( \|Dw\|_{L^{\infty}(\Omega_R)} + 1 \right).
\end{equation}
\end{enumerate}
The constants $\varepsilon_0$ and $C>0$ depend only on $n$, $\alpha$, $\kappa$, $R$,  a lower bound of $\gamma_0-\gamma$ and upper bounds of $\|f_{1}\|_{C^{2,\alpha}}$ and $\|f_{2}\|_{C^{2,\alpha}}$.
\end{prop}

\begin{proof}
(i) Using \eqref{Ficalpha} and \eqref{Ginfty}, we obtain
\begin{equation}\label{FG1}
\begin{aligned}
\sqrt{\gamma}\|F\|_{\frac{3}{2},\mathtt{B}_{R}}+\gamma\|G\|_{1,\mathtt{B}_{R}}\leq&\, C\|Dw\|_{L^{\infty}(\Omega_R)}+C,\\
\gamma^{1/2}\eta(x')^{-1/2}\|G\|_{L^{\infty}(\mathtt{B}_{\frac{1}{8}\sqrt{\eta(x')}} (x'))}\leq&\, C(\frac{1}{\gamma}\eta(x'))^{1/2}\Big(\|Dw\|_{L^{\infty}(\Omega_R)}+1\Big).
\end{aligned}
\end{equation}
By \eqref{fg}, for $x\in\Omega_{R}$ we have 
\begin{equation}\label{f1f2tx}
\begin{aligned}
&|f_1(x')-f_2(x')|\leq  \frac{1}{2}\|D^2(f_1-f_2)\|_{L^{\infty}(\Omega_R)}|x'|^2,\\
&|Df_1(x')-Df_2(x')|\leq  \|D^2(f_1-f_2)\|_{L^{\infty}(\Omega_R)}|x'|. 
\end{aligned}
\end{equation}
Let $x'=Ry'$ for $y'\in\mathtt{B}_1$, and define $t(y')=(f_1(Ry')- f_2(Ry'))/(2R^2)$. Then from \eqref{f1f2tx}, we choose $\tilde{\Lambda} = \frac{1}{4} \|D^2(f_1 - f_2)\|_{L^{\infty}(\Omega_R)}$, which allows us to apply Proposition~\ref{epsilongammato0} to $\bar{w}(Ry')$ with $\tilde{\varepsilon} = \varepsilon/(2R^2)$ and $\sigma = 2$, thereby obtaining \eqref{pointwisebarw} for all $\gamma \in (0, \gamma_0)$. 

(ii) For $x_0 \in \Omega_{3R/4}$, we use the assumption $$\frac{1}{\gamma}\sum_{i=1}^2\|\Psi_i\|^*_{C^{\alpha}(\Omega_{\frac{3}{4}\eta(x_0')}(x_0))}+\eta(x_0')\|\mathfrak{R}\|_{L^{\infty}(\Omega_{\frac{3}{4}\eta(x_0')}(x_0))}\leq C\frac{\eta(x_0')}{\gamma}.$$ 
Then, by Lemma~\ref{lemmac1alphaest} and \eqref{pointwisebarw}, we obtain
\begin{equation*}
\begin{aligned}
\|D w\|^{*}_{C^{\alpha}(\Omega_{\frac{1}{2}\eta(x_0')}(x_0))}\leq&\, C\|Dw\|_{L^{\infty}(\Omega_{\frac{1}{2}\eta(x_0')}(x_0))}+C\eta(x_0')^{\alpha}[Dw]_{C^{\alpha}(\Omega_{\frac{1}{2}\eta(x_0')}(x_0))}\\
\leq&\, C\frac{\eta(x_0')}{\gamma}\Big(\|Dw\|_{L^{\infty}(\Omega_R)}+1\Big)+C\|Dw\|_{L^{\infty}(\Omega_R)}.
\end{aligned}
\end{equation*}
Since $\eta(x_0')\sim\eta(x')$ for $x_0\in \Omega_{\frac{1}{8}\sqrt{\eta(x')}}(x')$, it follows that for $x \in \Omega_{R/2}$ with $\eta(x') < \gamma$,
$$ \underset{x_0\in \Omega_{\frac{1}{8}\sqrt{\eta(x')}} (x)}{\sup} \|D w\|^{*}_{C^{\alpha}(\Omega_{\frac{1}{2}\eta(x_0')}(x_0))}\leq C\|Dw\|_{L^{\infty}(\Omega_R)}+C.$$
Substituting this into \eqref{Ficalpha} gives
\begin{equation}\label{FG2}
\begin{aligned}
\left[F^i\right]_{C^{\alpha}(\mathtt{B}_{\frac{1}{8}\sqrt{\eta(x')}} (x'))}\leq&\, C\eta(x')^{\frac{3}{2}-\alpha}\|\partial_nw\|_{L^{\infty}(\Omega_{\frac{1}{8}\sqrt{\eta(x')}} (x))}+C\eta(x')^{\frac {3}{2}}[\partial_nw]_{C^{\alpha}(\Omega_{\frac{1}{8}\sqrt{\eta(x')}} (x'))}\\
\leq&\, C\eta(x')^{\frac{3}{2}-\alpha}\underset{x_0\in \Omega_{\frac{1}{8}\sqrt{\eta(x')}} (x)}{\sup} \|D w\|^{*}_{C^{\alpha}(\Omega_{\frac{1}{2}\eta(x_0')}(x_0))}\\
\leq&\,C\eta(x')^{\frac{3}{2}-\alpha}\Big(\|Dw\|_{L^{\infty}(\Omega_R)}+1\Big).
\end{aligned}
\end{equation}
Estimate \eqref{pointwisegradientbarw} now follows from Proposition~\ref{prop_gradient_v} together with \eqref{FG1} and \eqref{FG2}.
\end{proof}
 
\subsubsection{Energy estimate for $w-\overline{w}$}
\begin{lemma}\label{estgradientwbarw}
Under the assumptions of Theorem \ref{boundedgradientw}, let $w$ be a solution to \eqref{equw} and let $\overline{w}$ be defined as in \eqref{barw}. For $x_0\in\Omega_{\frac{R}{2}}$,
\begin{equation*}
 \int_{\Omega_{\eta(x_0')}(x_0)}|D(w-\overline{w})|^2\,dx\leq C\eta(x_0')^n\left( 1+Q(x_0)^2\right),
\end{equation*}
where $$Q(x_0):=\underset{x\in\Omega_{\frac{1}{16}\sqrt{\eta(x_0')}}(x_0)}{\sup}\left(|D \overline{w}(x)|+\frac{1}{\gamma} |\overline{w}(x)| \right).$$
The constant $C$ depends only on $n$, $\alpha$, $\kappa$, $\|f_{1}\|_{C^{2,\alpha}}$, $\|f_{2}\|_{C^{2,\alpha}}$, and $R$. 
\end{lemma}

\begin{proof}
We multiply \eqref{equw} by $(w-\overline{w})\xi^2$ and integrate by parts, where $\xi(x)=\xi(x')$ satisfies $\xi\in C_c^{\infty}(\Omega_t(x_0))$, $\xi=1$ in $\Omega_s(x_0)$, where $0<s<t<\frac{1}{16}\sqrt{\eta(x_0')}$, and $|D \xi|\leq\frac{C}{|s-t|}$. This gives
\begin{equation}\label{integralv}
\begin{aligned}
&\int_{\Omega_t(x_0)}D w\cdot D\big((w-\overline{w})\xi^2\big)+\mathfrak{R}(w-\overline{w})\xi^2\,dx=\sum_{i=1}^{2} \int_{\Gamma_{i,t}(x_0)}\frac{\tilde{\Psi}_i-w}{\gamma} \,(w-\overline{w})\xi^2dS.
\end{aligned}
\end{equation}

Let $w_{\Gamma_{i}}$ be the boundary trace given by \eqref{boundarytrace}.
Note that $$\partial_i\overline{w}(x')=\fint_{f_{2}(x')}^{\varepsilon+f_{1}(x')}\partial_iw (x',x_n)\,dx_n+\frac{\partial_if_{1}(x')}{\delta(x')}(w_{\Gamma_{1}} -\overline{w})(x')-\frac{\partial_if_{2}(x')}{\delta(x')}(w_{\Gamma_{2}}-\overline{w})(x'),$$ and 
$\int_{f_{2}(x')}^{\varepsilon+f_{1}(x')}(w-\overline{w})\,dx_n=0$, $\partial_n \overline{w}(x)=\partial_n \overline{w}(x')=0$. Hence,
\begin{align}\label{integralbarv}
&\int_{\Omega_t(x_0)}D\overline{w}\cdot D((w-\overline{w})\xi^2)\,dx\nonumber\\
&=\int_{|x'-x_0'|\leq t}(D_{x'}\overline{w}\cdot D_{x'}\xi^2)\Big(\int_{f_{2}(x')}^{\varepsilon+f_{1}(x')}(w-\overline{w})\,dx_n\Big)\,dx'+\int_{\Omega_t(x_0)}\xi^2 D_{x'}\overline{w}\cdot D_{x'}(w-\overline{w})\,dx\nonumber\\
&=\int_{\Omega_t(x_0)}\xi^2\frac{\overline{w}-w_{\Gamma_{1}}}{\delta(x')} D_{x'}\overline{w}\cdot D_{x'}f_{1}+\xi^2\frac{w_{\Gamma_{2}}-\overline{w}}{\delta(x')} D_{x'}\overline{w}\cdot D_{x'}f_{2}\,dx\\
&:=Q_{t}[f_{1},f_{2}].\nonumber
 \end{align}
Subtracting \eqref{integralbarv} from \eqref{integralv} yields
\begin{equation*}
\begin{aligned}
&\int_{\Omega_t(x_0)}D(w-\overline{w})\cdot D\big((w-\overline{w})\xi^2\big)\,dx+Q_{t}[f_{1},f_{2}]+ \int_{\Omega_t(x_0)}\mathfrak{R}(w-\overline{w})\xi^2\,dx\\
&= \frac{1}{\gamma}\sum_{i=1}^{2}\int_{\Gamma_{i,t}(x_0)}(\tilde{\Psi}_i-w)(w-\overline{w})\xi^2dS.
\end{aligned}
\end{equation*}
We now estimate each term.

{\bf Step 1. Estimating the first term on the left-hand side.} Using Cauchy inequality, 
\begin{equation}\label{lefthand1}
\begin{aligned}
&\int_{\Omega_t(x_0)} D(w-\overline{w})\cdot D\big((w-\overline{w})\xi^2\big)\,dx\\
&= \int_{\Omega_t(x_0)}| D(w-\overline{w})|^2\xi^2+ 2\xi(w-\overline{w}) D\xi\cdot D (w-\overline{w})\,dx\\
&\geq \int_{\Omega_{t}(x_0)}|D(w-\overline{w})|^2\xi^2\,dx-\frac{1}{8}\int_{\Omega_{t}(x_0)}|D(w-\overline{w})|^2\xi^2\,dx-C\int_{\Omega_{t}(x_0)}(w-\overline{w})^2|D \xi|^2\,dx.
\end{aligned}
\end{equation}
By H\"{o}lder's inequality, 
\begin{align}\label{lefthand11}
|w(x)-\overline{w}(x)|\leq&\,\int_{f_2(x')}^{\varepsilon+f_1(x')}|\partial_{n} w(x', x_n)|\,dx_n
\leq\, C\eta(x_0')^{1/2}\left(\int_{f_2(x')}^{\varepsilon+f_1(x')}|\partial_{n} w(x', x_n)|^2\,dx_n\right)^{1/2}\nonumber\\
\leq&\,C\eta(x_0')^{1/2}\left(\int_{f_2(x')}^{\varepsilon+f_1(x')}|D(w-\overline{w})(x', x_n)|^2\,dx_n\right)^{1/2}.
\end{align}
Substituting \eqref{lefthand11} into \eqref{lefthand1}, we obtain
\begin{equation*}
\begin{aligned}
& \int_{\Omega_t(x_0)} D(w-\overline{w})\cdot D\big((w-\overline{w})\xi^2\big)\,dx\\
&\geq \frac{7}{8}\int_{\Omega_{t}(x_0)}|D(w-\overline{w})|^2\xi^2\,dx -C\eta(x_0')^2\int_{\Omega_{t}(x_0)}|D(w-\overline{w})|^2|D \xi|^2\,dx.
\end{aligned}
\end{equation*}

{\bf Step 2. Estimating the second and third terms on the left-hand side.}
For $x\in\Omega_{\frac{1}{16}\sqrt{\eta(x_0')}}(x_0)$, we have $|D f_{1}(x')| + |D f_{2}(x')| \leq C \eta(x_0')^{1/2}$. Then, for $0 \leq t \leq \frac{1}{16} \sqrt{\eta(x_0')}$, 
 \begin{align*}
|Q_{t}[f_{1},f_{2}]|\leq&\,\int_{\Omega_t(x_0)}\xi^2\Big( | D_{x'}\overline{w}\cdot D_{x'}f_{1}| + | D_{x'}\overline{w}\cdot D_{x'}f_{2}|\Big) \fint_{f_{2}(x')}^{\varepsilon+f_{1}(x')}|\partial_n w|\,dx_n\,dx\\
=&\,\int_{\Omega_t(x_0)}\xi^2 \Big(| D_{x'}\overline{w}\cdot D_{x'}f_{1}|+ | D_{x'}\overline{w}\cdot D_{x'}f_{2}|\Big)|\partial_nw|\,dx\\
\leq&\,\frac{1}{16}\int_{\Omega_t(x_0)}\xi^2|D(w-\overline{w})|^2\,dx+C\eta(x_0')\int_{\Omega_t(x_0)}\xi^2|D\overline{w}|^2\,dx.
 \end{align*}
 
{\bf Step 3. Estimating the remaining terms.} Using \eqref{lefthand11} and H\"{o}lder's inequality again,
\begin{equation*}
\int_{\Omega_t(x_0)}\mathfrak{R}(w-\overline{w})\xi^2\,dx\leq \frac{1}{16}\int_{\Omega_t(x_0)}|D(w-\overline{w})|^2\xi^2\,dx+C\eta(x_0')^2\int_{\Omega_t(x_0)}\mathfrak{R}^2\,dx.
\end{equation*}
Similarly, for the boundary integrals on the right-hand side, 
\begin{equation*}
\begin{aligned}
&\frac{1}{\gamma}\int_{\Gamma_{1,t}(x_0)}(\tilde{\Psi}_1-w) (w-\overline{w})\xi^2dS\\
&= \frac{1}{\gamma}\int_{\Gamma_{1,t}(x_0)}\tilde{\Psi}_1(w-\overline{w})\xi^2-(w-\bar{w})^2\xi^2-\bar{w}(w-\bar{w})\xi^2dS\\
&\leq\frac{1}{\gamma}\int_{\Gamma_{1,t}(x_0)}|\tilde \Psi_1||w-\overline{w}|\xi^2+ |\overline{w}(w-\overline{w})|\,dS\\
&\leq\frac{1}{16} \int_{\Omega_t(x_0)}|D(w-\overline{w})|^2\xi^2\,dx +C\frac{1}{\gamma^2}\eta(x_0')\int_{\Gamma_{1,t}(x_0)}\Big(\tilde \Psi_1^2+\overline{w}^2\Big)\,dS,
\end{aligned}
\end{equation*}
and the same estimate holds on $\Gamma_{2,t}(x_0)$ with $\tilde{\Psi}_2$ in place of $\tilde{\Psi}_1$. 

Combining these estimates, we arrive at
\begin{equation}\label{iterationv-barv}
\begin{aligned}
&\int_{\Omega_{t}(x_0)}|D(w-\overline{w})|^2\xi^2\,dx\\
&\leq \frac{1}{2}\int_{\Omega_{t}(x_0)}|D(w-\overline{w})|^2\xi^2\,dx+C_{0}\eta(x_0')^2\int_{\Omega_{t}(x_0)}|D(w-\overline{w})|^2|D\xi|^2\,dx\\
&\quad +C_{0}\eta(x_0')\int_{\Omega_{t}(x_0)}|D\overline{w}|^2\,dx
+C\int_{\Omega_{t}(x_0)}\Big(\frac{{1}}{\gamma^2}(|\tilde\Psi_1|^2+|\tilde\Psi_2|^2+|\overline{w}|^2)
+\eta(x_0')^2\mathfrak{R}^2\Big)\,dx,
\end{aligned}
\end{equation}
where $C_0$ is a fixed universal constant. Let $\tilde{C}_0:=2C_0$. From \eqref{iterationv-barv} and the assumption \eqref{estH}, we obtain
\begin{equation*}
\begin{aligned}
 \int_{\Omega_{s}(x_0)}|D(w-\overline{w})|^2\,dx
\leq&\,\, \tilde{C}_0\frac{\eta(x_0')^2}{(s-t)^2}\int_{\Omega_{t}(x_0)}|D(w-\overline{w})|^2\,dx+ Ct^{n-1}\eta(x_0')(Q(x_0)+1)^2.
\end{aligned}
\end{equation*}

{\bf Step 4. Iteration and conclusion.} 
Proceeding with an iteration argument (as in Lemma~\ref{lemmagamma<eta}, from \eqref{gamma<eta4} to \eqref{gamma<eta41}), and noting that
$$\int_{\Omega_{\frac{1}{16}\sqrt{\eta(x_0')}}(x_0)}|D(w-\overline{w})|^2\,dx\leq \int_{\Omega_{\frac{ 1}{16}\sqrt{\eta(x_0')}}(x_0)}|Dw|^2\,dx,$$ we obtain
\begin{equation}\label{estequw-barw}
\int_{\Omega_{\eta(x_0')}(x_0)}|D(w-\overline{w})|^2\,dx\leq C\eta(x_0')^n\Big( \int_{\Omega_{\frac{ 1}{16}\sqrt{\eta(x_0')}}(x_0)}|Dw|^2\,dx+Q(x_0)^2+1\Big).
\end{equation}
Finally, we estimate the integral on the right-hand side of \eqref{estequw-barw}. Multiply the equation in \eqref{equw} by $w\xi^2$ and integrate by parts, where 
$$\xi\in C_c^{\infty}(\Omega_R),\quad  \xi=1~\mbox{in}~ \Omega_{\frac{3}{4}R},\quad\mbox{and}~ |D\xi|\leq \frac{C}{R},$$ using H\"{o}lder's inequality, we have
\begin{equation*}
\begin{aligned}
&\int_{\Omega_{\frac{3}{4}R}}|Dw|^2\xi^2\,dx+\frac{1}{\gamma}\int_{\Gamma_{1,\frac{3}{4}R}\cup\,\Gamma_{2,\frac{3}{4}R} }w^2\xi^2dS\\
&\leq C\left(\int_{\Omega_R }w^2|D \xi|^2+|\mathfrak{R}w\xi^2|\,dx+\frac{1}{\gamma}\sum_{i=1}^{2}\int_{\Gamma_{i,R} }\Psi_i^2\xi^2dS\right).
\end{aligned}
\end{equation*}
 By above inequality, the assumption that $\|w\|_{L^{\infty}(\Omega_R)}\leq 1$ and the fact that $\Omega_{\frac{1}{16}\sqrt{\eta(x_0')}}(x_0)\subset \Omega_{\frac{3}{4}R}$ for $x_0\in\Omega_{\frac{R}{2}}$, we have
$$\int_{\Omega_{\frac{ 1}{16}\sqrt{\eta(x_0')}}(x_0)}|Dw|^2\,dx\leq\int_{{\Omega_{\frac{3}{4}R}}}|Dw|^2\xi^2\,dx\leq C.$$
Substituting this into \eqref{estequw-barw} completes the proof.
\end{proof}

\subsection{Proof of Theorem \ref{boundedgradientw}}\label{subsectionboundedgra}

\begin{proof}[Proof of Theorem \ref{boundedgradientw}]
We now complete the proof by decoupling the estimates. Using the energy estimate from Lemma~\ref{estgradientwbarw}, we compare the gradient of $w$ with that of its vertical average $\overline{w}$. Combining this with estimates for $\overline{w}$ yields a pointwise bound on $Dw(x_0)$ in terms of the global norm $\|Dw\|_{L^{\infty}(\Omega_R)}$. A subsequent iteration over points then establishes the uniform boundedness of $\|Dw\|_{L^{\infty}(\Omega_R)}$.

We first prove that for every $x_0\in\Omega_{R/2}$ satisfying $\eta(x_0')<\gamma$ the following estimate holds:
\begin{equation}\label{estgradientw}
|Dw(x_0)|\leq C_1\left(\eta(x_0')^{\frac{1-\alpha}{2}}+\big(\frac{1}{\gamma}\eta(x_0')\big)^{1/2}\right)\Big(\|Dw\|_{L^{\infty}(\Omega_R)}+1\Big)+C_1, 
\end{equation}
where $C_1>1$ is a universal constant independent of $\varepsilon$, $\gamma$, and $x_0$.

To show this, we apply Lemma~\ref{lemmac1alphaest} together with the triangle inequality and the bounds $|\mathfrak{R}(x)|\leq \frac{C}{\delta(x')}$, and $\|\Psi_i\|^*_{C^{\alpha}(\Omega_{\frac{3}{4}\eta(x_0')}(x_0))}\leq C\gamma$ from~\eqref{estH}, obtaining
\begin{equation*}
\begin{aligned}
|Dw(x_0)|\leq&\,C\left(\frac{1}{\gamma}\|w\|_{L^{\infty}(\Omega_{\frac{3}{4}\eta(x_0')}(x_0))}+\eta(x_0')^{-n/2}\|Dw\|_{L^2(\Omega_{\frac{3}{4}\eta(x_0')}(x_0))}\right)+C\\
\leq&\,\frac{C}{\gamma}\left(\|w-\overline{w}\|_{L^{\infty}(\Omega_{\frac{3}{4}\eta(x_0')}(x_0))}+\|\overline{w}\|_{L^{\infty}(\Omega_{\frac{3}{4}\eta(x_0')}(x_0))}\right)\\
&+C\eta(x_0')^{-n/2}\|D(w-\overline{w})\|_{L^2(\Omega_{\frac{3}{4}\eta(x_0')}(x_0))}+C\|D \overline{w}\|_{L^{\infty}(\Omega_{\frac{3}{4}\eta(x_0')}(x_0))}+C.
\end{aligned}
\end{equation*}
By the mean-value theorem, $\|w-\overline{w}\|_{L^{\infty}(\Omega_{\frac{3}{4}\eta(x_0')}(x_0))}\leq C\eta(x_0')\|Dw\|_{L^{\infty}(\Omega_R)}$. Lemma~\ref{estgradientwbarw} implies
\begin{equation*}
\eta(x_0')^{-n/2}\|D(w-\overline{w})\|_{L^2(\Omega_{\eta(x_0')}(x_0))}\leq C\underset{x\in\Omega_{\frac{1}{16}\sqrt{\eta(x_0')}}(x_0)}{\sup}\left(|D \overline{w}(x)|+\frac{1}{\gamma}\overline{w}(x)\right)+C.
\end{equation*}
From Proposition~\ref{epsilongammato2}, we obtain
\begin{equation*}
\underset{x\in\Omega_{\frac{1}{16}\sqrt{\eta(x_0')}}(x_0)}{\sup}\left(|D \overline{w}(x)|+\frac{1}{\gamma}\overline{w}(x)\right)\leq C \left(\gamma^{\frac{\alpha}{2}}\eta(x_0')^{\frac{1-\alpha}{2}}+\big(\frac{1}{\gamma}\eta(x_0')\big)^{1/2}\right)\Big(\|Dw\|_{L^{\infty}(\Omega_R)}+1\Big).
\end{equation*}
Combining these estimates yields \eqref{estgradientw}.

Let $C_1$ be the large universal constant appearing in \eqref{estgradientw}. Define $\varepsilon_0$ and $R_0$ by
$\varepsilon_0 = \frac{1}{2} \left(\frac{1}{16C_1^2}\right)^{\frac{1}{1-\alpha}}$ and $\quad R_0 = \frac{1}{\sqrt{2}} \left( \frac{1}{4C_1} \right)^{\frac{1}{1-\alpha}}.$
By choosing $C_1$ sufficiently large, we ensure that $\varepsilon_0 < 1/4$ and $R_0 < R/2$. Then, for any $\varepsilon \in (0, \varepsilon_0)$ and any $x_0 \in \Omega_{R_0}$, it follows from the definition $\eta(x') = \varepsilon + |x'|^2$ that
$\eta(x_0')^{\frac{1-\alpha}{2}} \leq \left( \varepsilon_0 + R_0^2 \right)^{\frac{1-\alpha}{2}} \leq \frac{1}{4C_1}.$

Now set $\mu = 16C_1^2$ and proceed by considering two cases:

\textbf{Case 1: $\gamma \in (0, \mu\varepsilon]$.} 
Since $\gamma \leq \mu\varepsilon \leq \mu\eta(x')$ for all $x \in \Omega_{\frac{R}{2}}$, Lemma~\ref{lemmagamma<eta} directly yields \eqref{equboundedgradientw}.

\textbf{Case 2: $\gamma \in (\mu\varepsilon, \gamma_0)$.} 
Taking the supremum in \eqref{estgradientw} over all $x_0 \in \Omega_{\frac{R_0}{2}}$ with $\mu\eta(x_0') < \gamma$, we obtain
$$\underset{x \in \Omega_{R_0/2}, \mu\eta(x') < \gamma}{\sup} |Dw(x)| \leq \frac{1}{2} \|Dw\|_{L^{\infty}(\Omega_R)} + C.$$
Rearranging this inequality gives
$$ \underset{x \in \Omega_{R_0/2}, \mu\eta(x') < \gamma}{\sup} |Dw(x)| \leq \underset{x \in \Omega_{R}, \mu\eta(x') \geq \gamma}{\sup} |Dw(x)| + 2C.$$
By Lemma~\ref{lemmagamma<eta}, $|Dw(x)| \leq C$ whenever $\mu\eta(x') \geq \gamma$, hence
$$\underset{x \in \Omega_{R_0/2}}{\sup} |Dw(x)| \leq C.$$
Moreover, by Theorem~\ref{gradientsmallgamma3}, we have $\sup_{x \in \Omega_R \setminus \Omega_{R_0/2}} |Dw(x)| \leq C.$ Combining these estimates completes the proof of Theorem~\ref{boundedgradientw}.
\end{proof}

\section{Proof of Theorem \ref{upperbound}}\label{sec_5}

We prove Theorem \ref{upperbound} in this section via a decomposition technique adapted from \cite{BLY}. The strategy consists of separating the solution into singular and regular parts, determining the free constants multiplying the singular parts, and combining the resulting estimates to derive optimal gradient bounds. The estimates for the singular part rely on a specially constructed auxiliary function together with the results of Section \ref{sec_4}, while those for the regular parts follow directly from the latter.

Let $u_{j, \gamma, \varepsilon} \in C^2(\overline{\widetilde{\Omega}^{\varepsilon}})$, for $j = 1, 2, 3$, be the solutions to the following boundary value problems:
\begin{equation}\label{equv1}
 \left\{ \begin{aligned}
 \Delta u_{1,\gamma, \varepsilon}&=0 &\mbox{in}&~\tilde{\Omega}^{\varepsilon},\\
 u_{1,\gamma, \varepsilon}+\gamma\partial_{\nu}u_{1,\gamma, \varepsilon}&=1 &\mbox{on}&~\partial D_1^{\varepsilon},\\
 u_{1,\gamma, \varepsilon}+\gamma\partial_{\nu}u_{1,\gamma, \varepsilon}&=0 &\mbox{on}&~\partial D_2^{\varepsilon},\\
 u_{1,\gamma, \varepsilon}&=0 &\mbox{on}&~\partial \Omega,
     \end{aligned} \right.\quad\,\,\,\,\,\,
 ~~\left\{ \begin{aligned}
 \Delta u_{2,\gamma, \varepsilon}&=0 &\mbox{in}&~\tilde{\Omega}^{\varepsilon},\\
 u_{2,\gamma, \varepsilon}+\gamma\partial_{\nu}u_{2,\gamma, \varepsilon}&=0 &\mbox{on}&~\partial D_1^{\varepsilon},\\
 u_{2,\gamma, \varepsilon}+\gamma\partial_{\nu}u_{2,\gamma, \varepsilon}&=1 &\mbox{on}&~\partial D_2^{\varepsilon},\\
 u_{2,\gamma, \varepsilon}&=0 &\mbox{on}&~\partial \Omega,
 \end{aligned} \right.
\end{equation}
 and
\begin{equation}\label{equv3}
 \left\{ \begin{aligned}
 \Delta u_{3,\gamma, \varepsilon}&=0 &\mbox{in}&~\tilde{\Omega}^{\varepsilon},\\
 u_{3,\gamma, \varepsilon}+\gamma\partial_{\nu}u_{3,\gamma, \varepsilon}&=0 &\mbox{on}&~\partial D_1^{\varepsilon}\cup\partial D_2^{\varepsilon},\\
 u_{3,\gamma, \varepsilon}&=\phi &\mbox{on}&~\partial \Omega.
 \end{aligned} \right.
\end{equation}
When no confusion arises, we abbreviate $u_{j,\gamma, \varepsilon}$ as $u_j$ for $j = 1, 2, 3$, and drop the superscript $\varepsilon$ by writing $\widetilde{\Omega} := \widetilde{\Omega}^{\varepsilon}$, $D_1 := D_1^{\varepsilon}$, and $D_2 := D_2^{\varepsilon}$.
 
 By linearity, the solution may be expressed as
 \begin{equation}\label{decomposition}
 u=K_{1}u_{1}+K_{2}u_{2}+u_{3}\quad\text{in }\tilde{\Omega},
 \end{equation}
 where $K_i = K_i( \gamma , \varepsilon)$ ($i=1,2$) are free constants uniquely determined by the flux conditions in \eqref{lcctype}. This decomposition isolates specific physical effects: $u_{1}$ and $u_{2}$ are capacitary-type functions associated with each inclusion, whose gradients $\nabla u_{i}$ become singular in the narrow region between $\partial{D}_{1}$ and $\partial{D}_{2}$, while $|\nabla u_{3}|$ remains a regular term.

From \eqref{decomposition}, the gradient decomposes as
\begin{equation}\label{decomposition2}
D u=(K_1-K_2)Du_1+ K_2D(u_1+u_2)+Du_3,\quad\mbox{in}~ \tilde{\Omega}.
\end{equation}
The key step is to show that the singular behavior is captured entirely by $(K_1-K_2)D u_1$, while $K_2D(u_1+u_2)$ and $D u_3$ remain uniformly bounded. These estimates are established in Propositions~\ref{prop_gradient} and~\ref{prop_constant}, whose proofs are given in Subsections~\ref{sec_52} and~\ref{sec_53}, respectively.

Without loss of generality, we assume $\|\varphi\|_{C^{2,\alpha}(\partial \Omega)} \leq 1$; the general case follows by scaling $u/\|\varphi\|_{C^{2,\alpha}(\partial \Omega)}$. 

\begin{prop}\label{prop_gradient}
Let $u_1, u_2 \in H^1(\widetilde{\Omega})$ be the solutions to \eqref{equv1}, and let $u_3 \in H^1(\widetilde{\Omega})$ be the solution to \eqref{equv3} for $n \geq 2$, under assumptions \eqref{fg0}--\eqref{fg} and the normalization $\|\varphi\|_{C^{2,\alpha}(\partial\Omega)} \leq 1$. Let $\gamma_0$ be the constant specified by \eqref{gamma00}. Then, for all $\varepsilon \in (0, 1/4)$ and $\gamma \in (0, \gamma_0)$, the following estimates hold:
\begin{align}
&\|D u_j\|_{L^{\infty}(\tilde{\Omega}\backslash\Omega_R)}\leq C,\quad j=1,2,3,\label{gradient_outside}\\
&\|D(u_1+u_2)\|_{L^{\infty}(\tilde{\Omega})}\leq C,\quad\|D u_3\|_{L^{\infty}(\tilde{\Omega})}\leq C,\label{gradient_u1+u2}
\end{align}
and
\begin{align}
\frac{1}{C}\frac{1}{\gamma+\eta(x')}-C\leq |D u_j(x)|\leq \frac{C}{\gamma+\eta(x')}+C \quad\mbox{for}~x\in\,\Omega_R,~ j=1,2,\label{gradient_u1}
\end{align}
where the constant $C$ depends only on $n$, $R$, $\alpha$ and $\kappa$, a lower bound of $\gamma_0-\gamma$ and upper bounds of $\|f_{1}\|_{C^{2,\alpha}}$ and $\|f_{2}\|_{C^{2,\alpha}}$.
\end{prop}
Thus, the decomposition \eqref{decomposition2} isolates the singular contributions of $Du_{1}$ and $Du_{2}$ from the regular terms $D(u_{1}+u_{2})$ and $D u_{3}$.

The constants $K_{1}$ and $K_{2}$ are determined by the zero-flux conditions
$$\int_{\partial D_j} \partial_\nu u \, dS = 0\quad j=1,2,$$
which lead to the $2 \times 2$ linear system:
\begin{equation}\label{equfreconstant0}
\left\{ \begin{aligned}
a_{11}K_1+a_{12} K_2+b_1=0,\\
a_{21}K_1+a_{22} K_2+b_2=0,
\end{aligned} \right.
\end{equation}
where 
\begin{equation}\label{definitionaijbi0}
a_{ij}=\int_{\partial{D_i}} \frac{\partial u_j}{\partial \nu} dS,\quad \quad b_i=\int_{\partial{D_i}}\frac{\partial u_3}{\partial \nu}dS,\quad i,j=1,2.
\end{equation}

Lemma~\ref{freeconstant1} below shows that the diagonal entries $a_{ii}$ are positive, the off-diagonals $a_{ij} (i\neq j)$ are negative, and $a_{1j}+a_{2j}$ is bounded away from zero, ensuring the system is well-posed. The magnitude of $|K_{1}-K_{2}|$ is controlled by the size of $a_{ii}$. Proposition \ref{aii} provides upper and lower bounds for $a_{ii}$ using the gradient estimates from Proposition~\ref{prop_gradient}. 

The following estimates for $|K_{1} - K_{2}|$ partially cancel the singularity of $Du_{1}$, yielding sharp bounds for $Du$.

\begin{prop}\label{prop_constant}
Under the hypotheses of Proposition~\ref{prop_gradient}, we have $|K_1| + |K_2| \leq C$. Furthermore, for $n \geq 2$, $\varepsilon \in (0, 1/4)$, and $\gamma \in (0, \gamma_0)$, the following estimate holds:
\begin{equation*}
|K_1-K_2|\leq\,C\rho_{n}(\varepsilon,\gamma),~\mbox{where}~\rho_{n}(\varepsilon,\gamma):=
\left\{ \begin{aligned}
&\sqrt{\varepsilon+\gamma} &n=2,\\
&\dfrac{1}{|\ln(\varepsilon+\gamma)|} &n=3,\\
&1&n\geq 4,
\end{aligned} \right.
\end{equation*}
where the constant $C$ depends only on $n$, $R$, $\alpha $, $\kappa$, a lower bound of $\gamma_0-\gamma$ and upper bounds of $\|f_{1}\|_{C^{2,\alpha}}$ and $\|f_{2}\|_{C^{2,\alpha}}$.
\end{prop}

\subsection{Proof of Theorem \ref{upperbound}}

\begin{proof}[Proof of Theorem \ref{upperbound}]
Substitute the gradient estimates from Proposition~\ref{prop_gradient} and the constant estimates from Proposition~\ref{prop_constant} into the decomposition \eqref{decomposition2}:
\begin{equation*}
|D u|\leq |K_1-K_2|\,|D u_1|+| K_2|\,|D(u_1+u_2)|+|D u_3|.
\end{equation*}

In the narrow region $\tilde{\Omega}$, the term $|K_1-K_2|\,|D u_1(x)|$ is bounded by 
$$\frac{C\rho_{n}(\varepsilon,\gamma)}{\gamma+\varepsilon+|x'|^{2}},$$ yielding the dimension-dependent blow-up rate.  The terms $| K_2|\,|D(u_1+u_2)|$ and $|D u_3|$ are uniformly bounded. Outside the narrow region, all terms are uniformly bounded. Combining these cases completes the proof of Theorem \ref{upperbound}.
\end{proof}

\subsection{Proof of Proposition \ref{prop_gradient}}\label{sec_52}

This subsection is devoted to establishing pointwise gradient estimates for the functions $u_j~ (j=1,2,3)$, distinguishing between the narrow region $\Omega_R$ and the bulk $\tilde{\Omega}\backslash\Omega_R$. We begin by proving uniform boundedness of the solutions by the maximum principle, taking careful account of the Robin boundary conditions.

\begin{lemma}\label{Linftyui}
Under the assumption of Proposition \ref{prop_gradient}, we have
$$\|u_j\|_{L^{\infty}(\tilde{\Omega})}\leq 1,\quad\mbox{for}~j=1,2,3.$$
\end{lemma}

\begin{proof}
We first show $\|u_1\|_{L^{\infty}(\tilde{\Omega})}\leq 1$; the case for $u_{2}$ follows by symmetry.

Suppose, for contradiction, that
$$u_1(x_0)=\underset{x\in \bar{\tilde{\Omega}}}{\min}\, u_1(x)<0\quad~\mbox{for~ some}~ x_0\in \bar{\tilde{\Omega}}.$$ By the strong maximum principle, $x_0\in \partial D_1\cup\partial D_2$. At this point, we would have $ u_1(x_0)+\gamma\frac{\partial u_1}{\partial \nu}(x_0)<0$, which contradicts the boundary condition $ u_1+\gamma\frac{\partial u_1}{\partial \nu}=\delta_{1j}$ on $\partial D_j$. Hence, $\underset{x\in \bar{\tilde{\Omega}}}{\min}\, u_1(x)\geq 0$. A similar argument shows that $\underset{x\in \bar{\tilde{\Omega}}}{\max}\, u_1(x)\leq 1$.

To bound $u_3$, let $\tilde{u}_3\in H^1(\tilde{\Omega})$ be the solution to 
\begin{equation*}
\left\{ \begin{aligned}
\Delta \tilde{u}_3&=0 &\mbox{in}& ~\tilde{\Omega}\\
\tilde{u}_3+\gamma\partial_{\nu}\tilde{u}_3&=0&\mbox{on}&~\partial D_1\cup \partial D_2,\\
\tilde{u}_3&=1 &\mbox{on}&~\partial \Omega.
\end{aligned} \right.
\end{equation*}
The same reasoning gives $0\leq \tilde{u}_3(x)\leq 1$. Define
$$\hat{u}_3=u_3-\|\varphi\|_{C^{2,\alpha}(\partial{\Omega})}\tilde{u}_3.$$ Then $\hat{u}_3$ satisfies
\begin{equation*}
\left\{ \begin{aligned}
\Delta \hat{u}_3&=0 &\mbox{in}&~\tilde{\Omega},\\
\hat{u}_3+\gamma\partial_{\nu}\tilde{u}_3&=0 &\mbox{on}&~\partial D_1\cup \partial D_2,\\
\hat{u}_3&\leq0&\mbox{on}&~\partial \Omega.
\end{aligned} \right.
\end{equation*}
The maximum principle implies $\hat{u}_3(x)\leq 0$, so $u_3(x)\leq \|\varphi\|_{C^{2,\alpha}(\partial{\Omega})}\tilde{u}_3(x)$ for $x\in \bar{\tilde{\Omega}}$. Replacing $\hat{u}_3$ by $\hat{u}_3=-u_3-\|\varphi\|_{C^{2,\alpha}(\partial{\Omega})}\tilde{u}_3$ yields the lower bound. Thus,
$$\|u_3\|_{L^{\infty}(\tilde\Omega)}\leq \|\varphi\|_{C^{2,\alpha}(\partial{\Omega})}\leq 1.$$
The lemma is proved.
\end{proof}

\begin{proof}[Proof of Proposition \ref{prop_gradient}]

Let $\varepsilon_0$ be as in Theorem \ref{boundedgradientw}. 

For $\varepsilon \in [\varepsilon_0, 1/4)$, the estimates \eqref{gradient_outside}--\eqref{gradient_u1} follow directly from Theorem \ref{gradientsmallgamma3} and Lemma \ref{Linftyui}. 

We now focus on the essential case $\varepsilon \in (0, \varepsilon_0)$. Here, the exterior estimate \eqref{gradient_outside} again follows from Theorem \ref{gradientsmallgamma3} and Lemma \ref{Linftyui}.  To establish the gradient estimates in the narrow region $\Omega_{R}$, we apply Lemma \ref{Linftyui} and Theorem \ref{boundedgradientw} with $\mathfrak{R}=0$, $\Psi_1=\Psi_2=0$ to $1 - u_1 - u_2$ and $u_3$, which yields \eqref{gradient_u1+u2}.

 The core of the proof is to establish estimate \eqref{gradient_u1} in $\Omega_R$ for small $\varepsilon \in (0, \varepsilon_0)$. We present the details for $u_1$ only, since the case of $1 - u_2$ is identical. Define
\begin{equation}\label{def_tildew}
\tilde{w}_1=\frac{x_n-f_{2}(x')+\gamma}{\delta(x')+2\gamma}.
\end{equation}
To prove~\eqref{gradient_u1}, it suffices to show
\begin{equation}\label{u1minustildew1}
\|D(u_1-\tilde{w}_{1})\|_{L^{\infty}(\Omega_{R})}\leq C.
\end{equation}

Let $w_1:=u_1-\tilde{w}_1$. We verify that $w_1$satisfies equation~\eqref{equw} with right-hand side and boundary data satisfying the assumption~\eqref{estH}, so that Theorem~\ref{boundedgradientw} applies. 

We begin by some basic computation. Under assumption~\eqref{fg}, the following estimates hold for $i = 1, 2$: 
\begin{equation}\label{estfi}
\begin{aligned}
&\, |f_i(x')|\leq C|\eta(x')|,\quad |Df_i(x')|\leq C\eta(x')^{1/2},\quad |D^2f_i(x')|\leq C,\\
&\, [f_i]_{C^{\alpha}({\Omega_{\eta(x_0')}(x_0)})}\leq C |Df_i|_{L^{\infty}({\Omega_{\eta(x_0')}(x_0)})} \eta(x_0')^{1-\alpha}\leq C\eta(x_0')^{\frac{3}{2}-\alpha },\\
&\, [Df_i]_{C^{\alpha}({\Omega_{\eta(x_0')}(x_0)})}\leq C\eta(x_0')^{1-\alpha }.
\end{aligned}
\end{equation}
These imply, for $i=1,2$,
\begin{equation}\label{estsqrtfi}
\begin{aligned}
&\,|(1+|Df_i(x')|^2)^{-1/2}-1|\leq C\eta(x'),\quad 1\leq (1+|Df_i(x')|^2)^{-1/2}\leq C,\\ 
&\, |D(1+|Df_i(x')|^2)^{-1/2}|\leq C\eta(x')^{1/2}, \quad [(1+|Df_i(x')|^2)^{-1/2}]_{C^{\alpha}({\Omega_{\eta(x_0')}(x_0)})}\leq \eta(x')^{\frac{3}{2}-\alpha},\\
&\, [D(1+|Df_i(x')|^2)^{-1/2}]_{C^{\alpha}({\Omega_{\eta(x_0')}(x_0)})}\leq C.
\end{aligned}
\end{equation}

We now verify that $w_1$ satisfies \eqref{equw} with right-hand side and boundary data satisfying the assumptions \eqref{estH}. By virtue of \eqref{estfi}, a direct computation yields, for $1 \leq i \leq n-1$,
\begin{equation}\label{partialiwww}
\begin{aligned}
&\partial_{i} \tilde{w}_1 =\frac{-\partial_{i}f_{2}(\delta(x')+2\gamma)-(x_n-f_{2}+\gamma)\partial_{i}(f_1-f_2)}{(\delta(x')+2\gamma)^2},\\
&\partial_{n} \tilde{w}_1 =\frac{1}{\delta(x')+2\gamma},\quad\partial_{nn} \tilde{w}_1 =0;\\
&|D_{x'}\tilde{w}_{1}(x)|\leq \frac{C\delta(x')^{1/2}}{\gamma+\delta(x')},\quad |\partial_{n}\tilde{w}_{1}(x)|\leq \frac{C}{\gamma+\delta(x')},\\
&|D_{x'} ^2\tilde{w}_{1}(x)|\leq \frac{C}{\gamma+\delta(x')},\quad |D ^2\tilde{w}_{1}(x)|\leq \frac{C\delta(x')^{1/2}}{(\gamma+\delta(x'))^2},\\
& [D\tilde{w}_{1}(x)]_{C^{\alpha}({\Omega_{\eta(x_0')}(x_0)})}\leq \frac{C\delta(x')^{\frac{3}{2}-\alpha}}{(\gamma+\delta(x'))^2},\quad [D^{2}\tilde{w}_{1}(x)]_{C^{\alpha}({\Omega_{\eta(x_0')}(x_0)})}\leq \frac{C\delta(x_0')^{1-\alpha}}{(\delta(x_0')+\gamma)^2}.
\end{aligned}
\end{equation}
Since $\Delta u_1=0$ in $\Omega_{R}$, by \eqref{partialiwww} we have 
\begin{equation}\label{Deltaw1}
|\Delta w_1(x)|\leq\sum_{i=1}^{n-1}|\partial_{ii}\tilde{w}_1|
\leq \frac{C}{\delta(x')+\gamma}.
\end{equation}
Define the functions
\begin{equation*}
\begin{aligned}
&\Psi_{1}(x):=-\gamma\Big(\big({(1+|D_{x'}f_{1}|^2)}^{-1/2}-1\big)\partial_n\tilde{w}_1-{(1+|D_{x'}f_{1}|^2)}^{-1/2}(D_{x'}\tilde{w}_1\cdot D_{x'}f_{1})\Big), \\
&\Psi_{2}(x):=\gamma\Big(\big( {(1+|D_{x'}f_{2}|^2)}^{-1/2}-1\big)\partial_n\tilde{w}_1- {(1+|D_{x'}f_{2}|^2)}^{-1/2}(D_{x'}f_{2}\cdot D_{x'}\tilde{w}_1)\Big).
\end{aligned}
 \end{equation*}
On $\Gamma_{1,2R}$, a direct calculation gives 
\begin{equation*}
\begin{aligned}
&\,w_1+\gamma\partial_{\nu}w_1
=\,(u_1+\gamma\partial_{\nu}u_1)-(\tilde{w}_1+\gamma\partial_{\nu}\tilde{w}_1)\\
&=1-\Big(\frac{\delta(x')+\gamma}{\delta(x')+2\gamma}+\gamma {(1+|D_{x'}f_{1}|^2)}^{-1/2}\big(\partial_n\tilde{w}_1-D_{x'}\tilde{w}_1\cdot D_{x'} f_1(x')\big)\Big)\\
&=1-\Big(1+\gamma\Big(\big({(1+|D_{x'}f_{1}|^2)}^{-1/2}-1\big)\partial_n\tilde{w}_1-{(1+|D_{x'}f_{1}|^2)}^{-1/2}(D_{x'}\tilde{w}_1\cdot D_{x'}f_{1})\Big)\Big)\\
&= \Psi_{1}(x).
\end{aligned}
\end{equation*}
Similarly, ${w}_1+\gamma\partial_{\nu}{w}_1= \Psi_{2}(x)$ on $\Gamma_{2,2R}$.

Using~\eqref{estfi}, \eqref{estsqrtfi}, and~\eqref{partialiwww}, we obtain, for $i=1,2$, 
\begin{equation}\label{Psirobin}
|\Psi_i(x)|\leq C\gamma\delta(x')|\partial_{n}\tilde{w}_{1}| +C\gamma |D_{x'}\tilde{w}_1||D_{x'}f_i|\leq C\frac{\gamma\delta(x')}{\gamma+\delta(x')},
\end{equation}
and
\begin{equation}\label{nablaPsirobin}
\begin{aligned}
|D\Psi_{i}(x')|\leq&\, C\gamma \eta(x')^{1/2}|\partial_n\tilde{w}_1|+C\gamma\eta(x')|D^{2}\tilde{w}_1|+C\gamma\eta(x')^{1/2}|D_{x'}\tilde{w}_1||D_{x'}f_i|\\
&+C\gamma |D^{2}\tilde{w}_1||D_{x'}f_i|+C\gamma |D_{x'}\tilde{w}_1||D^{2} f_i|\\
\leq&\, C\frac{\gamma}{\gamma+\delta(x')};\\
[D\Psi_{i}]_{C^{\alpha}({\Omega_{\eta(x_0')}(x_0)})}\leq&\, C\gamma\eta(x_0')^{1/2}[D^2\tilde{w}_{1}]_{C^{\alpha}({\Omega_{\eta(x_0')}(x_0)})}+C\gamma\eta(x_0')^{1-\alpha}\|D^2\tilde{w}_1\|_{L^{\infty}({\Omega_{\eta(x_0')}(x_0)})}\\
&+C\gamma[D\tilde{w}_{1}]_{C^{\alpha}({\Omega_{\eta(x_0')}(x_0)})}+ C\gamma\|D\tilde{w}_1\|_{L^{\infty}({\Omega_{\eta(x_0')}(x_0)})} \\
\leq&\, C\Big(\frac{\gamma\delta(x_0')^{\frac{3}{2}-\alpha}}{(\gamma+\delta(x_0'))^2}+\frac{\gamma}{\gamma+\delta(x_0')}\Big).
\end{aligned}
\end{equation}
From \eqref{Deltaw1}, \eqref{Psirobin}, and \eqref{nablaPsirobin}, we conclude that
$$|\Delta w_1(x)|\leq \frac{C_1}{\gamma+\delta(x')}\,\,\, \mbox{and}~\,\|\Psi_{i}\|^{*}_{C^{1,\alpha}({\Omega_{\eta(x_0')}(x_0)})} \leq C_1\frac{\gamma\delta(x_0')}{\gamma+\delta(x_0')}.$$
By Lemma \ref{Linftyui}, we have $\|w_1\|_{L^{\infty}(\Omega_{2R})}\leq 2$. Theorem~\ref{boundedgradientw} now applies to $\frac{w_1}{C_1+2}$, and \eqref{u1minustildew1} is proved. Combining with \eqref{partialiwww}, we have \eqref{gradient_u1}.
\end{proof}

\begin{remark}
The proof of Proposition~\ref{prop_gradient} actually yields the stronger estimate:
\begin{equation*}
\Big|D (u_1(x)-\frac{x_n-f_{2}(x')+\gamma}{\delta(x')+2\gamma})\Big|\leq C\quad \forall x\in\Omega_{R/2}.
\end{equation*}
\end{remark}

\subsection{Proof of Proposition \ref{prop_constant}}\label{sec_53}

To bound the free constants $K_1$ and $K_2$ and compute $|K_1-K_2|$, we require several basic properties of the coefficients $a_{ij}$ and $b_i$ defined in \eqref{definitionaijbi0}.

\begin{lemma}\label{freeconstant1}
Let $a_{ij}$ and $b_i$ be defined by \eqref{definitionaijbi0}. Then 
 
(1) For any $\gamma\in(0,\infty)$, we have $ a_{ii}>0$ and $a_{ij}<0,\,\mbox{for}~i\neq j,$

(2) For any $M>0$ and $\gamma\in (0,M)$,  
$$
\frac{1}{C}<\sum_{i=1}^2a_{ij}<C\,\,\,\mbox{for}~j=1,2,
$$
where $C$ depend only on $n$, $R$, $\alpha $ and $\kappa$, $M$, $\|f_{1}\|_{C^{2,\alpha}}$, and $\|f_{2}\|_{C^{2,\alpha}}$.

(3) For $\gamma\in(0,\infty)$, 
$$\|\nabla u_3\|_{L^{\infty}(\tilde{\Omega})}\leq C\|\varphi\|_{C^{2,\alpha}(\tilde{\Omega})},\quad |b_i|<C\|\varphi\|_{C^{2,\alpha}(\tilde{\Omega})}\quad\mbox{for}~i=1,2.$$
\end{lemma}

\begin{proof}
(1) It suffices to prove the inequalities: 
\begin{equation}\label{normalderivativev1}
\frac{\partial u_1}{\partial \nu} {\Big|}_{\partial D_1}>0,\quad \frac{\partial u_1}{\partial \nu} {\Big|}_{\partial D_2}<0,\quad 
\frac{\partial u_1}{\partial \nu} {\Big|}_{\partial \Omega}<0.
\end{equation}
Let $x_0\in\partial D_1$. We claim $ u_1(x_0)<1$. Suppose, for contradiction, that $u_1(x_0)= 1$. Since $0\leq u_1\leq 1$ by lemma \ref{Linftyui}, the function $u_1$ attains its maximum at $x_{0}$. By the strong maximum principle and the Hopf lemma, $\frac{\partial u_1}{\partial\nu}(x_0)>0$, which implies $ u_1(x_0)+\frac{\partial u_1}{\partial\nu}(x_0)>1$, contradicting the boundary condition. Hence, $ u_1(x_0)<1$, and the first inequality in \eqref{normalderivativev1} follows. The second inequality can be obtained by proving $u_1(x_0)>0$ for any $x_0\in\partial D_2$. Finally, since $u_1$ attains its minimum value on the boundary $\partial \Omega$, the Hopf lemma yields the last estimate.

A similar argument shows that $0< u_2<1$ on $\partial D_1\cup \partial D_2$ and 
\begin{equation*}
 \frac{\partial u_2}{\partial \nu} {\Big|}_{\partial D_1}<0,\quad \frac{\partial u_2}{\partial \nu} {\Big|}_{\partial D_2}>0,\quad 
 \frac{\partial u_2}{\partial \nu} {\Big|}_{\partial \Omega}<0.
\end{equation*}

(2) Since $0 \leq u_1 \leq 1$, Theorem~\ref{gradientsmallgamma3} implies that $\|\nabla u_1\|_{L^{\infty}(\Omega\setminus\Omega_R)} \leq C_0$. 
Choose a universal constant $\gamma_1 \in (0, \gamma_0)$ such that $C_0 \gamma_1 < \frac{1}{2}$. 
We now consider two cases:

\textbf{Case 1: $\gamma \in (0, \gamma_1)$}.  
In this case, the boundary condition on $\partial D_1$ implies that $u_1|_{\partial D_1 \setminus \Omega_R} > \frac{1}{2}$. 
Applying Theorem~\ref{gradientsmallgamma3} to $u_1$ in $\tilde{\Omega} \setminus \Omega_R$, we find a point $\bar{x} \in \tilde{\Omega} \setminus \Omega_{2R}$ 
and a radius $r > 0$ (independent of $\gamma$ and $\varepsilon$) such that $u_1 > \frac{1}{4}$ in $B(\bar{x}, 2r) \subset \tilde{\Omega} \setminus \Omega_{2R}$. 
Let $D$ be a fixed smooth domain containing both $D_1$ and $D_2$, with $\bar{x} \in \partial D$. 
Define a smooth function $\phi$ on $\partial D$ such that $\phi = \frac{1}{4}$ on $B(\bar{x}, r) \cap \partial D$, 
$0 \leq \phi \leq \frac{1}{4}$ on $\big(B(\bar{x}, 2r) \setminus B(\bar{x}, r)\big) \cap \partial D$, 
and $\phi = 0$ on $\partial D \setminus \overline{B(\bar{x}, 2r)}$.

Let $\tilde{u}_1 \in H^1(\Omega \setminus \overline{D})$ be the solution to  
$$
\begin{cases}
\Delta \tilde{u}_1 = 0 & \text{in } \Omega \setminus \overline{D}, \\
\tilde{u}_1 = \phi & \text{on } \partial D, \\
\tilde{u}_1 = 0 & \text{on } \partial \Omega.
\end{cases}
$$
By the strong maximum principle, $0 < \tilde{u}_1 < \frac{1}{4}$ in $\Omega \setminus \overline{D}$. 
Moreover, the Hopf lemma implies $\frac{\partial \tilde{u}_1}{\partial \nu}\big|_{\partial \Omega} < -\frac{1}{C}$, 
where $C > 0$ is a constant independent of $\gamma$ and $\varepsilon$.

Since $u_1 \geq \phi = \tilde{u}_1$ on $\partial D$ and $u_1 = \tilde{u}_1 = 0$ on $\partial \Omega$, 
the maximum principle implies $u_1 - \tilde{u}_1 \geq 0$ in $\Omega \setminus \overline{D}$. 
The harmonic function $u_1 - \tilde{u}_1$ attains its minimum on $\partial \Omega$, so  
$$
\frac{\partial (u_1 - \tilde{u}_1)}{\partial \nu} \bigg|_{\partial \Omega} \leq 0.
$$  
Recalling that  
$$
a_{11} + a_{21} = -\int_{\partial \Omega} \frac{\partial u_1}{\partial \nu}  dS,
$$  
we complete the proof of part (2) in this case.

\textbf{Case 2: $\gamma \in [\gamma_1, M)$}.  
Rewrite $u_1=u_{1, \gamma, \varepsilon}$, which is also the unique minimizer of the strictly convex functional  
\begin{equation*}
I_{\gamma,\varepsilon}[w] := \int_{\tilde{\Omega}^{\varepsilon}} |\nabla w|^2  \,dx + \frac{1}{\gamma} \int_{\partial D_1^{\varepsilon}} (w - 1)^2  \,dS + \frac{1}{\gamma} \int_{\partial D_2^{\varepsilon}} w^2  \,dS
\end{equation*}  
over the space $\sA := \{ w \in H^1(\tilde{\Omega}^{\varepsilon}) : w = 0 \text{ on } \partial \Omega \}$. 
For any $\gamma \in [\gamma_1, M]$ and $\varepsilon \in [0, \varepsilon_0]$, we have  
$$
\|\nabla u_{1, \gamma, \varepsilon}\|^2_{L^2(\tilde{\Omega}^{\varepsilon})} \leq I_{\gamma,\varepsilon}(u_{1,\gamma,\varepsilon}) \leq I_{\gamma,\varepsilon}[0] = \frac{1}{\gamma} \int_{\partial D_1^{\varepsilon}} 1  \,dS \leq \frac{|\partial D_1^{\varepsilon}|}{\gamma_1}.
$$  
Together with the Dirichlet condition $u_{1,\gamma,\varepsilon} = 0$ on $\partial \Omega$, this yields 
$$
\|u_{1, \gamma, \varepsilon}\|^2_{H^1(\tilde{\Omega}^{\varepsilon})} \leq C.
$$  
We may extend $u_{1, \gamma, \varepsilon}$ to $H^1(\Omega)$ while maintaining  $\|u_{1, \gamma, \varepsilon}\|_{H^1(\Omega)} \leq C$.

We now claim that for any $\varepsilon \in [0, \varepsilon_0]$ and any $\gamma \in [\gamma_1, M]$,  
\begin{equation}\label{aiiclaim}
\int_{\partial \Omega} \frac{\partial u_{1,\gamma,\varepsilon}}{\partial \nu}  \,dS < -\frac{1}{C},
\end{equation}  
where $C$ is independent of $\varepsilon$ and $\gamma$. 
Suppose, for contradiction, that this fails. Then there exist sequences $\{\varepsilon_j\} \subset [0, \varepsilon_0]$, $\{\gamma_j\} \subset [\gamma_1, M]$ with $\gamma_j \to \gamma^*$, $\varepsilon_j \to \varepsilon^*$, and a function $u^* \in H^1(\Omega)$ such that  
\begin{align}
&\int_{\partial \Omega} \frac{\partial u_{1,\gamma_j,\varepsilon_j}}{\partial \nu}  \,dS \to 0,\label{eq13.50}\\
&u_{1,\gamma_j,\varepsilon_j} \rightharpoonup u^*\,\, \text{weakly in }H^1(\Omega).\label{eq13.47}
\end{align}
Since $u_{1,\gamma_j,\varepsilon_j}$ minimizes $I_{\gamma_j, \varepsilon_j}$, for any $\phi \in H_0^1(\Omega)$ we have  
$$
\int_{\tilde{\Omega}^{\varepsilon_j}} \nabla u_{1,\gamma_j,\varepsilon_j} \cdot \nabla \phi  \,dx + \frac{1}{\gamma_j} \int_{\partial D_1^{\varepsilon_j}} (u_{1,\gamma_j,\varepsilon_j} - 1) \phi  \,dS + \frac{1}{\gamma_j} \int_{\partial D_2^{\varepsilon_j}} u_{1,\gamma_j,\varepsilon_j} \phi  \,dS = 0.
$$  
Passing to the limit $j \to +\infty$ and using \eqref{eq13.47}, we obtain  
$$
\int_{\tilde{\Omega}^{\varepsilon^*}} \nabla u^* \cdot \nabla \phi  \,dx + \frac{1}{\gamma^*} \int_{\partial D_1^{\varepsilon^*}} (u^* - 1) \phi  \,dS + \frac{1}{\gamma^*} \int_{\partial D_2^{\varepsilon^*}} u^* \phi  \,dS = 0,
$$  
so that $u^* = u_{1, \gamma^*, \varepsilon^*}$ in $\tilde{\Omega}^{\varepsilon^*}$.

Now, by \eqref{eq13.47}, the divergence theorem, and \eqref{eq13.50}, 
$$
\begin{aligned}
&\int_{\partial D_1^{\varepsilon^*}} \frac{u_{1,\gamma^*, \varepsilon^*} - 1}{\gamma^*}  \,dS + \int_{\partial D_2^{\varepsilon^*}} \frac{u_{1,\gamma^*, \varepsilon^*}}{\gamma^*}  \,dS \\
&= \lim_{j \to +\infty} \left[ \int_{\partial D_1^{\varepsilon_j}} \frac{u_{1,\gamma_j, \varepsilon_j} - 1}{\gamma_j}  \,dS + \int_{\partial D_2^{\varepsilon_j}} \frac{u_{1,\gamma_j, \varepsilon_j}}{\gamma_j}  \,dS \right] \\
&= - \lim_{j \to +\infty} \int_{\partial D_1^{\varepsilon_j} \cup \partial D_2^{\varepsilon_j}} \frac{\partial u_{1,\gamma_j, \varepsilon_j}}{\partial \nu}  \,dS \\
&= -\lim_{j \to +\infty} \int_{\partial \Omega} \frac{\partial u_{1,\gamma_j, \varepsilon_j}}{\partial \nu}  \,dS = 0.
\end{aligned}
$$  
However, from part (1) we know $u_{1, \gamma^*, \varepsilon^*}<1$ on $\partial D_1^{\varepsilon^{*}}$ and $u_{1, \gamma^*, \varepsilon^*}>0$ on $\partial D_{2}^{\varepsilon^{*}}$, contradicting the above.
Thus, the claim \eqref{aiiclaim} is proved.

(3) The estimate for $b_i$ follows directly from its definition and Proposition~\ref{prop_gradient}.
\end{proof}

\begin{prop}\label{aii}
For $i = 1, 2$, let $a_{ii}$ be defined as in \eqref{definitionaijbi0}, and let $\gamma_0$ be given by \eqref{gamma00}. Then, for all $\varepsilon \in (0, 1/4)$ and $\gamma \in (0, \gamma_0)$,
\begin{equation}\label{aii111}
\frac{1}{C \, \rho_{n}(\varepsilon,\gamma)} \leq a_{ii} \leq \frac{C}{\rho_{n}(\varepsilon,\gamma)},
\end{equation}
where $\rho_{n}(\varepsilon,\gamma)$ is the function from Proposition~\ref{prop_constant}. The constant $C$ depends only on $n$, $R$, $\alpha$, $\kappa$, a lower bound of $\gamma_0-\gamma$ and upper bounds of $\|f_{1}\|_{C^{2,\alpha}}$ and $\|f_{2}\|_{C^{2,\alpha}}$.
\end{prop}

\begin{proof}
We compute only $a_{11}$; the case of $a_{22}$ is analogous.  
 
We first establish the estimate
\begin{equation}\label{aii11}
 \left| a_{11} - \int_{|x'|\leq R/2} \frac{1}{\delta(x') + 2\gamma}  dx' \right| \leq C.
\end{equation}
The proof relies on an application of the divergence theorem. Introduce a test function $\psi_1(x) \in C^2(\bar{\tilde{\Omega}})$ satisfying
$$
\psi_1(x) = 1 \text{ on } \partial D_1, \quad \psi_1(x) = 0 \text{ on } \partial D_2 \cup \partial \Omega,
$$
with $\|\psi_1\|_{C^2(\tilde{\Omega} \setminus \Omega_{R/2})} \leq C$, and which varies linearly in $x_n$ in the neck region:
$$
\psi_1(x) := \frac{x_n - f_2(x')}{\delta(x')} = \frac{x_n - f_2(x')}{\varepsilon + f_1(x') - f_2(x')} \quad \text{in } \Omega_R.
$$
By the divergence theorem, the flux $a_{11}$ can be expressed as
\begin{equation*}
 a_{11} = \int_{\partial D_1} \frac{\partial u_1}{\partial \nu}  dS 
     = \int_{\partial \tilde{\Omega}} \frac{\partial u_1}{\partial \nu} \psi_1  dS 
 = \int_{\tilde{\Omega}} \nabla \cdot (\psi_1 \nabla u_1)  dx 
 = \int_{\tilde{\Omega}} \nabla u_1 \cdot \nabla \psi_1  dx.
\end{equation*}
  
Since $\|Du_1\|_{L^{\infty}(\tilde{\Omega} \setminus \Omega_{R/2})} \leq C$, we obtain
\begin{equation*}
 a_{11} = \int_{\Omega_{R/2}} \nabla u_1 \cdot \nabla \psi_1  dx + \int_{\tilde{\Omega} \setminus \Omega_{R/2}} \nabla u_1 \cdot \nabla \psi_1  dx = \int_{\Omega_{R/2}} \nabla u_1 \cdot \nabla \psi_1  dx + O(1),
\end{equation*}
where $A = B + O(1)$ means $|A - B| \leq C$.
 
In the narrow region, we use the decomposition $u_1 = \tilde{w}_1 + w_1$, where $\tilde{w}_1$ is given by \eqref{def_tildew}:
$$\int_{\Omega_{R/2}} \nabla u_1 \cdot \nabla \psi_1  dx 
 = \int_{\Omega_{R/2}} \nabla w_1 \cdot \nabla \psi_1  dx 
 + \int_{\Omega_{R/2}} \nabla \tilde{w}_1 \cdot \nabla \psi_1  dx.$$
The term involving $\nabla w_1$ is bounded: since $\nabla w_1$ is bounded (see \eqref{u1minustildew1}) and $|\nabla \psi_1(x)| \leq C / \delta(x')$, we have
 $$\left| \int_{\Omega_{R/2}} \nabla w_1 \cdot \nabla \psi_1  dx \right| 
 \leq C \int_{\Omega_{R/2}} \frac{1}{\delta(x')}  dx \leq C.$$
For the term involving $\nabla \tilde{w}_1$,  by using
 $$|D_{x'} \tilde{w}_1(x)| + |D_{x'} \psi_1(x)| \leq \frac{C}{\delta(x')^{1/2}},$$
 we can estimate the tangential part $$\left|\int_{\Omega_{R/2}} D_{x'} \tilde{w}_1 \cdot D_{x'} \psi_1  dx\right| \leq C \int_{\Omega_{R/2}} \frac{1}{\delta(x')}  dx \leq C. $$
A direct computation shows that
\begin{equation*}
    \int_{\Omega_{R/2}} \partial_n \tilde{w}_1 \cdot \partial_n \psi_1  dx 
  = \int_{\Omega_{R/2}} \frac{1}{\delta(x')(\delta(x') + 2\gamma)}  dx 
  = \int_{|x'| \leq R/2} \frac{1}{\delta(x') + 2\gamma}  dx'. 
\end{equation*}
Combining the above estimates yields \eqref{aii11}.
 
Moreover, for the integral in \eqref{aii11}, we have
\begin{equation}\label{aii12}
 \frac{1}{C \rho_n(\varepsilon,\gamma)} 
 \leq \int_{|x'| \leq R/2} \frac{1}{\delta(x') + 2\gamma}  dx' 
 \leq \frac{C}{\rho_n(\varepsilon,\gamma)}.
\end{equation}
By Lemma  \ref{freeconstant1} (1)-(2), we also have
\begin{equation}\label{aii13}
 a_{11} > a_{11} + a_{21} > \frac{1}{C}.
\end{equation}
 
For $n \geq 4$, the desired estimate follows directly from \eqref{aii11}, \eqref{aii12}, and \eqref{aii13}.
 
For $n = 2, 3$, note that $1 / \rho_n(\varepsilon,\gamma) \to +\infty$ as $\varepsilon + \gamma \to 0$. Therefore, \eqref{aii111} holds when $\varepsilon + \gamma < 1/\tilde{M}$ for some sufficiently large universal constant $\tilde{M}$.
 
When $1/\tilde{M} \leq \varepsilon + \gamma \leq 1/4 + \gamma_0$ in dimensions $n = 2, 3$, we have
$$\frac{1}{C} \leq \frac{1}{\rho_n(\varepsilon , \gamma)} \leq C.$$
In this case, \eqref{aii111} follows again from \eqref{aii11} and \eqref{aii13}.
 
Combining all the above cases completes the proof.
\end{proof}

\begin{proof}[Proof of Proposition \ref{prop_constant}]   Solving the system \eqref{equfreconstant0} via Cramer's rule gives
$$K_1=\dfrac{-b_1a_{22}+b_2a_{12}}{a_{11}a_{22}-a_{12}a_{21}}, \quad K_2=\dfrac{-b_2a_{11}+b_1a_{21}}{a_{11}a_{22}-a_{12}a_{21}},$$
whence
\begin{equation}\label{c1c2C}
K_1-K_2=\dfrac{b_2(a_{12}+a_{11})-b_1(a_{21}+a_{22})}{a_{11}a_{22}-a_{12}a_{21}}.
\end{equation}
 Since 
$$a_{11}a_{22}-a_{12}a_{21}=(a_{22}+a_{12})(a_{11}-\frac{a_{11}+a_{21}}{a_{22}+a_{12}}a_{12}),$$
Lemma~\ref{freeconstant1} (1)–(2) implies $a_{22}>-a_{12}>0$, which together with the bounds in Lemma~\ref{freeconstant1} yields
\begin{equation}\label{c1c2A}
 0<\frac{1}{C}a_{11}< a_{11}a_{22}-a_{12}a_{21}<C(a_{11}+a_{22}).
\end{equation}
By symmetry, we also obtain $$ 0<\frac{1}{C}a_{22}< a_{11}a_{22}-a_{12}a_{21}<C(a_{11}+a_{22}).$$
Using the fact that $a_{22}>|a_{12}|$ and that $b_i(i=1,2)$ are bounded (Lemma~\ref{freeconstant1} (3)), we obtain
 \begin{equation*}
 |K_1|\leq \frac{C|b_1a_{22}|+C|b_2a_{12}|}{a_{22}}\leq C,
 \end{equation*}
and similarly, $| K_2|\leq C$. 

Finally, since the numerator in \eqref{c1c2C} is bounded (as both $b_{i}$ and $\|D(u_1+u_2)\|_{L^{\infty}(\tilde{\Omega})}$ are bounded), it follows that
$$|K_1-K_2|\leq \frac{C}{a_{11}}.$$ The result now follows from Proposition~\ref{aii}.
\end{proof}

\section{Proof of the lower bound}\label{sec_6}

This section is devoted to the proof of Theorem~\ref{lowerbound}, which establishes the optimality of the gradient estimates in Theorem~\ref{upperbound}. Since the decomposition \eqref{decomposition2} of $Du$ contains exactly one singular term, namely $Du_1$, while all other contributions remain uniformly bounded, it suffices to obtain a positive lower bound for $|K_1-K_2|$. This is achieved via the construction of a subsolution under a specific symmetry condition. 

\begin{proof}[Proof of Theorem \ref{lowerbound}]

From \eqref{c1c2A} and \eqref{c1c2C}, we obtain
\begin{equation}\label{lowerbdd}
|K_1-K_2|\geq \frac{1}{C(a_{11}+a_{22})}|b_2(a_{12}+a_{11})-b_1(a_{21}+a_{22})|.
\end{equation}
Under the symmetry assumptions of Theorem~\ref{lowerbound}, the domain $\tilde{\Omega}$ is symmetry with respect to $\{x_n=\frac{\varepsilon}{2}\}$, which implies 
$$a_{11}=a_{22},\quad a_{12}=a_{21}.$$ Moreover, by Lemma~\ref{freeconstant1} (2),
$$\frac{1}{C}<a_{11}+a_{12}=a_{22}+a_{21}=a_{11}+a_{12}<C. $$
Hence, inequality \eqref{lowerbdd} simplifies to
\begin{equation*}
|K_1-K_2|\geq \frac{1}{Ca_{11}}|b_1-b_2|.
\end{equation*}
It therefore suffices to prove that $|b_1-b_2|>\frac{1}{C}$.

To this end, introduce the change of variables
\begin{equation*}
y'=x',\quad\,y_n=x_n-\frac{\varepsilon}{2},
\end{equation*}
and define $w(y)=u_3(x)$. Then $w$ satisfies
\begin{equation*}
\left\{ \begin{aligned}
\Delta w&=0& \mbox{in}&~B_6\backslash (D_1^{\varepsilon}\cup D_2^{\varepsilon}),\\
w+\gamma\partial_{\nu}w&=0& \mbox{on}&~\partial D_1^{\varepsilon}\cup \partial D_2^{\varepsilon},\\
w&=y_n& \mbox{on}&~\partial \Omega,
 \end{aligned} \right.
\end{equation*}
where $B_6$ is the ball of radius $6$ centered at the origin, and $D_1^{\varepsilon}$, $D_2^{\varepsilon}$ are unit balls centered at $(0',1+\frac{\varepsilon}{2})$ and $(0',-1-\frac{\varepsilon}{2})$, respectively. By symmetry, $w(y',y_n)=-w(y',-y_n)$, so $w(y',0)=0$, and 
$$b_1-b_2=\int_{D_1^{\varepsilon}}\frac{\partial w}{\partial \nu}-\int_{D_2^{\varepsilon}}\frac{\partial w}{\partial \nu}=2\int_{D_1^{\varepsilon}}\frac{\partial w}{\partial \nu}.$$ 
We now claim that for any $\varepsilon,\gamma\in (0,\frac{1}{4})$, $-C<\int_{\partial D_1^{\varepsilon}}\frac{\partial w}{\partial \nu}\leq -\frac{1}{C}$ with $C>0$ independent of $\varepsilon$ and $\gamma$. 

To prove this, observe that $w$ solves the boundary value problem in the upper half-ball
\begin{equation*}
\left\{ \begin{aligned}
\Delta w&=0& \mbox{in}&~B_6^{+},\\
w+\gamma\partial_{\nu}w&=0& \mbox{on}&~\partial D_1^{\varepsilon},\\
w&=y_n& \mbox{on}&~\partial B_6^+,
\end{aligned} \right.
\end{equation*}
where $B_6^{+}=\{B_6\}\cap{y_n>0}$. The maximum principle implies $w\geq 0$ in $B_6^+$. 

We now construct a barrier function. Define
$$\hat{w}(y)=c_1(y_n-1-\frac{\varepsilon}{2})(|y'|^2+(y_n-1-\frac{\varepsilon}{2})^2-1)+c_2\gamma(y_n-1-\frac{\varepsilon}{2})^2$$ in the domain $A=B_6^+\cap \{y_n>1+\frac{\varepsilon}{2}\}\backslash \overline{{D}_1^{\varepsilon}}$, with $c_{1}=\frac{1}{500}$ and $c_2=\frac{1}{3000}$. 
A direct computation shows that
\begin{equation*}
\left\{ \begin{aligned}
\Delta \hat{w}&\geq 0& \mbox{in}&~A,\\
\hat{w}&\leq y_n& \mbox{on}&~A_1:=\partial B_6^{+}\cap\{y_n>1+\frac{\varepsilon}{2}\},\\
\hat{w}&=0& \mbox{on}&~A_2:=\{y_n=1+\frac{\varepsilon}{2}\}\cap \big( {B_6^+\backslash D_1^{\varepsilon}}\big),\\
\hat{w}+\gamma\frac{\partial \hat{w}}{\partial \nu}&\leq 0&\mbox{on}&~A_3:=\{y_n>1+\frac{\varepsilon}{2}\}\cap\partial D_1^{\varepsilon}.
\end{aligned} \right.
\end{equation*}
Since $w-\hat{w}$ is superharmonic in $A$, its minimum on $\bar{A}$ is attained on $\partial A$. Let $(w-\hat{w})(x_0)=\underset{x\in \bar{A}}{\min}(w-\hat{w})(x)$ for some $x_0\in\partial A$. Suppose $(w-\hat{w})(x_0)<0$. Since $w-\hat{w}\geq0$ on $A_1\cup A_2$, it follows that $x_0\in A_3$. At this point,
$$\frac{\partial(w-\hat{w})}{\partial\nu}(x_0)\leq 0,\quad\mbox{and}~ (w-\hat{w})(x_0)+\gamma\frac{\partial(w-\hat{w})}{\partial\nu}(x_0)<0,$$ which contradicts the boundary condition $w + \gamma\partial_{\nu}w = 0$ on $A_3$. Hence, $w \geq \hat{w}$ throughout $\overline{A}$.

For $\varepsilon\in(0,\frac{1}{4})$, we find that $w(x)\geq \hat{w}(x)\geq \frac{c_2}{64}\gamma$ on the portion $\partial D_1^{\varepsilon}\cap \{y_n\in(\frac{5}{4},2)\}$. Since $w\geq 0$ in $B_6^{+}$, it follows that
\begin{equation*}
\int_{\partial D_1^{\varepsilon}}\frac{\partial w}{\partial \nu}\,dS=\int_{\partial D_1^{\varepsilon}}-\frac{w}{\gamma}\,dS\leq \int_{\partial D_1^{\varepsilon}\cap \{y_n\in(\frac{5}{4},2)\}}-\frac{w}{\gamma}\,dS\leq \int_{\partial D_1^{\varepsilon}\cap \{y_n\in(\frac{5}{4},2)\}}-\frac{c_2}{64}\,dS\leq -\frac{1}{C}.
\end{equation*}
This, together with the upper bound from Lemma~\ref{freeconstant1}\,(3), yields
$$-C\leq\int_{\partial D_1^{\varepsilon}}\frac{\partial w}{\partial \nu}\,dS\leq -\frac{1}{C},$$
which completes the proof of Theorem \ref{lowerbound}.
\end{proof}

\appendix

\section{The perfect-conductivity limit}

We consider the perfect-conductor model
\begin{equation}\label{eqperfect}
\left\{ \begin{aligned}
\Delta u &=0&\mbox{in}&~\Omega\setminus \overline{(D_1 \cup D_2)},\\
u &= K_j &\mbox{on}&~\partial D_j, j = 1,2,\\
\int_{\partial D_j} \partial_\nu u \, \,dS& = 0 & j&=1,2,\\
u &= \varphi(x) &\mbox{on}&~\partial \Omega,
\end{aligned} \right.
 \end{equation}
where the free constant $K_j$ is uniquely determined by the zero-flux conditions (third line). The existence and uniqueness of the solution to \eqref{eqperfect} follows from \cite{BLY}. We now prove that solutions $u_{\gamma}$ of \eqref{lcctype} converge to the solution $u_0$ of \eqref{eqperfect} as $\gamma\to0$.

\begin{theorem}
Let $u_{\gamma}$ and $u_0$ be the solutions of \eqref{lcctype} and \eqref{eqperfect}, respectively. Then 
\begin{equation*}
u_{\gamma}\rightharpoonup u_0\quad\mbox{in}~H^1(\tilde{\Omega})\quad\mbox{as}~\gamma\to+0,
\end{equation*}
and $$ \underset{\gamma\to0}{\lim}\|u_{\gamma}-u_0\|_{L^2(\Omega\backslash\overline{D_1\cup D_2})}=0.$$
\end{theorem}

\begin{proof}
By the uniqueness of $u_0$, it suffices to show that any subsequence $\{u_{\gamma_j}\}_{j=1}^{j=+\infty}$ with $\gamma_j\to0^{+}$ as $j\to+\infty$ has a subsequence (still denoted by $\{u_{\gamma_j}\}$), such that 
\begin{equation*}
u_{\gamma_j}\rightharpoonup u_0\quad\mbox{in}~H^1(\tilde{\Omega})\quad\mbox{as}~\gamma_j\to+0,
\end{equation*}
and $$ \underset{\gamma_j\to0}{\lim}\|u_{\gamma_{j}}-u_0\|_{L^2(\Omega\backslash\overline{D_1\cup D_2})}=0.$$
The proof is divided into three steps.
 
\textbf{Step 1. Uniform bound.} We first establish a uniform bound $|u_{\gamma}|_{H^1(\tilde{\Omega})} \leq M$ for some constant $M > 0$ independent of $\gamma$.

Indeed, since $u_{\gamma}$ is the minimizer of $I_{\gamma}$ in $\sA := \{ v \in H^{1}(\widetilde\Omega):\, v = \varphi~~\mbox{on}~~\partial\Omega \}$, we have, for any fixed function $\eta\in H_{\varphi}^{1}$ with $\eta=1$ on $\partial D_1\cup \partial D_2$, 
$$\frac{1}{2}\int_{\tilde{\Omega}}|D u_{\gamma}|^2\,dx \leq I_{\gamma}(u_{\gamma})\leq I_{\gamma}(\eta)=\frac{1}{2}\int_{\tilde{\Omega}}|D \eta|^2\,dx.$$
Combining this with the boundary condition $u_{\gamma}=\varphi$ on $\partial\Omega$, we obtain the uniform estimate
$$\|u_{\gamma}\|_{H^1(\tilde{\Omega})}\leq M.$$ Consequently, there exists a subsequence (still denoted by $\{u_{\gamma_j}\}$) that converges weakly in $H^1(\tilde{\Omega})$ to some limit $u_0$.
 
\textbf{Step 2. The limit is harmonic.} We show that $u_0$ is harmonic in $\tilde{\Omega}$.

For any test function $\eta \in C_{c}^{\infty}(\tilde{\Omega})$, the weak formulation gives
$$\int_{\tilde{\Omega}}Du_{\gamma}\cdot D\eta=0.$$ By the weak convergence of $u_{\gamma_j}$ to $u_0$, we may pass to the limit to obtain $\int_{\tilde{\Omega}}Du_{0}\cdot D\eta=0$, which implies $\Delta u_0 = 0$ in $\tilde{\Omega}$.

\textbf{Step 3. Boundary conditions.} We verify that $u_0$ satisfies the boundary conditions of \eqref{eqperfect}.

The energy bound $I_{\gamma}(u_{\gamma})\leq M$ implies $$\sum_{i=1}^2\int_{\partial D_i}|u_{\gamma}-(u_{\gamma})_{\partial D_i}|^2\,dS\leq 2M\gamma,$$ 
where $(u_{\gamma})_{\partial D_i}$ denotes the average of $u_{\gamma}$ over $\partial D_i$. Since $u_{\gamma_j} \rightharpoonup u_0$ in $H^1(\tilde{\Omega})$, the trace operator implies convergence in $L^2(\partial D_1)$. Therefore, for any $\eta \in L^{2}(\partial D_1)$, we have 
$$\int_{\partial D_1}(u_{0}-(u_{0})_{\partial_{D_1}})\eta \,dS=\underset{j\to\infty}{\lim}\int_{\partial D_1}(u_{\gamma_j}-(u_{\gamma_j})_{\partial_{D_1}})\eta \,dS=0.$$
It follows that $u_{0}-(u_{0})_{\partial_{D_1}}=0$ on $\partial D_1$, i.e., $u_0 = K_1$ on $\partial D_1$. A similar argument shows $u_0 = K_2$ on $\partial D_2$.

Finally, we prove the flux conditions. Let $\eta\in C^{\infty}(\bar{\tilde{\Omega}})$ be a test function such that $\eta=1$ on $\partial D_1$ and $\eta=0$ on $\partial\Omega\cup \partial D_2$. Consider the variation $$i(t)=I_{\gamma}(u_{\gamma}+t\eta)\quad (t\in \bR).$$
The first variation condition yields $$i'(0)=\frac{di}{dt}\big|_{t=0}=\int_{\tilde{\Omega}}Du_{\gamma}\cdot D\eta \,dx=0.$$
Using the weak convergence $u_{\gamma_j} \rightharpoonup u_0$ and applying the divergence theorem, we find
$$\int_{\partial D_1}\partial_{\nu}u_0 \,dS=\int_{\tilde{\Omega}}Du_{0}\cdot D\eta \,dx=\underset{j\to\infty}{\lim}\int_{\tilde{\Omega}}Du_{\gamma_j}\cdot D\eta \,dx=0.$$
An analogous argument applied to $\partial D_2$ shows that the flux through $\partial D_2$ also vanishes. This completes the proof that $u_0$ is the solution to \eqref{eqperfect}.
\end{proof}

\section*{Declarations}

\noindent{\bf Data availability.} This is a theoretical research, and no data was used or generated. 

\noindent{\bf Conflict of interest.} The authors declare that they have no conflict of interest.


\begin{thebibliography}{99}

\bibitem{AKLLL}
H. Ammari, H. Kang, H. Lee, J. Lee, and M. Lim, Optimal estimates for the electric field in two dimensions, J. Math. Pures Appl. (9) 88 (2007), no. 4, 307–324.

\bibitem{akllz} H. Ammari, H. Kang, H. Lee, M. Lim, H. Zribi, Decomposition theorems and fine estimates for electrical fields in the presence of closely located circular inclusions, J. Differential Equations 247 (2009), 2897-2912.

\bibitem{AKL} H. Ammari, H. Kang, and M. Lim, Gradient estimates for solutions to the conductivity problem, Math. Ann. 332 (2005), no. 2, 277–286.


\bibitem{ArRo} F. F. T. Araujo and H. M. Rosenberg, The thermal conductivity of epoxy-resin/metal-powder composite materials from 1.7 to 300k, J. Phys. D 9 (1976), no. 4, 665.

\bibitem{BASL}
I. Babu\u{s}ka, B. Andersson, P.J. Smith, and K. Levin, Damage analysis of fiber composites. I. Statistical analysis on fiber scale, Comput. Methods Appl. Mech. Engrg. 172 (1999), no. 1-4, 27–77.

\bibitem{BLY}
E.S. Bao, Y.Y. Li, B. Yin, Gradient estimates for the perfect conductivity problem. Arch. Ration. Mech. Anal. 193, 195–226 (2009).

\bibitem{BLY2}
E. Bao, Y.Y. Li, and B. Yin, Gradient estimates for the perfect and insulated conductivity problems with multiple inclusions, Comm. Partial Differential Equations 35 (2010), no. 11, 1982–2006.

\bibitem{Benv}
Y. Benveniste, Effective thermal conductivity of composites with a thermal contact resistance between the constituents: Nondilute case, J. Appl. Phys. 61 (1987), no. 8, 2840–2843.

\bibitem{BenMil}
Y. Benveniste and T. Miloh, Neutral inhomogeneities in conduction phenomena, J. Mech. Phys. Solids 47 (1999), no. 9, 1873–1892.

\bibitem{BoVe}
E. Bonnetier and M. Vogelius, An elliptic regularity result for a composite medium with “touching” fibers of circular cross-section, SIAM J. Math. Anal. 31 (2000), no. 3, 651–677.

\bibitem{BC} B. Budiansky; G.F. Carrier. High shear stresses in stiff fiber composites. J. App. Mech. 51 (1984), 733-735.

\bibitem{dl} H.J. Dong; H.G. Li, Optimal Estimates for the Conductivity Problem by Green's Function Method. Arch. Ration. Mech. Anal. 231 (2019), no. 3, 1427-1453.


\bibitem{dz} H.J. Dong, H. Zhang, On an elliptic equation arising from composite materials. Arch. Ration. Mech. Anal., 222(1), 47-89, 2016.

\bibitem{DY} H.J. Dong, Z.L. Yang, Optimal estimates for transmission problems including relative conductivities with different signs. Adv. Math. 428 (2023), Paper No. 109160, 28 pp.

\bibitem{DLY2} H. Dong, Y.Y. Li, and Z. Yang, Gradient estimates for the insulated conductivity problem: The non-umbilical case, J. Math. Pures Appl. (9) 189 (2024), 103587.

\bibitem{DLY} H. Dong, Y.Y. Li, and Z. Yang, Optimal gradient estimates of solutions to the insulated conductivity problem in dimension greater than two. J. Eur. Math. Soc. (JEMS) 27 (2025), no. 8, 3275–3296.

\bibitem{dx2022} H.J. Dong, L.J. Xu, Higher regularity for solutions to equations arising from composite materials. arXiv:2206.06321. Preprint.

\bibitem{DYZ} H. Dong, Z. Yang, H. Zhu, Gradient estimates for the conductivity problem with imperfect bonding interfaces. arXiv:2409.05652 (2024).

\bibitem {FlaKel}
J. E. Flaherty and J. B. Keller, Elastic behavior of composite media, Comm. Pure Appl. Math. 26 (1973), 565–580.

\bibitem {FJKL}
S. Fukushima, Y.-G. Ji, H. Kang, and X. Li, Finiteness of the stress in presence of closely located inclusions with imperfect bonding, Math. Ann. (2024). DOI 10.1007/s00208-024-02968-9.

\bibitem {GT} D. Gilbarg; N.S. Trudinger, Elliptic partial differential equations of second order, Reprint of the 1998 edition, Classics in Mathematics. Springer, Berlin, 2001.

\bibitem {Hash}
Z. Hashin, Extremum principles for elastic heterogenous media with imperfect interfaces and their application to bounding of effective moduli, J. Mech. Phys. Solids 40 (1992), no. 4, 767-781.

\bibitem{JK} Y-G. Ji, H. Kang, Spectrum of the Neumann-Poincar\'e operator and optimal estimates for transmission problems in presence of two circular inclusions, Int. Math. Res. Not. 2023 (9) (2023) 7638-7685.

\bibitem {Kang}
H. Kang, Quantitative analysis of field concentration in presence of closely located inclusions of high contrast, ICM—International Congress of Mathematicians. Vol. 7. Sections 15–20, [2023] ©2023, pp. 5680–5699.

\bibitem{kly} H. Kang, M. Lim, K. Yun, Asymptotics and computation of the solution of the conductivity equation in the presence of adjacent inclusions with extreme conductivities. J. Math. Pures Appl. (9) 99 (2013), 234-249.

\bibitem{kly2} H. Kang, M. Lim, K. Yun, Characterization of the electric field concentration between two adjacent spherical perfect conductors. SIAM J. Appl. Math. 74 (2014), 125-146.

\bibitem {Keller}
J. B. Keller, Conductivity of a Medium Containing a Dense Array of Perfectly Conducting Spheres or Cylinders or Nonconducting Cylinders, J. Appl. Phys. 34 (1963), no. 4, 991–993.

\bibitem {Keller2}
J. B. Keller, Stresses in Narrow Regions, J. Appl. Mech. 60 (1993), no. 4, 1054–1056.

\bibitem{Kim} Y. Kim, Piecewise regularity results for linear elliptic systems with piecewise regular coefficients. Math. Ann. 391 (2025), no. 1, 613–684.

\bibitem{Kr07}
N. V. Krylov, 
Parabolic and elliptic equations with VMO coefficients. Comm. Partial Differential Equations, 32 (2007), no. 1-3, 453–475.

\bibitem{Li} H.G. Li. Asymptotics for the electric field concentration in the perfect conductivity problem. SIAM J. Math. Anal. 52 (2020), no.4, 3350-3375.

\bibitem{lly} H.G. Li, Y.Y. Li, Z.L. Yang, Asymptotics of the gradient of solutions to the perfect conductivity problem. Multiscale Model. Simul. 17 (2019), no. 3, 899-925.

\bibitem{lwx} H.G. Li, F. Wang, L.J. Xu, Characterization of Electric Fields between two Spherical Perfect Conductors with general radii in 3D. J. Differential Equations 267 (2019), 6644-6690.

\bibitem {LiZhao}
H.G. Li and Y. Zhao, Optimal gradient estimates for the insulated conductivity problem with general convex inclusions case (2024). arXiv:2404.17201.


\bibitem {LiNe}
Y.Y. Li and L. Nirenberg, Estimates for elliptic systems from composite material, Comm. Pure Appl. Math. 56 (2003), no. 7, 892–925.

\bibitem {LiVe}
 Y.Y. Li and M. Vogelius, Gradient estimates for solutions to divergence form elliptic equations with discontinuous coefficients, Arch. Ration. Mech. Anal. 153 (2000), no. 2, 91–151.

\bibitem{LY1} Y.Y. Li, Z.L. Yang, Gradient estimates of solutions to the insulated conductivity problem in dimension greater than two. Math. Ann. 385 (2023), no. 3-4, 1775–1796.

 \bibitem{limyun} M. Lim, K. Yun, Blow-up of electric fields between closely spaced spherical perfect conductors, Comm. Partial Differential Equations 34 (2009), no. 10-12, 1287-1315.



\bibitem{Lieberm} G. M. Lieberman, Oblique derivative problems for elliptic equations. World Scientific Publishing Co. Pte. Ltd., Hackensack, NJ, 2013. xvi+509 pp. ISBN: 978-981-4452-32-8.

\bibitem{LipVer}
R. Lipton and B. Vernescu, Composites with imperfect interface, Proc. R. Soc. Lond. A 452 (1996), no. 1945, 329–358.

\bibitem{M} X. Markenscoff, Stress amplification in vanishingly small geometries. Computational Mechanics 19 (1996), 77-83.


\bibitem{TorRin}
S. Torquato and M. D. Rintoul, Effect of the interface on the properties of composite media, Phys. Rev. Lett. 75 (1995), 4067–4070.

\bibitem{Wen1} B. Weinkove, The insulated conductivity problem, effective gradient estimates and the maximum principle, Math. Ann. 385 (2023), no. 1-2, 1–16.

\bibitem{y1} K. Yun, Estimates for electric fields blown up between closely adjacent conductors with arbitrary shape, SIAM J. Appl. Math. 67 (2007), 714-730.

\end{thebibliography}
\end{document}